\documentclass[onecolumn,
superscriptaddress,
notitlepage,
aps,
jmp,
amsmath, 
amssymb,
showkeys,
longbibliography,
10pt
]{revtex4-1}
\usepackage[utf8]{inputenc}
\usepackage{amsmath,amssymb}
\usepackage{amsthm}
\usepackage{indentfirst}
\usepackage{mathrsfs}
\usepackage{graphicx}
\usepackage{hyperref}
\usepackage{amsthm}
\usepackage{thmtools}
\usepackage{graphicx}
\usepackage{caption}
\usepackage{subcaption}
\usepackage[usenames, dvipsnames]{color}
\usepackage{csquotes}
\usepackage{verbatim}
\usepackage{bookmark}
\usepackage{booktabs}
\usepackage[]{algorithm2e}
\SetKwRepeat{Do}{do}{while}%
\usepackage{url,hyperref}
\usepackage{enumerate}
\usepackage{dsfont}
\usepackage{geometry}

\providecommand{\keywords}[1]{\textbf{\textit{Keywords: }} #1}

\theoremstyle{plain}
\newtheorem{theorem}{Theorem}[section]
\newtheorem{lemma}[theorem]{Lemma}

\newtheorem{prop}[theorem]{Proposition}
\newtheorem{coro}[theorem]{Corollary}

\theoremstyle{definition}
\newtheorem{defn}[theorem]{Definition}
\theoremstyle{remark}
\newtheorem{remark}[theorem]{Remark}
\theoremstyle{example}
\newtheorem{example}[theorem]{Example}

\newcommand{\thmref}[1]{Theorem~\ref{#1}}
\newcommand{\coref}[1]{Corollary~\ref{#1}}
\newcommand{\propref}[1]{Proposition~\ref{#1}}
\newcommand{\lemref}[1]{Lemma~\ref{#1}}
\newcommand{\secref}[1]{Sec.~\ref{#1}}
\newcommand{\figref}[1]{Figure~\ref{#1}}

\newcommand{\ie}{\textit{i.e.}}
\newcommand{\eg}{\textit{e.g.}}

\makeatletter
\newcommand*{\rom}[1]{\expandafter\@slowromancap\romannumeral #1@}
\makeatother

\newcommand{\wt}[1]{\widetilde{#1}}

\hypersetup{
    unicode=false,          
    pdftoolbar=true,        
    pdfmenubar=true,        
    pdffitwindow=false,     
    pdfstartview={FitH},    
    pdftitle={My title},    
    pdfauthor={Author},     
    pdfsubject={Subject},   
    pdfcreator={Creator},   
    pdfproducer={Producer}, 
    pdfkeywords={keyword1} {key2} {key3}, 
    pdfnewwindow=true,      
    colorlinks=true,       
    linkcolor=blue,          
    citecolor=blue,        
    filecolor=magenta,      
    urlcolor=blue           
}

\DeclareMathOperator{\tr}{Tr}

\DeclareMathOperator{\divop}{div}

\newcommand{\ud}{\,\mathrm{d}}

\newcommand{\Real}{\mathbb{R}}
\newcommand{\Complex}{\mathbb{C}}

\newcommand{\vect}[1]{\vec{#1}}

\newcommand{\eps}{\epsilon}

\newcommand{\abs}[1]{\lvert#1\rvert}
\newcommand{\Abs}[1]{\left\lvert#1\right\rvert}

\newcommand{\Norm}[1]{\left\lVert#1\right\rVert}
\newcommand{\eva}{\biggr\rvert}

\newcommand{\Inner}[2]{\left\langle#1, #2\right\rangle}

\def\bigl{\mathopen\big}

\def\bigr{\mathclose\big}

\newcommand{\unit}{\mathbb{I}}
\newcommand{\id}{\mathcal{I}}

\newcommand{\Anticomm}[2]{\left\{#1, #2\right\}}
\newcommand{\Comm}[2]{\left[#1, #2\right]}
\newcommand{\qms}{\mathcal{P}} 

\newcommand{\Bra}[1]{\left\langle#1\right\rvert}

\newcommand{\Ket}[1]{\left\lvert#1\right\rangle}

\newcommand{\hbtcn}{\Complex^n} 
\newcommand{\mset}{\mathfrak{M}}
\newcommand{\msetJ}{\mset^{\cardn}}
\newcommand{\aset}{\mathfrak{A}} 
\newcommand{\psetplus}{\mathfrak{P}_{+}}
\newcommand{\dset}{\mathfrak{D}}
\newcommand{\dsetplus}{\mathfrak{D}_{+}} 

\newcommand{\gradt}{\text{grad}}
\newcommand{\tang}{\mathcal{T}} 

\newcommand{\lbop}{\mathcal{L}}

\newcommand{\renyi}{R{\'e}nyi}
\newcommand{\renyid}{\renyi{} divergence}
\newcommand{\swch}{sandwiched R{\'e}nyi divergence}

\newcommand{\lb}{Lindblad}

\newcommand{\lsi}{log-Sobolev inequality}
\newcommand{\bbren}{Benamou-Brenier}
\newcommand{\fp}{Fokker-Planck}

\newcommand{\svi}{Stroock-Varopoulos inequality}

\newcommand{\relaentropy}[2]{D\left(#1 || #2\right)}
\newcommand{\renyientropy}[1]{H_{\alpha} (\rho)}
\newcommand{\renyidivg}[2]{D_{\alpha}\left(#1 || #2\right)}
\newcommand{\renyidivgalpha}[3]{D_{#3}\left(#1 || #2\right)}
\newcommand{\fdrenyidivg}[2]{\frac{\delta\renyidivg{#1}{#2}}{\delta #1}}
\newcommand{\fdrenyidivgalpha}[3]{\frac{\delta\renyidivgalpha{#1}{#2}{#3}}{\delta #1}}
\newcommand{\relafisheralpha}[3]{\mathscr{I}_{#3}\left(#1 || #2\right)}

\newcommand{\powop}[3]{I_{#1,#2}\left(#3\right)}
\newcommand{\entfun}[2]{\text{Ent}_{#1,\sigma}\left(#2\right)}
\newcommand{\dirichlet}[2]{\mathcal{E}_{#1,\lbop}(#2)}
\newcommand{\lpsp}{\mathbb{L}_{\alpha}}
\newcommand{\cardn}{\mathsf{J}} 
\newcommand{\sigmalpha}{\sigma^{\frac{1-\alpha}{2\alpha}}}
\newcommand{\mop}[2]{\left[#1\right]_{#2}}
\newcommand{\moprenyi}[2]{\left[#1\right]_{\alpha,#2}}
\newcommand{\moprenyialpha}[3]{\left[#1\right]_{#3,#2}}

\newcommand{\diffoptilde}[2]{\mathfrak{D}_{\alpha, #1,#2}}
\newcommand{\RealJ}{\Real^{\cardn}}
\newcommand{\rhodep}{{\alpha,\rho,\lbop}}
\newcommand{\InnerAlphaLB}[2]{\Inner{#1}{#2}_{\rhodep}}

\newcommand{\mani}{\mathcal{M}}
\newcommand{\rie}{Riemannian}

\newcommand{\specmax}[1]{\lambda_{\max}(#1)}
\newcommand{\specmin}[1]{\lambda_{\min}(#1)}

\newcommand{\poin}{Poincar{\'e}}
\newcommand{\lbopgap}{\lambda_{\lbop}}

\newcommand{\waittime}{\tau_{\alpha,\eps}}
\newcommand{\waittimealpha}[1]{\tau_{#1,\eps}}

\newcommand\myeq[1]{\mathrel{\stackrel{\makebox[0pt]{\mbox{\normalfont\tiny #1}}}{=}}}
\newcommand\myge[1]{\mathrel{\stackrel{\makebox[0pt]{\mbox{\normalfont\tiny #1}}}{\ge}}}
\newcommand{\revise}[1]{{#1}}

\begin{document}

	\title{Gradient flow structure and exponential decay of the sandwiched R{\'e}nyi divergence for primitive Lindblad equations with GNS-detailed balance}

	\author{Yu Cao}
	\email{yucao@math.duke.edu}
	\affiliation{Department of Mathematics, Duke University, Box 90320, Durham, NC 27708, USA}

	\author{Jianfeng Lu}
	\email{jianfeng@math.duke.edu}
	\affiliation{Department of Mathematics, Duke University, Box 90320, Durham, NC 27708, USA}
	\affiliation{Department of Physics and Department of Chemistry, Duke University, Box 90320, Durham, NC 27708, USA}

	\author{Yulong Lu}
	\email{yulonglu@math.duke.edu}
	\affiliation{Department
		of Mathematics, Duke University, Box 90320, Durham, NC 27708, USA}

\date{\today}

\begin{abstract}
  We study the entropy production of the sandwiched \renyi{} divergence under the primitive Lindblad equation with GNS-detailed balance.  We
  prove that the Lindblad equation can be identified as the
  gradient flow of the sandwiched \renyi{} divergence
  of any order
  $\alpha\in (0,\infty)$.
  This extends a previous result by Carlen and Maas
  [Journal of Functional Analysis, 273(5), 1810--1869] for the quantum relative entropy (\ie, $\alpha=1$). Moreover, we show that
  the sandwiched \renyi{} divergence of any order $\alpha\in (0,\infty)$
  decays exponentially fast under the time-evolution of such a \lb{}
  equation.
\end{abstract}


\maketitle

\section{Introduction and main result}
The characterization of dynamics for open classical and quantum systems is an essential and fundamental topic. One approach to characterize the dynamics for open systems is entropy production \revise{\cite{spohn_entropy_1978,benatti_entropy_1988,breuer_quantum_2003,fagnola_entropy_2015, muller-hermes_entropy_2016,arnold_entropies_2004}}, from an information theory perspective.
For Markovian dynamics, the entropy production is  positive, which relates to the data processing inequality \cite{lindblad_completely_1975,frank_monotonicity_2013,muller-hermes_monotonicity_2017} and complies with the second law of thermodynamics \cite{cover2012elements,brandao_second_2015}.
The entropy production is also related with the gradient flow structure on well-chosen Riemannian manifolds for those Markov semigroups \revise{\cite{JKO,Otto01_porous,erbar_heat_2010,MAAS20112250, erbar_gradient_2014,carlen_analog_2014,carlen_gradient_2017,carlen_non-commutative_2018}}.
Furthermore, the validity of log-Sobolev inequalities \cite{gross_logarithmic_1975,gross_hypercontractivity_1975,diaconis_logarithmic_1996,otto_generalization_2000,guionnet2003lectures,CIT-064,bakry2014} helps to prove the exponential decay of those entropies under respective Markov semigroups.

In this paper, we focus on the entropy production of
sandwiched \renyi{} divergences, under primitive Lindblad equations with GNS-detailed balance, for
finite dimensional quantum systems.
Let $n$ be the dimension of the quantum system.
The Lindblad equation is an ordinary differential equation (ODE) $\dot{\rho}_t = \lbop^{\dagger}(\rho_t)$,
where $\rho_t$ are $n$-by-$n$ density matrices, and the Lindblad super-operator $\lbop^{\dagger}$ is a linear operator acting on the space of density matrices.
The \lb{} equation with GNS-detailed balance (see \thmref{thm::lindblad_detail_balance} below) has the form
\begin{equation}
\label{eqn::lb}
\dot{\rho}_t = \lbop^{\dagger}(\rho_t) = \sum_{j=1}^{\cardn} e^{-\omega_j/2} \left(\Comm{V_j \rho_t}{V_j^{*}} + \Comm{V_j}{\rho_t V_j^{*}}\right),
\end{equation}
where the positive integer $\cardn \le n^2-1$;
$\omega_j\in \Real$; $V_j$ are $n$-by-$n$ matrices;
constraints for $\omega_j$ and $V_j$ will be given in \thmref{thm::lindblad_detail_balance} below; the commutator of matrices $A$ and $B$ is defined as $\Comm{A}{B}:= AB - BA$.
The notion GNS-detailed balance will be explained in \secref{sec::qms_gns_detailed_balance} below.
The Lindblad equation is called \emph{primitive} (also sometimes referred as ergodic, \eg, see \cite{carlen_gradient_2017,burgarth_ergodic_2013}) if $\text{Ker}(\lbop)$ is spanned by $\unit$, which is the $n$-by-$n$ identity matrix.

Before delving more into Lindblad equations for open quantum systems, we
would like to review some existing results of Fokker-Planck
equations for open classical systems, for the purpose of comparison.  In
fact, almost all the conclusions in this paper are motivated by the classical results for Fokker-Planck equations.
Let us consider, for instance, a $n$-dimensional Brownian particle  under an external potential $\Psi(x)$ with $x\in \Real^n$. Then the time evolution of the probability density function $p_t(x)$ of the Brownian particle is described by the \fp{} equation
\begin{align}
\partial_t p_t(x) = \divop \left(p_t(x) \nabla \Psi(x) + \nabla p_t(x)\right).
\end{align}
It has been shown by Jordan,
Kinderlehrer, and Otto \cite{JKO} that such types of \fp{} equations can be viewed as the Wasserstein gradient flow dynamics of the relative entropy (or the free energy functional). 
Moreover, based on the method of entropy production and the log-Sobolev inequality
\cite{otto_generalization_2000},
one can prove that the solution of the
\fp{} equation converges to the Gibbs (stationary) distribution exponentially fast in the relative entropy.  In our recent
work \cite{prevpaper}, we generalized these results to the classical
\renyid{} of any order $\alpha\in(0,\infty)$: (1) the same \fp{} equation can
be formally identified as the gradient flow dynamics of the classical \renyid{} of any order $\alpha$
with respect to a transportation distance defined by a certain metric tensor;
(2) the solution of the \fp{} equation converges exponentially fast to the
Gibbs distribution in the classical \renyid{}. Apart from
Fokker-Planck equations, several other dissipative dynamics can also be
viewed as the gradient flow dynamics of various transport metrics, for instance,
the heat equation \cite{erbar_heat_2010}, finite Markov Chains
\cite{MAAS20112250} and porous medium equations
\cite{Otto01_porous, erbar_gradient_2014}.

As for open quantum systems, \lb{} equations \cite{lindblad_generators_1976, gorini_completely_1976}, generally regarded as the quantum analog of \fp{} equations, could be derived from the weak-coupling limit of the system and baths \cite{davies_markovian_1974, breuer_book}.
Due to the similar role that \lb{} and \fp{} equations play for open quantum and classical systems respectively, it is interesting to study \lb{} equations from the perspective of
gradient flows.
In a previous work by Carlen and Maas, they showed that fermionic Fokker-Planck equation can be identified as the gradient flow dynamics of the quantum relative entropy \cite{carlen_analog_2014};
in their more recent work,
they extended results from \cite{carlen_analog_2014} and showed that for finite dimensional quantum systems,
primitive \lb{} equations with GNS-detailed balance can indeed
be identified as the gradient flow dynamics of the quantum relative entropy
\cite{carlen_gradient_2017};
\revise{at almost the same time when the first preprint of this work appeared, they also established a unified picture of gradient flows including the previous two cases and Markov Chains in the setting of $\text{C}^*$-algebras \cite{carlen_non-commutative_2018}.}
The generalization of their results
from finite dimensional matrices
to infinite dimensional noncommutative algebras has been recently
considered in \cite{wirth_noncommutative_2018}.
In the present paper, we focus on the finite dimensional setting and
aim to characterize the gradient flow structure
of Lindblad equations in terms of a family of quantum divergences --- \swch{}s  \cite{muller-lennert_quantum_2013,wilde_strong_2014}, including
the quantum relative entropy as a specific instance. Moreover, we are also interested in quantifying the decaying property of the solution $\rho_t$  of the Lindblad equation \eqref{eqn::lb} in terms of the \swch{}.

The sandwiched \renyi{} divergence is defined as follows.

\begin{defn}[Sandwiched \renyi{} divergence \cite{muller-lennert_quantum_2013}]
	\label{def::sandwiched_renyi}
	Given two positive semi-definite matrices $\rho,\sigma$ such that $\rho \ll \sigma$ and $\rho\neq 0$, the \swch{}  $\renyidivg{\rho}{\sigma}$ of order $\alpha \in (0, \infty)$ is defined by
	\begin{align}
	\label{eqn::swch}
	\renyidivg{\rho}{\sigma} := \left\{
	\begin{aligned}
	\frac{1}{\alpha-1} \log\left(\  \tr \left[ \bigl(\sigma^{\frac{1-\alpha}{2\alpha}} \rho \sigma^{\frac{1-\alpha}{2\alpha}} \bigr)^{\alpha} \right]\ \right), & \qquad \alpha\in (0,1)\cup (1,\infty);\\
	\tr\left(\rho \log(\rho) - \rho \log(\sigma)\right), & \qquad \alpha = 1. \\
	\end{aligned}
	\right.
    	\end{align}
\end{defn}
The notation $\rho \ll \sigma$ means the kernel of $\sigma$ is a subspace of the kernel of $\rho$.
The term $\tr(\rho)$ in \cite[Definition 2]{muller-lennert_quantum_2013} is simply set as one, since we only consider $\rho$ to be a density matrix in this paper.
$\sigma$ is assumed to be a full-rank density matrix, so $\rho \ll \sigma$ is always satisfied for this paper.
For any fixed $\rho$ and full-rank $\sigma$, the sandwiched \renyi{} divergence is continuous with respect to the order $\alpha$, in particular, $\renyidivgalpha{\rho}{\sigma}{} = \lim_{\alpha\rightarrow 1} \renyidivgalpha{\rho}{\sigma}{\alpha}$.

The sandwiched \renyi{} divergence is a quantum analog of the classical \renyi{} divergence, and has recently received much attention due to its use in quantum information theory; see \eg, \cite{wilde_strong_2014,mosonyi_quantum_2015}.
The sandwiched
\renyi{} divergence unifies many entropy measures: when $\alpha=1$, it reduces to the quantum relative entropy; when
$\alpha \rightarrow \infty$, it converges to the quantum relative max-entropy;
when $\alpha=1/2$ and $\alpha=2$, it is closely connected to the quantum fidelity
and $\chi^2$ divergence respectively \cite{temme_2-divergence_2010}.
The sandwiched
\renyi{} divergence satisfies the data processing inequality for any order $\alpha \ge \frac{1}{2}$ and for any completely positive trace-preserving (CPTP) map \cite{frank_monotonicity_2013,muller-lennert_quantum_2013,beigi_sandwiched_2013}. For fixed density matrices $\rho$ and $\sigma$, it is monotonically increasing with respect to the order $\alpha\in (0,\infty)$ \cite[Theorem 7]{muller-lennert_quantum_2013}.
The sandwiched \renyi{} divergence has also been used to study the quantum second laws of thermodynamics \cite{brandao_second_2015}, which is a major motivation of our work. Its global convergence rate and mixing time under quantum Markov semigroups has been studied in \cite{muller-hermes_sandwiched_2018}.

We remark that there is another well-studied  family of quantum \renyi{} divergences  known as \emph{Petz-\renyi{}} divergence \cite{petz_quasi-entropies_1986}, which is defined by $\wt{D}_{\alpha}(\rho||\sigma) := \frac{1}{\alpha-1} \log\left(\tr\left(\rho^{\alpha}\sigma^{1-\alpha}\right)\right)$.
When $\rho$ and $\sigma$ commute,
the sandwiched \renyi{} divergence is identical to the Petz-\renyi{} divergence.
We also remark that \cite{audenaert_alpha-z-renyi_2015} has studied more generalized \emph{$(\alpha,z)$-quantum \renyi{} divergences},
which include sandwiched \renyi{} divergences and Petz-\renyi{} divergences as special instances;
the discussion on the data processing inequality of $(\alpha,z)$-quantum \renyi{} divergences can be found in a recent review paper \cite{carlen_inequalities_2018} and the references therein.

The first main result of this paper is the following theorem.
\begin{theorem}
	\label{thm::renyi_gradient}
	Consider a primitive \lb{} equation with GNS-detailed balance \eqref{eqn::lb} and suppose its unique stationary state $\sigma$ is full-rank.
	For any order $\alpha\in(0,\infty)$, consider the energy functional $\renyidivg{\rho}{\sigma}$ with respect to $\rho$.
	In the space of strictly positive density matrices
	$\dsetplus$ equipped with the metric tensor
	defined in the Definition~\ref{def::metric_tensor_renyi},
	the gradient flow dynamics of the sandwiched \renyi{} divergence $\renyidivg{\rho}{\sigma}$
	is exactly the \lb{} equation \eqref{eqn::lb}.
\end{theorem}

This theorem immediately implies the monotonicity of the sandwiched \renyi{} divergence under the evolution of such \lb{} equations, summarized in the following \coref{coro::monotonicity_renyi}. For $\alpha \ge \frac{1}{2}$, this corollary can be proved simply by applying the data processing inequality \cite{frank_monotonicity_2013}.  In our case, the monotonicity also holds for $\alpha\in (0,\frac{1}{2})$ for primitive \lb{} equations with GNS-detailed balance, \revise{though the data processing inequality does not generally hold for sandwiched \renyi{} divergences when $\alpha<\frac{1}{2}$ \cite[Sec. 4]{berta_variational_2017}}.
\begin{coro}
	\label{coro::monotonicity_renyi}
		Under the same conditions as in \thmref{thm::renyi_gradient}, suppose $\rho_{\cdot}: [0,\infty) \rightarrow \dsetplus$ is the solution of that Lindblad equation.
For any $\alpha\in(0,\infty)$,  the sandwiched \renyi{} divergence $\renyidivg{\rho_t}{\sigma}$, is non-increasing with respect to the time $t$, i.e. $\partial_t \renyidivg{\rho_t}{\sigma} \leq 0$.
\end{coro}

\revise{\thmref{thm::renyi_gradient} only provides a sufficient detailed balance condition for Lindblad equations to be the gradient flow dynamics of sandwiched \renyi{} divergences. Later in \secref{sec::necessary_condition}, we will discuss in \thmref{thm::lb_necessary} the necessary condition for Lindblad equations to be possibly expressed as the gradient flow dynamics of sandwiched \renyi{} divergences, by adapting and extending the analysis in \cite[Theorem 2.9]{carlen_non-commutative_2018}. Loosely speaking, the family of Lindblad equations that can be possibly written as the gradient flow dynamics of sandwiched \renyi{} divergences of any order $\alpha\in (0,\infty)$ is not substantially larger than the class of Lindblad equations with GNS-detailed balance, which explains the reason why we restrict our attention to Lindblad equations with GNS-detailed balance in \thmref{thm::renyi_gradient}. However, identifying a sufficient and necessary condition still remains open.}

Next, we would like to study the exponential decay of sandwiched \renyi{} divergences.
For \fp{} equations in classical systems, the exponential decay of the relative entropy could be proved via the (classical) \lsi{}, namely, the relative entropy is bounded above by the relative Fisher information.
This approach has been generalized in our previous
work \cite{prevpaper} to obtain the exponential decay of
the classical \renyi{} divergence by introducing the \emph{relative $\alpha$-Fisher information}.
For quantum systems, in the same flavor, let us first define the \emph{quantum relative $\alpha$-Fisher information}
as follows.
\begin{defn}
Suppose $\lbop^{\dagger}$ is the generator of a primitive \lb{} equation with GNS-detailed balance and the unique stationary state $\sigma$ is full-rank. The \emph{quantum relative $\alpha$-Fisher information}
between $\rho\in\dsetplus$ and $\sigma$ is defined as
\begin{equation}
\label{eqn::q_relative_fisher}
\relafisheralpha{\rho}{\sigma}{\alpha} :=
\begin{cases} -\Inner{\fdrenyidivg{\rho}{\sigma}}{\lbop^{\dagger}(\rho)}, & \text{ if } \alpha \neq 1;\\
-\Inner{\log(\rho)-\log(\sigma)}{\lbop^{\dagger}(\rho)}, & \text{ if } \alpha = 1.
\end{cases}
\end{equation}
\end{defn}
When $\alpha = 1$, we denote $\relafisheralpha{\rho}{\sigma}{} \equiv \relafisheralpha{\rho}{\sigma}{1}$ and call it quantum relative Fisher information for simplicity.
The inner product $\Inner{\cdot}{\cdot}$ herein is the Hilbert-Schmidt inner product, defined by
$\Inner{A}{B} := \text{Tr}(A^\ast B)$ for all matrices $A, B$.
Notice that in the definition above, the quantum relative $\alpha$-Fisher information
depends on the Lindblad super-operator $\lbop^{\dagger}$, which is different from the
classical setting where the relative $\alpha$-Fisher information only depends on the Gibbs stationary distribution (see \cite{prevpaper}).  Also, observe that if $\rho_t$ solves the Lindblad equation \eqref{eqn::lb}, then $\relafisheralpha{\rho_t}{\sigma}{\alpha} = -\partial_t \renyidivg{\rho_t}{\sigma}$.
Since by \coref{coro::monotonicity_renyi} $ \partial_t \renyidivg{\rho_t}{\sigma}\leq 0$, the quantum relative $\alpha$-Fisher information $\relafisheralpha{\rho_t}{\sigma}{\alpha} $ is non-negative.

\begin{defn}
Suppose the density matrix $\sigma$ is the unique stationary state of the Lindblad equation \eqref{eqn::lb}. For a fixed order $\alpha\in (0,\infty)$, $\sigma$ is said to satisfy the \emph{quantum $\alpha$-log Sobolev inequality} (or quantum $\alpha$-LSI in short) if there exists $K_{\alpha} > 0$, such that
\begin{equation}
\label{eqn::lsi_alpha}
\renyidivgalpha{\rho}{\sigma}{\alpha}\le \frac{1}{2K_{\alpha}}\relafisheralpha{\rho}{\sigma}{\alpha},\qquad \forall \rho\in \dset,
\end{equation}
where $\dset$ is the space of density matrices. $K_{\alpha}$ is called the \emph{$\alpha$-log Sobolev constant}. When $\alpha=1$, we call \eqref{eqn::lsi_alpha} quantum log-Sobolev inequality and denote $K \equiv K_1$ for simplicity.
\end{defn}
The $\alpha$-log Sobolev constant $K_{\alpha}$ has been considered in \cite[Definition 3.1]{muller-hermes_sandwiched_2018}. When $\alpha=1$, the above form reduces to the quantum log-Sobolev inequality in \cite[Eq. (8.17)]{carlen_gradient_2017}.

\begin{remark}
Another definition of quantum $\alpha$-LSI was given
previously under the framework of noncommutative $\lpsp$ spaces (see \eg, \cite[Definition 3.5]{olkiewicz_hypercontractivity_1999}, \cite[Definition 11]{kastoryano_quantum_2013} and \cite[Sec. 2.5]{beigi_quantum_2018}). We will briefly revisit this concept in \secref{sec::noncommutative_Lp} and we shall denote the log-Sobolev constant in this setting up by $\kappa_{\alpha}$.
The essential difference is that for the definition of $\kappa_{\alpha}$, the entropy function uses the quantum relative entropy for a density matrix raising to the power of $\alpha$, instead of using sandwiched \renyi{} divergence (see \eqref{eqn::kappa_alpha_v2} for more details).
\end{remark}

For \lb{} equations with GNS-detailed balance, the primitivity of \lb{} equations is equivalent to the validity of the quantum LSI, summarized in the following \propref{prop::primitive_lsi}. It is important to note that by \lemref{lem::lb_quadratic} below, if the \lb{} equation satisfies GNS-detailed balance, then $-\lbop$ is a positive semi-definite operator, with respect to the inner product $\Inner{\cdot}{\cdot}_{1/2}$ which is defined by $\Inner{A}{B}_{1/2} := \tr\left(A^* \sigma^{1/2} B \sigma^{1/2}\right)$ for all matrices $A, B$. As a consequence,  primitivity of $\mathcal{L}$ implies that $-\mathcal{L}$ has a positive spectral gap, denoted by $\lbopgap$. The proposition below provides two-side estimates of $K$ in terms of $\lbopgap$.

\begin{prop}
	\label{prop::primitive_lsi}
	Consider the \lb{} equation with GNS-detailed balance \eqref{eqn::lb} and suppose a full-rank density matrix $\sigma$ is its stationary state.
\revise{Then,
$\sigma$ satisfies the quantum LSI iff the \lb{} equation is primitive.
More quantitatively, if the \lb{} equation is primitive, we have
\[ \frac{1}{1-\log(\sqrt{\specmin{\sigma}})} \lbopgap \le K \le \lbopgap,\]
where $\specmin{\sigma}$ is the smallest eigenvalue of $\sigma$.
}
\end{prop}
\revise{Proposition \ref{prop::primitive_lsi} follows mainly from \cite[Theorem 5.3]{muller-hermes_sandwiched_2018}, \cite[Theorem 16]{kastoryano_quantum_2013} (or see \cite[Theorem 4.1]{muller-hermes_sandwiched_2018}), and the quantum \svi{} \cite[Corollary 16]{beigi_quantum_2018}.
The proof is postponed to \secref{sec::proof_qlsi}.
We remark that the lower bound above is not sharp and it remains an interesting question to find a tighter lower bound.}

It is not hard to see that if the  quantum $\alpha$-LSI \eqref{eqn::lsi_alpha} is valid,  then it follows immediately  from Gr\"onwall inequality that  (see also \cite[Theorem 3.2]{muller-hermes_sandwiched_2018})
\begin{equation}
\label{eqn::renyi_exp_decay}
\renyidivgalpha{\rho_t}{\sigma}{\alpha}\le \renyidivgalpha{\rho_0}{\sigma}{\alpha} e^{-2K_{\alpha}t},\qquad \forall t \ge 0.
\end{equation}

The validity of \eqref{eqn::lsi_alpha}  with a positive $K_\alpha$ is in general difficult to establish. For depolarizing semigroup, $K_2$ has been computed in \cite[Theorem 3.3]{muller-hermes_sandwiched_2018} and $K$ has been studied in \cite{muller-hermes_relative_2016}.
For the depolarizing semigroup with stationary state $\sigma$ being the maximally mixed state, both upper bound and lower bound of $K_{\alpha}$ with $\alpha \ge 2$ have been provided in \cite[Theorem 3.4]{muller-hermes_sandwiched_2018}.
Moreover, it has been proved
in  \cite[Theorem 4.1]{muller-hermes_sandwiched_2018} that the $\alpha$-log Sobolev constant cannot exceed the spectral gap, namely for any $\alpha \ge 1$, $K_{\alpha} \le \lbopgap$. In the present paper, instead of pursuing the  global decay of the sandwiched \renyi{} divergence $\renyidivgalpha{\rho_t}{\sigma}{\alpha}$ as in \eqref{eqn::renyi_exp_decay}, we focus on the exponential decay of  $\renyidivgalpha{\rho_t}{\sigma}{\alpha}$ in the large time regime. More specifically, our second main result (Theorem \ref{thm::decay_quantum_renyi}) shows that the  sandwiched \renyi{} divergence $\renyidivgalpha{\rho_t}{\sigma}{\alpha}$ decays exponentially with the sharp rate $2\lambda_\mathcal{L}$ when $t$ is large.

\begin{theorem}
	\label{thm::decay_quantum_renyi}
Under the same conditions as in \thmref{thm::renyi_gradient},
for any given solution $\rho_{\cdot}: [0,\infty)\rightarrow \dsetplus$ of the \lb{} equation and any $\eps\in \left(0,\frac{\specmin{\sigma}^2}{2}\right)$, there exist $C_{\alpha,\eps}$ and $\waittime$ such that
\begin{equation}
\label{eqn::decay_quantum_renyi}
	\renyidivg{\rho_t}{\sigma}\le C_{\alpha,\eps} \renyidivg{\rho_0}{\sigma} e^{-2\lbopgap t},\qquad \forall t\ge \waittime,
\end{equation}
where
\begin{equation*}
C_{\alpha,\eps} = \frac{e^{\renyidivgalpha{\rho_0}{\sigma}{2}} - 1}{\renyidivg{\rho_0}{\sigma}} \exp\left( \Theta(\alpha-2) 2\lbopgap  T_{\alpha,\eps}\right),
\end{equation*}
and
\begin{equation*}
\waittime = \left\{
\begin{split}
& 0,  \qquad & \alpha \le 2; \\
& T_{\alpha,\eps} + \max\left(0, \frac{1}{2K}  \log(\relaentropy{\rho_0}{\sigma}/\eps)\right),  \qquad & \alpha > 2. \\
\end{split}
\right.
\end{equation*}
In the above, $\specmin{\sigma}$ is the smallest eigenvalue of $\sigma$; $T_{\alpha,\eps}=\frac{1}{2K \eta} \log(\alpha-1)$ is a function depending on $\alpha$ and $\eps$;
$\Theta$ is the Heaviside function;
	\begin{align*}
	\eta &= \min\left(\frac{1}{2}, \min_{j=1}^{\cardn} \left\{\frac{2\sqrt{e^{\omega_j} \frac{1}{\Lambda}}}{1+e^{\omega_j} \Lambda}\right\}\right),\\
	\Lambda &= \frac{\specmax{\sigma}}{\specmin{\sigma}} \exp\left(2 \sqrt{2\eps}\frac{2\specmin{\sigma} - \sqrt{2\eps}}{\specmin{\sigma}(\specmin{\sigma} - \sqrt{2\eps})}\right),
	\end{align*}
where $\omega_j$ are \revise{Bohr frequencies determined by the stationary state $\sigma$} (see \secref{sec::qms_gns_detailed_balance}) \revise{and $\specmax{\sigma}$ is the largest eigenvalue of $\sigma$.}
\end{theorem}

The proof of Theorem \ref{thm::decay_quantum_renyi} can be found in \secref{sec::exp_decay}.
Here we sketch the strategies used in the proof. First, we prove this theorem for the case $\alpha=2$ by using a uniform lower bound of $\relafisheralpha{\rho}{\sigma}{2}$ (see \eqref{eqn::fisher_bound_2}).
Next, we prove that the quantum LSI \eqref{eqn::lsi_alpha} holds (see \propref{prop::primitive_lsi}).
Based on the quantum LSI (see \eqref{eqn::lsi_alpha} for $\alpha =1$), we  prove a quantum comparison theorem  (see \propref{prop::comparison}), which implies the
exponential decay of the sandwiched \renyi{} divergence of any order $\alpha\in(0,\infty)$ (see \propref{prop::exp_decay_renyi}).

Some remarks of this theorem are in order:
\begin{remark}
\normalfont
\begin{enumerate}[1)]

\item The upper bounds for $\waittime$ and $C_{\alpha,\eps}$ may not be sharp; a better bound might be important if one wants to study the mixing time for quantum decoherence, which however will not be pursued here.
For fixed $\alpha > 2$, when $\eps$ becomes smaller, $T_{\alpha,\eps}$ decreases and we have a smaller prefactor $C_{\alpha,\eps}$, at the expense of waiting longer time $\waittime$ since the term $\relaentropy{\rho_0}{\sigma}/\eps$ blows up as $\eps\downarrow 0^{+}$.

\item The artificial $\eps$ and the complicated expression of
$T_{\alpha,\eps}$ come from \propref{prop::comparison}.
There is a different way to quantify $\eta$ under the assumption of the quantum $2$-LSI in the sense of noncommutative $\lpsp$ spaces; please see \secref{sec::noncommutative_Lp} for background and the discussion in \secref{sec::diss_q_comparison} for more details.
We prefer this version because the \lsi{} for $\alpha=1$ is also used for the classical result in \cite[Theorem 1.2]{prevpaper}. One could remove the dependence of the parameter $\eps$ by simply choosing, \eg, $\eps = \specmin{\sigma}^2/8$ in this theorem, then
\[\Lambda = \frac{\specmax{\sigma}}{\specmin{\sigma}} e^{3}, \qquad \eta = \min\left(\frac{1}{2}, \min_{j=1}^{\cardn} \left\{\frac{2e^{-3/2}\sqrt{e^{\omega_j} \frac{\specmin{\sigma}}{\specmax{\sigma}}}}{1+e^{\omega_j} \frac{\specmax{\sigma}}{\specmin{\sigma}} e^{3}}\right\}\right).\]
If we further assume that $\sigma$ is the maximally mixed state, \ie, $\sigma = \frac{1}{n}\unit$, then
\[\Lambda = e^3, \qquad \eta = \frac{2e^{-3/2}}{1+e^3}\]
are independent of the dimension $n$.
\end{enumerate}
\end{remark}

{\bf Our contribution.}
We summarize here contributions of the current work.
\begin{enumerate}[1)]
	\item We extend the work of Carlen and Maas in \cite{carlen_gradient_2017} by proving that primitive \lb{} equations with GNS-detailed balance can be identified as the gradient flow dynamics of the sandwiched \renyi{} divergence of any order $\alpha\in(0,\infty)$; see \thmref{thm::renyi_gradient}.
	\revise{In addition, we follow a recent work of Carlen and Maas \cite{carlen_non-commutative_2018} to study the necessary condition for Lindblad equations to be possibly written as the gradient flow dynamics of sandwiched \renyi{} divergences; see \thmref{thm::lb_necessary}.}

	\item We prove that the sandwiched \renyi{} divergence decays exponentially fast under the evolution of primitive \lb{} equations with GNS-detailed balance; see \thmref{thm::decay_quantum_renyi}.
	Our result \eqref{eqn::decay_quantum_renyi} shows that the sandwiched \renyi{} divergence decays with an asymptotic rate $2\lbopgap$, which is sharp according to \cite[Theorem 4.1]{muller-hermes_sandwiched_2018}.

	To prove \thmref{thm::decay_quantum_renyi}, we first prove a quantum comparison theorem (see \propref{prop::comparison}), which is of interest in its own right.
	It is {essentially} a hypercontractivity-type estimation of the solution of the \lb{} equation, enabling us to bound $D_{\alpha_1}(\rho_T || \sigma)$
	by $\renyidivgalpha{\rho_0}{\sigma}{\alpha_0}$ with some $T > 0$ when $\alpha_1 > \alpha_0$. As a direct consequence of the comparison theorem,  \propref{prop::exp_decay_renyi}
	shows that the sandwiched \renyi{} divergence exponentially decays for all orders $\alpha\in (0,\infty)$.
	Such a comparison result for classical \renyi{} divergences was proved for the Fokker-Planck equation in \cite[Theorem 1.2]{prevpaper}.
	For quantum dynamics, a similar comparison result was also obtained in \cite[Corollary 17]{beigi_quantum_2018}. However,
	a major difference between our result and theirs is that we use $K$ in the estimation of the waiting time, while they use $\kappa_2$ (see \eqref{eqn::lsic_Lp}). \revise{See a detailed discussion in \secref{sec::diss_q_comparison}.}
\end{enumerate}

The rest of the paper is organized as follows.
In \secref{sec::preliminaries}, we will recall several preliminary results on, \eg,
gradient flows, \lb{} equations with detailed balance,
and noncommutative $\mathbb{L}_\alpha$ spaces.
\secref{sec::proof} will be devoted to the proof of \thmref{thm::renyi_gradient} and \coref{coro::monotonicity_renyi}.
\revise{In \secref{sec::necessary_condition}, we will study the necessary detailed balance condition for Lindblad equations to be the gradient flow dynamics of the sandwiched \renyi{} divergence of order $\alpha\in (0,\infty)$; see \thmref{thm::lb_necessary}.}
\propref{prop::primitive_lsi} and \thmref{thm::decay_quantum_renyi} will be proved in \secref{sec::exp_decay}. Finally,  we will discuss in \secref{sec::outlook} some potential directions in the future, in particular, the generalization of quantum Wasserstein distance involving the sandwiched \renyi{} divergence.

\section{Preliminaries}
\label{sec::preliminaries}
In this section, we recall some technical notions: gradient flow
dynamics on the \rie{} manifold, \lb{} equations with GNS-detailed balance,
quantum analog of gradient and divergence operators, chain rule identity, and quantum noncommutative $\lpsp$ spaces.
Most results from \secref{sec::qms_gns_detailed_balance} to \secref{sec::chain_rule} follow from \cite{carlen_gradient_2017}; \secref{sec::noncommutative_Lp} is mostly based on \cite{beigi_quantum_2018}.

{\bf Notations:}
Throughout this section and the rest of this paper, we will use the following notations.
Denote $\mset$ as the space of all $n\times n$ complex-valued
matrices.  The space of Hermitian matrices is $\aset$;
the space of strictly positive matrices is $\psetplus$; the space of
traceless Hermitian matrices is $\aset_0$. The space of density
matrices is $\dset$ and the space of invertible density matrices is
$\dsetplus$.  Unity $\unit$ is the $n$-by-$n$ identity matrix; $\id$
is the identity super-operator on $\mset$.

Hilbert-Schmidt inner product is
$\Inner{A}{B} := \tr\left(A^* B\right)$ for $A, B\in \mset$.
We extend the inner product $\Inner{\cdot}{\cdot}$ to vector fields over $\mset$ by
$\bigl\langle \vect{A}, \vect{B} \bigr\rangle:= \sum_{j=1}^{\cardn} \Inner{A_j}{B_j}$, where $\vect{A}=\begin{pmatrix}A_1 & A_2 & \cdots A_\cardn\end{pmatrix}$,
$\vect{B}=\begin{pmatrix}B_1 & B_2 & \cdots B_\cardn\end{pmatrix}$
and $A_j, B_j\in \mset$ for all $1\le j \le \cardn$; here $\cardn$ is a fixed positive integer. The space of such vector fields is denoted by $\msetJ$.
Such a convention is also applied analogously to all inner products on $\mset$.

Denote $\specmin{\rho}$ as the smallest eigenvalue of a density matrix $\rho$ and $\specmax{\rho}$ as the largest one. Asterisk $*$ in the superscript means Hermitian-conjugate; dagger $\dagger$ means the adjoint operator with respect to the Hilbert-Schmidt inner product.

To simplify the notation, a \emph{weight operator} $\Gamma_{\sigma}: \mset\rightarrow \mset$ is defined by
\begin{equation}
\label{eqn::gamma_sigma}
\Gamma_{\sigma} (A) := \sigma^{\frac{1}{2}} A \sigma^{\frac{1}{2}},
\end{equation}
for any matrix $A\in\mset$. Hence, for any $\gamma\in \Real$,
\begin{align*}
\Gamma_{\sigma}^{\gamma} (A) = \sigma^{\gamma/2} A \sigma^{\gamma/2}.
\end{align*}

\subsection{Gradient flow dynamics}
\label{sec::preliminary_gradient}

Consider a \rie{} manifold $(\mani, g_{\rho})$: $\mani$ is a smooth manifold and $g_{\rho}(\cdot, \cdot)$ is an inner product on the tangent space $\tang_{\rho}\mani$ for any $\rho\in \mani$.
Given any (energy) functional $E$ on $\mani$,
the gradient, denoted by $\gradt{}E \vert_{\rho}$ at any state $\rho\in \mani$, is defined as the element in the tangent space $\tang_{\rho}\mani$ such that
\begin{equation}
\label{eqn::riemannian_grad}
g_{\rho}\left(\gradt{} E\eva_{\rho},  \nu\right) = \frac{\ud}{\ud \eps} E(\rho + \eps \nu) \eva_{\eps=0},\qquad \forall \nu \in \tang_{\rho}\mani.
\end{equation}
Then, the gradient flow dynamics refers to following autonomous differential equation
\begin{equation}
\label{eqn::riemannian_gradient_flow}
\frac{\ud}{\ud t}{\rho}_t = - \gradt{} E\eva_{\rho_t}.
\end{equation}
This gradient flow dynamics is essentially a generalization of
steepest descend dynamics in the Euclidean space.  In this paper, we will take $\mani$ to be the space of strictly positive density matrices $\dsetplus$. The choice of $g_{\rho}$ is more subtle and technical; thus we will explain it below in \secref{sec::proof}.

\subsection{Quantum Markov semigroup with GNS-detailed balance}
\label{sec::qms_gns_detailed_balance}
For a continuous time-parameterized semigroup $(\qms_t)_{t\ge 0}$, acting on the space of linear operators on a finite dimensional Hilbert space $\hbtcn$, it is called a \emph{quantum Markov semigroup} (QMS) if $\qms_t$ is completely positive and $\qms_t (\unit) = \unit$, for any $t\ge 0$. The generator of the QMS is denoted by $\lbop$, \ie, $\qms_t = e^{t\lbop}$.
Physically, this semigroup $\qms_t$ usually refers to the time evolution super-operator for observables in the Heisenberg picture.

To explain quantum detailed balance,
let us define inner products $\Inner{\cdot}{\cdot}_s$ parameterized by $s\in [0,1]$, for a given full-rank density matrix $\sigma$, in the following way: for any $A, B\in \mset$,
\begin{equation}
\label{eqn::inner_s}
\Inner{A}{B}_s := \tr\left(\sigma^s A^* \sigma^{1-s} B\right).
\end{equation}
Essentially, $\Inner{\cdot}{\cdot}_s$
is a $\sigma$-weighted Hilbert-Schmidt inner product.
The QMS is said to satisfy \emph{GNS-detailed balance}
if
$\qms_t$ is self-adjoint with respect to the inner product $\Inner{\cdot}{\cdot}_1$ for all $t\ge 0$, that is to say, $\Inner{\qms_t (A)}{B}_1 = \Inner{A}{\qms_t (B)}_1$ for all $A, B\in \mset$ and $t\ge 0$.
In some literature, the case $s = 1/2$ is considered instead and the
quantum Markov semigroup is said to satisfy \emph{KMS-detailed balance} if
$\qms_t$ is self-adjoint with respect to the inner product
$\Inner{\cdot}{\cdot}_{1/2}$. It is not hard to verify that the QMS satisfies KMS-detailed balance iff
\begin{equation}
\label{eqn::lbop_lbop_dag}
\Gamma_{\sigma}^{-1} \circ \lbop^{\dagger}\circ \Gamma_{\sigma} = \lbop.
\end{equation}
An important relation is that
GNS-detailed balance is a more restrictive situation and it implies
KMS-detailed balance; reversely, it is not true.  More discussion on
the relation between GNS and KMS detailed balance could be found in
\cite{carlen_gradient_2017}, in particular, Theorem 2.9 therein.
Another important consequence is that if $\qms_t$ satisfies detailed
balance (either KMS or GNS), then $\sigma$ is invariant under
dynamical evolution $\qms_t^{\dagger}$, that is to say,
$\qms_t^{\dagger} (\sigma) = \sigma$. We add the prefix GNS and KMS to explicitly distinguish these two conditions. \revise{More results on the quantum detailed balance can be found in \eg{} \cite{alicki_detailed_1976,fagnola_generators_2007,fagnola_generators_2010,temme_2-divergence_2010} and the references therein.}

The general form of quantum Markov semigroups with GNS-detailed balance
is given by the theorem below.
\begin{theorem}[{\cite[Theorem 3.1]{carlen_gradient_2017}}]
\label{thm::lindblad_detail_balance}
	Suppose that the QMS $\qms_t$ satisfies GNS-detailed balance, then
	\begin{equation}
	\label{eqn::heisenberg_lb}
	\begin{split}
	\lbop (A) &= \sum_{j=1}^{\cardn}  e^{-\omega_j/2} V_j^{*} \Comm{A}{V_j} + e^{\omega_j/2} \Comm{V_j}{A} V_j^{*}  \\
	&= \sum_{j=1}^{\cardn} e^{-\omega_j/2} \left( V_j^{*} \Comm{A}{V_j} + \Comm{V_j^*}{A}V_j\right),
	\end{split}
	\end{equation}
	where $\omega_j\in \Real$ and $\cardn \le n^2-1$. Moreover
    $\left\{V_j\right\}_{j=1}^{\cardn}$ satisfy the following:
	\begin{enumerate}
		\item $\tr(V_j) = 0$ and $\Inner{V_j}{V_k} = \delta_{j,k} c_j$ for all $1\le j, k \le \cardn$; constants $c_j > 0$ (for all $1\le j \le \cardn$) come from normalization;
		\item for each $j$, there exists $1\le j'\le \cardn$ such that $V_j^{*} = V_{j'}$;
		\item $\Delta_{\sigma} (V_j)  = e^{-\omega_j} V_j$, where the modular operator $\Delta_{\sigma}: \mset\rightarrow\mset$ is defined as
	\begin{align*}
	\Delta_{\sigma} (A) := \sigma A \sigma^{-1},\qquad \text{ for } A\in \mset;
	\end{align*}
		\item if $V_j^* = V_{j'}$ then $c_j  = c_{j'}$ and $\omega_j = -\omega_{j'}$.
	\end{enumerate}
\end{theorem}

Then it could be easily derived that (see also Eq. (3.7) in \cite{carlen_gradient_2017})
\begin{equation}
\label{eqn::sch_lb}
\begin{split}
\lbop^{\dagger} (A) &= \sum_{j=1}^{\cardn} \left(e^{-\omega_j/2} \Comm{V_j A}{V_j^*} + e^{\omega_j/2} \Comm{V_j^*}{A V_j} \right) \\
&= \sum_{j=1}^{\cardn} e^{-\omega_j/2} \left(\Comm{V_j A}{V_j^{*}} + \Comm{V_j}{A V_j^{*}}\right),
\end{split}
\end{equation}
which is well known as the \lb{} super-operator.

\begin{example}
Consider the depolarizing semigroup with the generator
$\lbop^{\dagger}_{\text{depol}}(A) = \gamma (\sigma\tr(A) - A)$, and $\gamma > 0$. It is not hard to verify that $\lbop_{\text{depol}}(A) = \gamma\left(\tr(\sigma A) \unit - A\right)$ and that $\lbop_{\text{depol}}$ is self-adjoint with respect to the inner product $\Inner{\cdot}{\cdot}_{1}$. Thus this depolarizing semigroup satisfies the GNS-detailed balance.
\end{example}

\revise{The following lemma is a special case of \cite[Lemma 2.5]{carlen_gradient_2017} and is enough for our purpose.}
\begin{lemma}
	\label{lem::commute_L_modular}
Suppose QMS $\qms_t$ satisfies the GNS-detailed balance, then
\begin{equation}
\Comm{\qms_t}{\Delta_{\sigma}} = 0,\quad \text{ and equivalently }\quad  \Comm{\lbop}{\Delta_{\sigma}} = 0.
\end{equation}
\end{lemma}

\subsection{Noncommutative analog of gradient and divergence}

Given those $V_j$ from \eqref{eqn::heisenberg_lb},
define the \emph{noncommutative partial derivative} $\partial_j$ by
\begin{equation}
\label{eqn::partial_grad}
\partial_j (A) := \Comm{V_j}{A},
\end{equation}
for any matrix $A\in \mset$;
 as a result, its adjoint operator $\partial_j^{\dagger}$ has the form
\begin{equation*}
\partial_j^{\dagger} (A) = \Comm{V_j^*}{A}.
\end{equation*}
Let $\msetJ$ be the vector fields over $\mset$, \ie, each element $\vect{A}\in \msetJ$ has the form $\vect{A} = \begin{pmatrix} A_1 & A_2 & \cdots & A_{\cardn}\end{pmatrix}$
where $A_j\in \mset$ for all $1\le j\le \cardn$.
The \emph{noncommutative gradient} $\nabla: \mset\rightarrow \msetJ$ is defined by
\begin{equation}
\nabla A := \begin{pmatrix}
\partial_1 A & \partial_2 A & \cdots & \partial_{\cardn} A
\end{pmatrix} \in \msetJ.
\end{equation}

The \emph{noncommutative divergence}, acting on $\msetJ$, is defined by
\begin{equation}
\label{eqn::divop}
\divop(\vect{A}) := -\sum_{j=1}^{\cardn}\partial_j^{\dagger} (A_j) \equiv \sum_{j=1}^{\cardn} \Comm{A_j}{V_j^{*}},
\end{equation}
for $\vect{A} = \begin{pmatrix}
A_1 & A_2 & \cdots A_{\cardn}
\end{pmatrix} \in \msetJ$. The divergence is defined as this so that
it is the adjoint operator of gradient, with a multiplicative factor $-1$, analogous to the classical case.

\begin{lemma}[{\cite[Eq. (5.3)]{carlen_gradient_2017}}]
	\label{lem::lb_quadratic}
	For a quantum Markov semigroup with GNS-detailed balance,
	\begin{equation}
	\label{eqn::lb_quadratic}
	\Inner{\nabla A}{\nabla A}_{1/2} = \Inner{A}{-\lbop (A)}_{1/2}.
	\end{equation}
	Thus $-\lbop$ is a positive semi-definite operator on the Hilbert space $(\mset, \Inner{\cdot}{\cdot}_{1/2})$.
\end{lemma}

\begin{lemma}[{\cite[Theorem 5.3]{carlen_gradient_2017}}]
	\label{lem::null_sp_comm}
	Suppose the \lb{} equation satisfies GNS-detailed balance. Then the kernel of $\lbop$ equals to the commutant of $\{V_j\}_{j=1}^{\cardn}$. In other words, $\text{Ker}(\lbop) = \text{Ker}(\nabla)$.
\end{lemma}

The lemma below gives a sufficient condition for \lb{} equations with GNS-detailed balance to be primitive, though the analysis in the rest of this manuscript does not rely on this result. However, the condition $\cardn = n^2-1$ is not necessary for \lb{} equations with GNS-detailed balance to be primitive. For example, the case $n = 2$, $\sigma = \frac{1}{2} \unit$, $V_1 = \sigma_X$ and $V_2 = \sigma_Z$ with $\cardn = 2$ still leads to a primitive Lindblad equation, though $\cardn < n^2-1$; $\sigma_{X}$ and $\sigma_{Z}$ here are Pauli matrices.
\begin{lemma}
	\label{lem::eg_primitive}
	If the \lb{} equation satisfies the GNS-detailed balance and $\cardn = n^2-1$, then the \lb{} equation is primitive.
\end{lemma}

\begin{proof}
Note that $\{\unit, V_1, V_2, \cdots V_{\cardn}\}$ form a (non-normalized) basis of $\mset$. Hence for any matrix $A\in \mset$, one can decompose $A = \frac{\tr(A)}{n}\unit + \wt{A}$ where $\tr(\wt{A}) = 0$.
By the definition of primitive \lb{} equation and \lemref{lem::null_sp_comm}, it is equivalent to show that the commutant of $\{V_j\}_{j=1}^{\cardn}$ is spanned by $\unit$, that is, if $A$ is in the commutant of $\{V_j\}_{j=1}^{\cardn}$, then $\wt{A} = 0$.

If $A$ is in the commutant, then $\Comm{\wt{A}}{V_j} = 0$ for all $j$. Since $\{V_j\}_{j=1}^{\cardn}$ spans the space of traceless matrices, $0 = \Comm{\wt{A}}{\Ket{j}\Bra{k}}$ for all $j\neq k$, that is, $\wt{A}\Ket{j}\Bra{k} = \Ket{j}\Bra{k} \wt{A}$. Considering the matrix element $\Bra{a}\cdot \Ket{b}$, we have $\wt{A}_{a,j} \delta_{k,b} = \wt{A}_{k,b} \delta_{a,j}$. Case (1): take $a = b = k$, then we know that $\wt{A}_{k,j} = 0$, hence off-diagonal terms of $\wt{A}$ are all zero. Case (2): take $a = j$ and $b=k$, then we know that $\wt{A}_{j,j}=\wt{A}_{k,k}$; since $0 = \tr(\wt{A}) = \sum_{j} \wt{A}_{j,j}$, then $\wt{A}_{j,j} = 0$ for all $j$. In summary, $\wt{A} = 0$, thus completes the proof.
\end{proof}

\subsection{Chain rule identity}
\label{sec::chain_rule}
One of the key steps in connecting the \lb{} equation and the gradient flow dynamics of the quantum relative entropy is the following chain rule identity. In \secref{sec::proof} below, this chain rule identity is also the key in proving \lemref{lem::fd}.
\begin{lemma}[Chain rule identity {\cite[Lemma 5.5]{carlen_gradient_2017}}]
	\label{lem::chain_rule}
For any $V\in \mset$, $X\in \psetplus$ and $\omega\in \Real$,
\begin{equation}
\label{eqn::chain_rule}
\mop{X}{\omega} \left(V \log(e^{-\omega/2} X) - \log(e^{\omega/2} X) V \right) = e^{-\omega/2} V X - e^{\omega/2} X V,
\end{equation}
where the operator $\mop{X}{\omega}: \mset \rightarrow \mset$ is defined by
\begin{equation}
\label{eqn::kmb_op}
\mop{X}{\omega} (A) := \int_{0}^{1} e^{\omega(s-1/2)} X^s A X^{1-s}\ \ud s.
\end{equation}
\end{lemma}
This operator $\mop{X}{\omega}$ is a noncommutative multiplication of
the operator $X$;
\revise{for convenience, when $\omega = 0$, let $\mop{X}{} \equiv \mop{X}{0}$.}
\revise{The operator $\mop{X}{\omega}$} can be extended and defined for vector fields
$\msetJ$.  For
$\vect{A} = \begin{pmatrix} A_1 & A_2 & \cdots &
  A_{\cardn} \end{pmatrix}\in \msetJ$
and
$\vect{\omega} = \begin{pmatrix} \omega_1 & \omega_2 & \cdots &
  \omega_{\cardn}
\end{pmatrix} \in \RealJ$, define the operator $\mop{X}{\vect{\omega}}: \msetJ\rightarrow \msetJ$ by
\begin{equation}
\mop{X}{\vect{\omega}} (\vect{A}) :=
\begin{pmatrix}
\mop{X}{\omega_1}(A_1) &
\mop{X}{\omega_2}(A_2) &
\cdots &
\mop{X}{\omega_{\cardn}}(A_{\cardn})
\end{pmatrix}.
\end{equation}

As a reminder, these $\omega_j$ with $1\le j \le \cardn$ come from the spectrum of the modular operator $\Delta_{\sigma}$ (see \thmref{thm::lindblad_detail_balance}).

\begin{lemma}
	\label{lem::positive_X_op}
	If $X\in \psetplus$, then for any $\omega\in \Real$, $\mop{X}{\omega}$ is a strictly positive operator on $\mset$. As a consequence, $\mop{X}{\omega}^{-1}$ is a well-defined strictly positive operator.
\end{lemma}

\begin{proof} Notice that
	\begin{align*}
	\Inner{A}{\mop{X}{\omega} A} &= \int_{0}^{1} e^{\omega(s-1/2)} \tr(A^* X^s A X^{1-s})\ \ud s \\
	&= \int_{0}^{1} e^{\omega(s-1/2)}\Inner{X^{s/2}A X^{(1-s)/2}}{X^{s/2}A X^{(1-s)/2}}\ \ud s \ge 0.
	\end{align*}
	Moreover, $\Inner{A}{\mop{X}{\omega} A} = 0$ iff $A = 0$. Thus $\mop{X}{\omega}$ is a strictly positive operator on $\mset$ for any $\omega\in \Real$. This completes the proof.
\end{proof}

More explicitly, the inverse of $\mop{X}{\omega}$ is (see \cite[Lemma 5.8]{carlen_gradient_2017})
\begin{align}
\label{eqn::inverse_mop}
\begin{split}
\mop{X}{\omega}^{-1} (A) &= \int_{0}^{\infty} \frac{1}{t + e^{\omega/2} X} A \frac{1}{t + e^{-\omega/2} X}\ \ud t\\
&\equiv \revise{\int_{0}^{1} \frac{1}{1-s + s e^{\omega/2} X} A \frac{1}{1-s + s e^{-\omega/2} X}\ \ud s,}
\end{split}
\end{align}
\revise{where we change the variable $s = \frac{1}{t+1}$ from the first to the second line above.}

By expressions \eqref{eqn::kmb_op} and \eqref{eqn::inverse_mop}, it is straightforward to verify the following Lemma.
\begin{lemma}[{\cite[Lemma 5.8]{carlen_gradient_2017}}]
\label{lem::adjoint_mop}
For any $\omega\in \Real$ and $A\in \mset$,
\begin{equation}
\left(\mop{X}{\omega} (A)\right)^{*} = \mop{X}{-\omega} (A^*),\qquad (\mop{X}{\omega}^{-1} (A))^{*} = \mop{X}{-\omega}^{-1} (A^*).
\end{equation}
\end{lemma}

\subsection{Noncommutative $\lpsp$ spaces}
\label{sec::noncommutative_Lp}
In this subsection, we briefly recall some concepts from noncommutative $\lpsp$ spaces and log-Sobolev constants in that setting \cite{olkiewicz_hypercontractivity_1999,kastoryano_quantum_2013,beigi_quantum_2018}, We also recall an implication of the quantum Stroock-Varopoulos inequality \cite{beigi_quantum_2018}, which compares the magnitudes of various log-Sobolev constants.

Let us first introduce the \emph{weighted $\lpsp$ norm}, defined as
\begin{align}
\label{eqn::Lp_sigma}
\Norm{A}_{\alpha,\sigma} := \left(\tr\left( \Abs{\Gamma_{\sigma}^{1/\alpha}(A)}^{\alpha}\right) \right)^{1/\alpha}.
\end{align}
Then we need to introduce the \emph{power operator}, for any $\alpha, \beta\in \Real$, $A\in \mset$,
\begin{equation*}
\powop{\beta}{\alpha}{A} := \Gamma_{\sigma}^{-\frac{1}{\beta}} \left(\Abs{\Gamma_{\sigma}^{\frac{1}{\alpha}}(A)}^{\frac{\alpha}{\beta}}\right).
\end{equation*}
When $A$ is positive and commutes with $\sigma$, $\powop{\beta}{\alpha}{A} = A^{\alpha/\beta}$.

Two most important ingredients in noncommutative $\lpsp$ spaces are entropy function and Dirichlet form.
The $\alpha$-\emph{Entropy function} for $X\in \psetplus$ is defined as
\begin{align}
\begin{split}
\entfun{\alpha}{X} :=  &\tr\left[\left(\Gamma_{\sigma}^{1/\alpha}(X)\right)^\alpha \log\left(\left(\Gamma_{\sigma}^{1/\alpha}(X)\right)^\alpha\right) \right] \\
& - \tr\left[\left(\Gamma_{\sigma}^{1/\alpha}(X)\right)^\alpha \log (\sigma)\right]
 - \Norm{X}_{\alpha,\sigma}^\alpha \log\left(\Norm{X}_{\alpha,\sigma}^\alpha\right),
 \end{split}
\end{align}
and the $\alpha$-\emph{Dirichlet form}, for $\alpha\in (0,1)\cup(1,\infty)$, is defined as
\begin{align}
\dirichlet{\alpha}{X} := \frac{\alpha \wt{\alpha}}{4} \Inner{\powop{\wt{\alpha}}{\alpha}{X}}{-\lbop(X)}_{1/2},
\end{align}
where $1/\wt{\alpha} + 1/\alpha = 1$.
When $\alpha = 1$, the Dirichlet form is defined by taking the limit $\alpha \rightarrow 1$, \ie,
\begin{align*}
\dirichlet{1}{X} := \lim_{\alpha\rightarrow 1}\dirichlet{\alpha}{X} = \frac{1}{4}\Inner{\log(\Gamma_{\sigma}(X)) - \log(\sigma)}{-\lbop(X)}_{1/2}.
\end{align*}
The $\alpha$-log Sobolev constant in this setting is defined as
\begin{equation}
\label{eqn::lsic_Lp}
\kappa_\alpha \equiv \kappa_\alpha(\lbop) := \inf_{X > 0} \frac{\dirichlet{\alpha}{X}}{\entfun{\alpha}{X}},
\end{equation}
and thus the quantum $\alpha$-LSI refers to
\begin{equation}
\label{eqn::lsi_Lp}
\kappa_\alpha \entfun{\alpha}{X} \le \dirichlet{\alpha}{X},\qquad \forall X > 0.
\end{equation}

In order to compare \eqref{eqn::lsi_alpha} and \eqref{eqn::lsi_Lp},
for any matrix $X > 0$, let us introduce $\rho :=\Gamma_{\sigma}(X)/\tr\left(\Gamma_{\sigma}(X)\right)$, which is a strictly positive density matrix.
By inverting the operator $\Gamma_{\sigma}$, $X = \tr\left(\Gamma_{\sigma}(X)\right) \Gamma_{\sigma}^{-1}(\rho)$. By the fact that both $\dirichlet{\alpha}{X}$ and $\entfun{\alpha}{X}$ are homogeneous with respect to $X$ (see, \eg, Proposition 4 (ii) and Proposition 8 (ii) in \cite{beigi_quantum_2018}), we have
\begin{align*}
\kappa_\alpha &= \inf_{\rho \in \dsetplus} \frac{\dirichlet{\alpha}{\Gamma_{\sigma}^{-1}(\rho)}}{\entfun{\alpha}{\Gamma_{\sigma}^{-1}(\rho)}} \\
&= \inf_{\rho \in \dsetplus} \frac{\frac{\alpha^2}{4(\alpha-1)}\Inner{
\Gamma_{\sigma}^{\frac{1-\alpha}{\alpha}}\left(\left(\Gamma_{\sigma}^{\frac{1-\alpha}{\alpha}}(\rho)\right)^{\alpha-1}\right)
}{-\lbop(\Gamma_{\sigma}^{-1}(\rho))}_{1/2}}{Z \relaentropy{\left(\Gamma_{\sigma}^{(1-\alpha)/\alpha}(\rho)\right)^{\alpha}/Z}{\sigma}},\\
\end{align*}
where $Z = \tr\left[\left(\Gamma_{\sigma}^{(1-\alpha)/\alpha}(\rho)\right)^{\alpha}\right]$, which turns out to be the exponent in the sandwiched \renyid{} \eqref{eqn::swch}.
Provided that \lb{} equations satisfy KMS-detailed balance, by \eqref{eqn::lbop_lbop_dag} and \lemref{lem::fd} (see below),
\begin{align}
\label{eqn::kappa_alpha_v2}
\begin{split}
\kappa_\alpha &= \inf_{\rho \in \dsetplus} \frac{\frac{\alpha^2}{4(\alpha-1)}\frac{1}{Z}\Inner{
		\Gamma_{\sigma}^{\frac{1-\alpha}{\alpha}}\left(\left(\Gamma_{\sigma}^{\frac{1-\alpha}{\alpha}}(\rho)\right)^{\alpha-1}\right)
	}{-\lbop^{\dagger}(\rho)}}{ \relaentropy{\left(\Gamma_{\sigma}^{(1-\alpha)/\alpha}(\rho)\right)^{\alpha}/Z}{\sigma}}\\
&= \inf_{\rho \in \dsetplus} \frac{\frac{\alpha}{4}\Inner{\fdrenyidivgalpha{\rho}{\sigma}{\alpha}
	}{-\lbop^{\dagger}(\rho)}}{ \relaentropy{\left(\Gamma_{\sigma}^{(1-\alpha)/\alpha}(\rho)\right)^{\alpha}/Z}{\sigma}} = \inf_{\rho \in \dsetplus} \frac{\frac{\alpha}{4}\relafisheralpha{\rho}{\sigma}{\alpha}}{ \relaentropy{\left(\Gamma_{\sigma}^{(1-\alpha)/\alpha}(\rho)\right)^{\alpha}/Z}{\sigma}}.\\
\end{split}
\end{align}
As one might observe, the $\alpha$-Dirichlet form becomes the quantum relative $\alpha$-Fisher information up to a multiplicative constant, after the above transformation. However, the denominator in the above equation is different from the sandwiched \renyi{} divergence, which makes two definitions of the quantum $\alpha$-LSI in \eqref{eqn::lsi_alpha} and \eqref{eqn::lsi_Lp} different.
When $\alpha = 1$, this log-Sobolev constant $\kappa_1$ becomes
\begin{align}
\label{eqn::kappa_1_K}
\kappa_1 = \frac{1}{4} \inf_{\rho \in \dsetplus} \frac{\relafisheralpha{\rho}{\sigma}{}}{\relaentropy{\rho}{\sigma}} =  \frac{1}{4} 2K = \frac{K}{2}.
\end{align}
Therefore, both definitions \eqref{eqn::lsi_alpha} and \eqref{eqn::lsi_Lp} coincide when $\alpha = 1$, \revise{up to a multiplicative constant $1/2$}.

The following implication of the quantum \svi{} (see \cite[Theorem 14]{beigi_quantum_2018}) characterizes the relationship between various $\kappa_\alpha$.
\begin{theorem}[{\cite[Corollary 16]{beigi_quantum_2018}}]
	\label{thm::qsvi}
	Consider primitive \lb{} equations with GNS-detailed balance. Then $\kappa_{\alpha}$ is non-increasing with respect to $\alpha\in [0,2]$. In particular,
	\begin{equation}
	\label{eqn::qsvi_1_2}
	\kappa_1 \ge \kappa_2.
	\end{equation}
\end{theorem}

\section{\lb{} equations as the gradient flow dynamics of the sandwiched \renyi{} divergence}
\label{sec::proof}

This section devotes to the proof of Theorem 2. We will focus on the case that $\alpha \neq 1$
since the case $ \alpha = 1$ has been treated previously in
\cite{carlen_gradient_2017}.
Alternatively, one may easily adapt the proof below to the case
$\alpha=1$. The proof follows from a
sequence of lemmas whose proofs are postponed to the end of this section.

\subsection{Proof of Theorem \ref{thm::renyi_gradient}}

To simplify notations,
for any density matrix $\rho\in \dset$, define
\begin{equation}
\label{eqn::rho_sigma}
	\rho_{\sigma} := \Gamma_{\sigma}^{\frac{1-\alpha}{\alpha}} (\rho) \equiv \sigma^{\frac{1-\alpha}{2\alpha}} \rho \sigma^{\frac{1-\alpha}{2\alpha}}.
\end{equation}

Then the \swch{} for $\alpha \neq 1$ can then be rewritten as
\begin{equation*}
\renyidivg{\rho}{\sigma} = \frac{1}{\alpha-1}\log\left( \tr(\rho_{\sigma}^{\alpha})\right).
\end{equation*}

The following lemma, connecting the functional derivative of \swch{} and \lb{} equation,
plays an essential role in the proof of  \thmref{thm::renyi_gradient}.
\begin{lemma}[Functional derivative]
	\label{lem::fd}
	The functional derivative of the \swch{} in the space $\dset$ is
	\begin{equation}
	\label{eqn::fdrenyi}
	\fdrenyidivg{\rho}{\sigma} = \frac{\alpha}{\alpha-1} \frac{\Gamma_{\sigma}^{\frac{1-\alpha}{\alpha}}(\rho_{\sigma}^{\alpha-1})  }{\tr\left(\rho_{\sigma}^{\alpha}\right)},
	\end{equation}
	and moreover for $\rho\in \dsetplus$,
	\begin{equation}
	\label{eqn::pdfdrenyi}
		\moprenyi{\rho}{\omega_j} \partial_j \left(\fdrenyidivg{\rho}{\sigma}\right)
		=   e^{-\omega_j/2} V_j \rho - e^{\omega_j/2} \rho V_j,
	\end{equation}
	where the operator $\moprenyi{X}{\omega}:\mset\rightarrow \mset$, for positive $X\in \psetplus$ and $\omega\in \Real$, is defined by
		\begin{equation}
		\label{eqn::moprenyi}
		\moprenyi{X}{\omega} :=
		\frac{ \tr\left(\bigl(\Gamma_{\sigma}^{\frac{1-\alpha}{\alpha}} (X)\bigr)^{\alpha} \right)}{\alpha}
		\Gamma_{\sigma}^{\frac{\alpha-1}{\alpha}}
		\circ \mop{\Gamma_{\sigma}^{\frac{1-\alpha}{\alpha}} (X)}{\omega/\alpha}
		\circ \mop{\bigl(\Gamma_{\sigma}^{\frac{1-\alpha}{\alpha}} (X)\bigr)^{\alpha-1}}{(\alpha-1)\omega/{\alpha}}^{-1}
		\circ \Gamma_{\sigma}^{\frac{\alpha-1}{\alpha}}.
		\end{equation}
\end{lemma}

\begin{remark}
	\begin{enumerate}[1)]
    \item When $\alpha = 1$,
      $\Gamma_{\sigma}^{\frac{1-\alpha}{\alpha}} = \id$ and
      $\rho_{\sigma} = \rho$. Then $\moprenyi{\rho}{\omega}$ reduces
      to $ [\rho]_{\omega}$. Hence, $\moprenyi{\rho}{\omega}$ is the generalization of $\mop{\rho}{\omega}$ to the case of \swch{}.

    \item When $\alpha=2$, the above operator
      $\moprenyialpha{\rho}{\omega}{2}$ has a simple form
		\begin{align}
		\moprenyialpha{\rho}{\omega}{2} = \frac{\tr(\rho_{\sigma}^2)}{2} \Gamma_{\sigma}  = \frac{\tr(\sigma^{-1/2}\rho \sigma^{-1/2} \rho)}{2} \Gamma_{\sigma}.
		\end{align}
		Thus the mapping
		\begin{align}
		(A, \rho) \rightarrow \Inner{A}{\moprenyialpha{\rho}{\omega}{2} (A)} = \frac{ \Inner{\rho}{\sigma^{-1/2}\rho \sigma^{-1/2}}\Inner{A}{\sigma^{1/2} A \sigma^{1/2}}}{2}
		\end{align}
		is the multiplication of quadratic terms with respect to both $A$ and $\rho$.
	\end{enumerate}

\end{remark}

The properties of the operator $\mop{X}{\omega}$ stated
in \lemref{lem::positive_X_op} and \lemref{lem::adjoint_mop}  generalize to the operator
 $\moprenyi{X}{\omega}$,  as we summarize in the next lemma.
\begin{lemma}
	\label{lem::moprenyid_propty}
	If $X \in \psetplus$ and $\omega\in \Real$, then
\begin{enumerate}
	\item $\moprenyi{X}{\omega}: \mset\rightarrow \mset$ is strictly positive;
	\item For any matrix $A$,
	\begin{align*}
	\left(\moprenyi{X}{\omega} (A)\right)^{*}  = \moprenyi{X}{-\omega} (A^*).
	\end{align*}
\end{enumerate}
\end{lemma}

The operator $\moprenyi{X}{\omega}$ can be similarly extended to the space $\msetJ$.
For vector fields $\vect{A} = \begin{pmatrix}
A_1 & A_2 & \cdots & A_{\cardn} \end{pmatrix}\in \msetJ$ and $\vect{\omega} = \begin{pmatrix}
\omega_1 & \omega_2 & \cdots & \omega_{\cardn}
\end{pmatrix} \in \RealJ$, we define
\begin{equation}
\moprenyi{X}{\vect{\omega}} (\vect{A}) :=
\begin{pmatrix}
\moprenyi{X}{\omega_1} (A_1) &
\moprenyi{X}{\omega_2} (A_2) &
\cdots &
\moprenyi{X}{\omega_{\cardn}} (A_{\cardn})
\end{pmatrix},
\end{equation}
and analogously,
\begin{equation}
\moprenyi{X}{\vect{\omega}}^{-1} (\vect{A}) :=
\begin{pmatrix}
\moprenyi{X}{\omega_1}^{-1} (A_1) &
\moprenyi{X}{\omega_2}^{-1} (A_2) &
\cdots &
\moprenyi{X}{\omega_{\cardn}}^{-1} (A_{\cardn})
\end{pmatrix}.
\end{equation}

The following lemma will be useful to characterize the tangent space of density matrices.
\begin{lemma}
\label{lem::diffoptilde}
Assume that the Lindblad equation \eqref{eqn::lb} is primitive and $X \in \psetplus$. Denote the space of traceless matrices by $\mset_0$ and denote the space of traceless Hermitian matrices by $\aset_0$, that is, $\aset_0 = \mset_0\cap \aset$.
\begin{enumerate}
	\item Given $\vect{\omega}\in \RealJ$,
	define the operator $\diffoptilde{X}{\vect{\omega}}:\mset_0\rightarrow \mset_0$ by
\begin{equation*}
	\diffoptilde{X}{\vect{\omega}}(A) := - \divop\left(\moprenyi{X}{\vect{\omega}} (\nabla A)\right).
	\end{equation*}
	This operator $\diffoptilde{X}{\vect{\omega}}$ is strictly positive, and thus invertible.

\item For any $A\in \mset_0$,
\begin{align*}
(\diffoptilde{X}{\vect{\omega}} (A))^{*} = \diffoptilde{X}{\vect{\omega}} (A^{*}),
\end{align*}
and the restriction of $\diffoptilde{X}{\vect{\omega}}$ on $\aset_0$ is still strictly positive and invertible.
\end{enumerate}
\end{lemma}
In the case $\alpha = 1$, this lemma has been proved in  \cite[Lemma 3]{rouze_relating_2017}.

Before defining the metric tensor, we need to first introduce an inner product, weighted by $\moprenyi{\rho}{\vect{\omega}}$.
\begin{defn}[Inner product]
	Given $\alpha\in (0,\infty)$, $\rho\in \dsetplus$, for vector fields $\vect{A} = \begin{pmatrix}
	A_1 & A_2 & \cdots & A_{\cardn} \end{pmatrix}$ and
	$\vect{B} = \begin{pmatrix}
	B_1 & B_2 & \cdots & B_{\cardn} \end{pmatrix}$,
	an inner product on vector fields is defined as
	\begin{equation}
	\InnerAlphaLB{\vect{A}}{\vect{B}}
	:= \sum_{j=1}^{\cardn} \Inner{A_j}{\moprenyi{\rho}{\omega_j} (B_j)}.
	\end{equation}
\end{defn}

Given $\rho\in \dsetplus$, the tangent space is $\aset_0$. For $\nu_1$ and $\nu_2\in \aset_0$, by \lemref{lem::diffoptilde} part (2),
there is a unique $U_k\in \aset_0$, such that
 \begin{equation}
 \label{eqn::tang_sp_relation}
 \nu_k = \diffoptilde{\rho}{\vect{\omega}} (U_k) = -\divop\left(\moprenyi{\rho}{\vect{\omega}}(\nabla U_k)\right),\qquad k = 1, 2.
 \end{equation}

Finally, we are ready to define the metric tensor we need.
\begin{defn}[Metric tensor]\label{def::metric_tensor_renyi}
	Given $\rho\in \dsetplus$, for $\nu_{k}\in \tang_{\rho} \dset \equiv \aset_0$, define a metric tensor $g_\rhodep$ by
	\begin{equation}
	\label{eqn::metric_tensor_renyi}
	g_{\rhodep}(\nu_1, \nu_2) := \InnerAlphaLB{\nabla U_1}{\nabla U_2} \equiv \sum_{j=1}^{\cardn} \Inner{\partial_j U_1}{\moprenyi{\rho}{\omega_j}(\partial_j U_2)},
	\end{equation}
	where $U_k$ and $\nu_k$ are linked via \eqref{eqn::tang_sp_relation}.
\end{defn}

\begin{proof}[Proof of \thmref{thm::renyi_gradient}]

Up to now, the \rie{} structure $(\dsetplus, g_{\rhodep})$ has been specified, as well as the energy functional $E(\rho) \equiv \renyidivg{\rho}{\sigma}$. Then, we will verify that the gradient flow dynamics defined in  \eqref{eqn::riemannian_gradient_flow} is exactly the \lb{} equation in \eqref{eqn::lb}.
To achieve that, we need to compute the gradient at any
$\rho_0\in \dsetplus$.

Recall the notion of gradient in the \rie{} manifold \eqref{eqn::riemannian_grad}, $\gradt E\vert_{\rho_0}$ is defined as the traceless Hermitian matrix such that for any differentiable trajectory $\rho_t\in \dsetplus$ with $\rho_0$ at time $t=0$,
\begin{equation*}
\begin{split}
g_{\alpha,\rho_0,\lbop}\left(\gradt E\eva_{\rho_0}, \dot{\rho}_t\eva_{t=0}
\right) \equiv \frac{\ud }{\ud t}\eva_{t=0} E(\rho_t) \equiv \frac{\ud }{\ud t}\eva_{t=0} \renyidivg{\rho_t}{\sigma}.
\end{split}
\end{equation*}
By direct computation,
\begin{equation*}
\begin{split}
\frac{\ud }{\ud t}\eva_{t=0} \renyidivg{\rho_t}{\sigma} &= \Inner{\fdrenyidivg{\rho}{\sigma}\eva_{\rho=\rho_0}}{\dot{\rho}_t \eva_{t=0}} \\
&= \Inner{\fdrenyidivg{\rho}{\sigma}\eva_{\rho=\rho_0}}{-\divop\left(\moprenyi{\rho_0}{\vect{\omega}} (\nabla U_0) \right)} \\
&= \Inner{\nabla \fdrenyidivg{\rho}{\sigma}\eva_{\rho=\rho_0}}{\moprenyi{\rho_0}{\vect{\omega}} (\nabla U_0)}\\
&= \Inner{\nabla \fdrenyidivg{\rho}{\sigma}\eva_{\rho=\rho_0}}{\nabla U_0}_{\alpha,\rho_0, \lbop},
\end{split}
\end{equation*}
where we used the following fact that follows from \lemref{lem::diffoptilde}: there exists a
unique family of traceless Hermitian matrices $\{U_t\}_{t\geq 0}$ such that $\dot{\rho}_t = - \divop(\moprenyi{\rho_t}{\vect{\omega}}\nabla U_t)$. After combing the last two equations, we have
\begin{equation*}
g_{\alpha,\rho_0,\lbop}\left(\gradt E\eva_{\rho_0}, \dot{\rho}_t\eva_{t=0}
\right) = \Inner{\nabla \fdrenyidivg{\rho}{\sigma}\eva_{\rho=\rho_0}}{\nabla U_0}_{\alpha,\rho_0, \lbop}.
\end{equation*}
By the definition of the metric tensor $g_{\rhodep}(\cdot, \cdot)$  \eqref{eqn::metric_tensor_renyi},
\begin{align*}
\gradt E\eva_{\rho_0} &= \diffoptilde{\rho_0}{\vect{\omega}}\left(\fdrenyidivg{\rho}{\sigma}\eva_{\rho=\rho_0}  \right) \\
&=  -  \divop\left(\moprenyi{\rho_0}{\vect{\omega}} \left(\nabla \fdrenyidivg{\rho}{\sigma}\eva_{\rho=\rho_0} \right)\right).
\end{align*}
Consequently, the gradient flow dynamics \eqref{eqn::riemannian_gradient_flow} is
\begin{equation*}
\dot{\rho}_t = - \gradt E\eva_{\rho_t} = \divop\left(\moprenyi{\rho_t}{\vect{\omega}} \left(\nabla \fdrenyidivg{\rho}{\sigma}\eva_{\rho=\rho_t}\right)\right).
\end{equation*}
Let us rewrite the term on the right hand side
\begin{align}
\label{eqn::lb_renyi_var}
\begin{split}
&\divop\left(\moprenyi{\rho_t}{\vect{\omega}} \left(\nabla \fdrenyidivg{\rho}{\sigma}\eva_{\rho=\rho_t}\right)\right) \\
\myeq{\eqref{eqn::divop}}& \sum_{j=1}^{\cardn} \Comm{\moprenyi{\rho_t}{\omega_j}\left(\partial_j  \fdrenyidivg{\rho}{\sigma}\eva_{\rho=\rho_t} \right)}{V_j^*} \\
\myeq{\eqref{eqn::pdfdrenyi}}& \sum_{j=1}^{\cardn} \Comm{e^{-\omega_j/2} V_j \rho_t - e^{\omega_j/2}\rho_t V_j}{V_j^*}\\
\myeq{\eqref{eqn::sch_lb}}& \lbop^{\dagger}(\rho_t).
\end{split}
\end{align}
Thus the gradient flow dynamics is exactly the \lb{} equation in \eqref{eqn::lb}.
\end{proof}

\begin{proof}[{Proof of \coref{coro::monotonicity_renyi}}]
If $\rho_t$ is the solution of the Lindblad equation,  then $U_0 =- \fdrenyidivg{\rho}{\sigma}\eva_{\rho=\rho_0}$. Hence,
\begin{align*}
\frac{\ud }{\ud t}\eva_{t=0} \renyidivg{\rho_t}{\sigma} &= \Inner{\nabla \fdrenyidivg{\rho}{\sigma}\eva_{\rho=\rho_0}}{\nabla U_0}_{\alpha,\rho_0, \lbop}\\
&= -\Inner{\nabla \fdrenyidivg{\rho}{\sigma}\eva_{\rho=\rho_0}}{\nabla \fdrenyidivg{\rho}{\sigma}\eva_{\rho=\rho_0}}_{\alpha,\rho_0, \lbop} \le 0.
\end{align*}
\end{proof}

\subsection{Proof of Lemmas} Here we present the proofs of lemmas that we used in the previous subsection.
\begin{proof}[Proof of \lemref{lem::fd}]
	For any differentiable curve $\rho_t \in \dset$ passing through $\rho$ at time $t=0$, that is, $\rho_0 = \rho$, we have
	\begin{align*}
	\Inner{\fdrenyidivg{\rho}{\sigma}}{\dot{\rho}_0} & \equiv \frac{\ud }{\ud t}\eva_{t=0} \renyidivg{\rho_t}{\sigma} \\
	&= \frac{\alpha}{\alpha-1} \frac{\tr\left( \rho_{\sigma}^{\alpha-1} \sigmalpha\ \dot{\rho}_0\ \sigmalpha \right)}{\tr\left(\rho_{\sigma}^{\alpha}\right)} \\
	&= \frac{\alpha}{\alpha-1} \frac{\tr\left( \sigmalpha \rho_{\sigma}^{\alpha-1} \sigmalpha\ \dot{\rho}_0 \right)}{\tr\left(\rho_{\sigma}^{\alpha}\right)}.
	\end{align*}
	Then it is straightforward to obtain \eqref{eqn::fdrenyi}.

	Next, we compute the partial derivative of $\fdrenyidivg{\rho}{\sigma}$. In fact,
	\begin{align}
	\label{eqn::pdfdrenyi_part}
	\begin{split}
	& \partial_j \left(\fdrenyidivg{\rho}{\sigma}\right)\\
	=&\frac{\alpha}{(\alpha-1)\tr\left(\rho_{\sigma}^{\alpha}\right)} \Comm{V_j}{\sigmalpha \rho_{\sigma}^{\alpha-1} \sigmalpha} \\
	=& \frac{\alpha}{(\alpha-1) \tr\left(\rho_{\sigma}^{\alpha} \right)} \sigmalpha \left(\sigma^{\frac{\alpha-1}{2\alpha}} V_j \sigmalpha \rho_{\sigma}^{\alpha-1} - \rho_{\sigma}^{\alpha-1} \sigmalpha V_j \sigma^{\frac{\alpha-1}{2\alpha}} \right) \sigmalpha\\
	=& \frac{\alpha}{(\alpha-1) \tr\left(\rho_{\sigma}^{\alpha} \right)} \Gamma_{\sigma}^{\frac{1-\alpha}{\alpha}}
	\left(e^{-\frac{\alpha-1}{2\alpha}\omega_j} V_j \rho_{\sigma}^{\alpha-1} - e^{\frac{\alpha-1}{2\alpha}\omega_j}\rho_{\sigma}^{\alpha-1} V_j \right)\\
	\myeq{\eqref{eqn::chain_rule}}& \begin{aligned}
	 \frac{\alpha}{(\alpha-1) \tr\left(\rho_{\sigma}^{\alpha} \right)} \Gamma_{\sigma}^{\frac{1-\alpha}{\alpha}}
	\Big(\mop{\rho_{\sigma}^{\alpha-1}}{\frac{\alpha-1}{\alpha}\omega_j} \Big(& V_j \log( e^{-\frac{\alpha-1}{2\alpha}\omega_j} \rho_{\sigma}^{\alpha-1})
	 - \log(e^{\frac{\alpha-1}{2\alpha} \omega_j} \rho_{\sigma}^{\alpha-1}) V_j \Big) \Big) 
	\end{aligned} \\
	=& \frac{\alpha}{ \tr\left(\rho_{\sigma}^{\alpha} \right)} \left(\Gamma_{\sigma}^{\frac{1-\alpha}{\alpha}}
	\circ \mop{\rho_{\sigma}^{\alpha-1}}{\frac{\alpha-1}{\alpha}\omega_j}\right)
	\left(V_j \log(e^{-\frac{\omega_j}{2\alpha}} \rho_{\sigma}) - \log(e^{\frac{\omega_j}{2\alpha}} \rho_{\sigma}) V_j \right).
	\end{split}
	\end{align}
	To get the fifth line, we have used \eqref{eqn::chain_rule} with $X = \rho_{\sigma}^{\alpha-1}$ and $\omega = \frac{\alpha-1}{\alpha} \omega_j$ and $V = V_j$.
Note that since $\rho, \sigma\in \dsetplus$, $\rho_{\sigma}, \rho_{\sigma}^{\alpha-1}\in \psetplus$.

	Furthermore, applying \eqref{eqn::chain_rule} again with $X = \rho_{\sigma}$, $\omega = \omega_j/\alpha$ and $V = V_j$,
	\begin{equation*}
	\begin{split}
	& \mop{\rho_{\sigma}}{\omega_j/\alpha}
	\left(V_j \log\bigl(e^{-\omega_j/(2\alpha)} \rho_{\sigma}\bigr) - \log\bigl(e^{\omega_j/(2\alpha)} \rho_{\sigma}\bigr) V_j \right) \\
	=& e^{-\omega_j/(2\alpha)} V_j \rho_{\sigma} - e^{\omega_j/(2\alpha)} \rho_{\sigma} V_j  \\
	=& e^{-\omega_j/(2\alpha)} \sigmalpha \left(\sigma^{\frac{\alpha-1}{2\alpha}} V_j \sigmalpha\right) \rho \sigmalpha - e^{\omega_j/(2\alpha)} \sigmalpha \rho \left(\sigmalpha V_j \sigma^{\frac{\alpha-1}{2\alpha}}\right) \sigmalpha \\
	=&  \Gamma_{\sigma}^{\frac{1-\alpha}{\alpha}} \left( e^{-\omega_j/2} V_j \rho - e^{\omega_j/2} \rho V_j \right).
	\end{split}
	\end{equation*}
	Therefore,
	\begin{equation}
	\begin{split}
	&e^{-\omega_j/2} V_j \rho - e^{\omega_j/2} \rho V_j 	= \left(\Gamma_{\sigma}^{\frac{\alpha-1}{\alpha}} \circ \mop{\rho_{\sigma}}{\omega_j/\alpha} \right)
	\left(V_j \log\bigl(e^{-\omega_j/(2\alpha)} \rho_{\sigma}\bigr) - \log\bigl(e^{\omega_j/(2\alpha)} \rho_{\sigma}\bigr) V_j \right). \\
	\end{split}
	\end{equation}
	After plugging it back into \eqref{eqn::pdfdrenyi_part}, one could straightforwardly obtain \eqref{eqn::pdfdrenyi} after arranging a few terms and
	\begin{equation}
	\moprenyi{\rho}{\omega_j}^{-1} := \frac{\alpha}{ \tr\left(\rho_{\sigma}^{\alpha} \right)} \left(\Gamma_{\sigma}^{\frac{1-\alpha}{\alpha}}
	\circ \mop{\rho_{\sigma}^{\alpha-1}}{\frac{\alpha-1}{\alpha}\omega_j}\right) \circ \left(\Gamma_{\sigma}^{\frac{\alpha-1}{\alpha}} \circ \mop{\rho_{\sigma}}{\omega_j/\alpha} \right)^{-1}.
	\end{equation}
	This becomes \eqref{eqn::moprenyi} after we invert this operator.
\end{proof}

\begin{proof}[Proof of \lemref{lem::moprenyid_propty}]
\begin{enumerate}
\item Since $X\in \psetplus$, $\sigma\in \dsetplus$, then $X_{\sigma} := \Gamma_{\sigma}^{\frac{1-\alpha}{\alpha}}(X)$ is strictly positive. For any $A\in \mset$, let $\wt{A} := \Gamma_{\sigma}^{\frac{\alpha-1}{\alpha}} (A)$, then
\begin{align*}
& \Inner{A}{\moprenyi{X}{\omega} (A)} \\
=& \frac{\tr(X_{\sigma}^\alpha)}{\alpha} \Inner{A}{\Gamma_{\sigma}^{\frac{\alpha-1}{\alpha}} \circ \mop{X_{\sigma}}{\omega/\alpha}\circ \mop{X_{\sigma}^{\alpha-1}}{(\alpha-1)\omega/\alpha}^{-1} \circ \Gamma_{\sigma}^{\frac{\alpha-1}{\alpha}} (A)} \\
=& \frac{\tr(X_{\sigma}^\alpha)}{\alpha} \Inner{\wt{A}}{\mop{X_{\sigma}}{\omega/\alpha}\circ \mop{X_{\sigma}^{\alpha-1}}{(\alpha-1)\omega/\alpha}^{-1} (\wt{A})} \\
\myeq{\eqref{eqn::kmb_op}}& \frac{\tr(X_{\sigma}^\alpha)}{\alpha}
\int_{0}^{1} e^{\frac{\omega}{\alpha}(s-\frac{1}{2})} \tr\left(\wt{A}^{*}  X_{\sigma}^{s} \left(\mop{X_{\sigma}^{\alpha-1}}{(\alpha-1)\omega/\alpha}^{-1} (\wt{A})\right)   X_{\sigma}^{1-s}\right)\ud s \\
=& \frac{\tr(X_{\sigma}^\alpha)}{\alpha}
\int_{0}^{1} e^{\frac{\omega}{\alpha}(s-\frac{1}{2})} \tr\left(X_{\sigma}^{\frac{1-s}{2}}\wt{A}^{*} X_{\sigma}^{\frac{s}{2}} X_{\sigma}^{\frac{s}{2}} \left(\mop{X_{\sigma}^{\alpha-1}}{(\alpha-1)\omega/\alpha}^{-1} (\wt{A})\right)   X_{\sigma}^{\frac{1-s}{2}}\right)\ud s \\
\myeq{\eqref{eqn::inverse_mop}}& \frac{\tr(X_{\sigma}^\alpha)}{\alpha}
\int_{0}^{1} e^{\frac{\omega}{\alpha}(s-\frac{1}{2})} \tr\left(X_{\sigma}^{\frac{1-s}{2}}\wt{A}^{*} X_{\sigma}^{\frac{s}{2}} \left(\mop{X_{\sigma}^{\alpha-1}}{(\alpha-1)\omega/\alpha}^{-1} (X_{\sigma}^{\frac{s}{2}} \wt{A} X_{\sigma}^{\frac{1-s}{2}})\right)   \right)\ud s \\
=& \frac{\tr(X_{\sigma}^\alpha)}{\alpha}
\int_{0}^{1} e^{\frac{\omega}{\alpha}(s-\frac{1}{2})} \Inner{X_{\sigma}^{\frac{s}{2}} \wt{A} X_{\sigma}^{\frac{1-s}{2}} }{\mop{X_{\sigma}^{\alpha-1}}{(\alpha-1)\omega/\alpha}^{-1} (X_{\sigma}^{\frac{s}{2}} \wt{A} X_{\sigma}^{\frac{1-s}{2}}) }\ud s \ge 0,
\end{align*}
since $\mop{X_{\sigma}^{\alpha-1}}{(\alpha-1)\omega/\alpha}^{-1}$ is a strictly positive operator by \lemref{lem::positive_X_op}.
Additionally, $\Inner{A}{\moprenyi{X}{\omega} (A)} = 0$ if and only if $0 = X_{\sigma}^{\frac{s}{2}}\wt{A} X_{\sigma}^{\frac{1-s}{2}} \equiv X_{\sigma}^{\frac{s}{2}} \left(\Gamma_{\sigma}^{\frac{\alpha-1}{\alpha}} (A)\right) X_{\sigma}^{\frac{1-s}{2}}$, which implies that $A = 0$.

\item For any matrix $A$, by \lemref{lem::adjoint_mop},
\begin{align*}
\left(\moprenyi{X}{\omega} (A)\right)^{*}
&= \frac{\tr\left(X_{\sigma}^{\alpha} \right)}{\alpha} \Gamma_{\sigma}^{\frac{\alpha-1}{\alpha}} \left(\mop{X_{\sigma}}{\omega/\alpha}\circ \mop{X_{\sigma}^{\alpha-1}}{(\alpha-1)\omega/\alpha}^{-1} \circ \Gamma_{\sigma}^{\frac{\alpha-1}{\alpha}} (A) \right)^{*} \\
&= \frac{\tr\left(X_{\sigma}^{\alpha} \right)}{\alpha} \Gamma_{\sigma}^{\frac{\alpha-1}{\alpha}}\circ \mop{X_{\sigma}}{-\omega/\alpha} \left( \mop{X_{\sigma}^{\alpha-1}}{(\alpha-1)\omega/\alpha}^{-1} \circ \Gamma_{\sigma}^{\frac{\alpha-1}{\alpha}} (A) \right)^{*}\\
&= \frac{\tr\left(X_{\sigma}^{\alpha} \right)}{\alpha} \Gamma_{\sigma}^{\frac{\alpha-1}{\alpha}}\circ \mop{X_{\sigma}}{-\omega/\alpha} \circ \mop{X_{\sigma}^{\alpha-1}}{-(\alpha-1)\omega/\alpha}^{-1} \left(\Gamma_{\sigma}^{\frac{\alpha-1}{\alpha}} (A) \right)^{*} \\
&= \frac{\tr\left(X_{\sigma}^{\alpha} \right)}{\alpha} \Gamma_{\sigma}^{\frac{\alpha-1}{\alpha}}\circ \mop{X_{\sigma}}{-\omega/\alpha} \circ \mop{X_{\sigma}^{\alpha-1}}{-(\alpha-1)\omega/\alpha}^{-1} \circ \Gamma_{\sigma}^{\frac{\alpha-1}{\alpha}} (A^*) \\
&= \moprenyi{X}{-\omega} (A^*).
\end{align*}
\end{enumerate}
\end{proof}

\begin{proof}[Proof of \lemref{lem::diffoptilde}]
\begin{enumerate}
\item Consider a traceless matrix $A\in \mset_0$. It is easy to verify that
$\tr\left(\diffoptilde{X}{\vect{\omega} }(A)\right) = 0$.
From the definition of the divergence operator \eqref{eqn::divop}, we have
\begin{align*}
\Inner{A}{\diffoptilde{X}{\vect{\omega}} (A)} = \sum_{j=1}^{\cardn} \Inner{\partial_j A}{\moprenyi{X}{\omega_j} (\partial_j A)} \ge 0.
\end{align*}
Moreover, since $\moprenyi{X}{\omega_j}$  is strictly positive by \lemref{lem::moprenyid_propty}, we know that  $\Inner{A}{\diffoptilde{X}{\vect{\omega}}
(A)} = 0$ iff $\partial_j A = 0$ for all $1\le j\le \cardn$.
Thanks to  \lemref{lem::null_sp_comm} and the assumption that the
\lb{} equation \eqref{eqn::lb} is primitive, we have $A = 0$ since $A$ is assumed to be traceless.
Therefore, $\diffoptilde{X}{\vect{\omega}}$ is a strictly positive operator
on the space of traceless matrices $\mset_0$.

\item
We only need to prove that $(\diffoptilde{X}{\vect{\omega}} (A))^{*}
= \diffoptilde{X}{\vect{\omega}} (A^{*})$.
The rest simply follows from the first part of the lemma. Note that
\begin{align*}
 \diffoptilde{X}{\vect{\omega}}(A) &= - \divop\left(\moprenyi{X}{\vect{\omega}} (\nabla A)\right)
\myeq{\eqref{eqn::divop}} \sum_{j} \Comm{V_j^*}{\moprenyi{X}{\omega_j}(\partial_j A) } \\
&= \sum_{j} V_j^* \left(\moprenyi{X}{\omega_j}(V_j A - A V_j)\right) - \left(\moprenyi{X}{\omega_j}(V_j A - A V_j)\right)  V_j^* .
\end{align*}
Hence by \thmref{thm::lindblad_detail_balance} and \lemref{lem::moprenyid_propty}, we have
\begin{align*}
& \left(\diffoptilde{X}{\vect{\omega}}(A)\right)^* \\
=& \sum_{j} \left(\moprenyi{X}{-\omega_j}(A^* V_j^* - V_j^* A^*)\right) V_j
- V_j \left(\moprenyi{X}{-\omega_j}(A^* V_j^* - V_j^* A^*)\right) \\
=& \sum_{j'} \left(\moprenyi{X}{\omega_{j'}}(A^* V_{j'} - V_{j'} A^*)\right) V_{j'}^*
- V_{j'}^{*} \left(\moprenyi{X}{\omega_{j'}}(A^* V_{j'} - V_{j'} A^*)\right) \\
=& \diffoptilde{X}{\vect{\omega}} (A^{*}) .
\end{align*}
\end{enumerate}
\end{proof}

\section{The necessary condition for Lindblad equations to be the gradient flow dynamics of sandwiched \renyi{} divergences}
\label{sec::necessary_condition}

Recall from \thmref{thm::renyi_gradient} that Lindblad equations with GNS-detailed balance can be regarded as the gradient flow dynamics of sandwiched \renyi{} divergences, including the quantum relative entropy.
A natural and immediate following-up question is the extent that one can possibly generalize \thmref{thm::renyi_gradient} in the sense of considering a larger family of Lindblad equations.
In a recent paper \cite{carlen_non-commutative_2018}, such an issue has been briefly addressed; however, currently, there is still a gap between the class of Lindblad equations (\ie, primitive Lindblad equations with GNS detailed balance) that are known to be the gradient flow dynamics of the quantum relative entropy and the necessary condition (\ie, BKM detailed balance condition) for a primitive Lindblad equation to be possibly expressed as the gradient flow dynamics of the quantum relative entropy; this gap will be revisited and explained in more details below.

In this section, we shall explore the necessary condition for \lb{} equations to be the gradient flow dynamics of sandwiched \renyi{} divergences, summarized in  \thmref{thm::lb_necessary}, which adapts the argument from
 \cite[Theorem 2.9]{carlen_non-commutative_2018}. Next, in \secref{sec::necessity_decomposition}, we will discuss detailed balance conditions arising from \thmref{thm::lb_necessary} and in \secref{sec::necessity_relation}, we will discuss relations between various detailed balance conditions.

\subsection{Necessary condition}

 Below is the main result of this section. Recall the notation $\mop{X}{\omega}
 $ \eqref{eqn::kmb_op} and recall that $\mop{X}{} \equiv \mop{X}{0}$.
\begin{theorem}
\label{thm::lb_necessary}
Suppose that the dual QMS $\qms^{\dagger}_t = e^{t\lbop^{\dagger}}$ is primitive and its unique stationary state is $\sigma\in \dsetplus$. If there exists a continuously differentiable metric tensor $g_{\rho}(\cdot, \cdot): \aset_0\times \aset_0\rightarrow\Real$ for any $\rho \in \dsetplus$ such that the Lindblad equation $\dot{\rho}_t = \lbop^{\dagger}(\rho_t)$ is the gradient flow dynamics of the sandwiched \renyi{} divergence $\renyidivg{\rho}{\sigma}$, then $\lbop$ is self-adjoint with respect to the inner product weighted by an operator $\mathscr{W}_{\sigma, \alpha}: \mset \rightarrow \mset$ defined by
\begin{align}
\label{eqn::weight_op_w}
\begin{split}
\mathscr{W}_{\sigma, \alpha}(A) &:= \mop{\sigma^{\frac{1}{\alpha}}}{} \circ \mop{\sigma^{\frac{\alpha-1}{\alpha}}}{}^{-1} \circ \Gamma_{\sigma}^{\frac{2(\alpha-1)}{\alpha}}(A) \\
& \equiv \int_{0}^{1} \int_{0}^{1} \left(\frac{\bigl(\sigma^{\frac{1}{\alpha}}\bigr)^{1-s} \sigma^{\frac{\alpha-1}{\alpha}}}{(1-r) \unit + r \sigma^{\frac{\alpha-1}{\alpha}}}\right) A \left(\frac{\sigma^{\frac{\alpha-1}{\alpha}} \bigl(\sigma^{\frac{1}{\alpha}}\bigr)^{s} }{(1-r) \unit + r \sigma^{\frac{\alpha-1}{\alpha}}}\right) \ud r \ud s,
\end{split}
\end{align}
for all $A\in \mset$;
more specifically, we have for all $A, B\in \mset$,
\begin{align}
\label{eqn::general_db_cond}
\Inner{\lbop(A)}{B}_{\mathscr{W}_{\sigma, \alpha}} = \Inner{A}{ \lbop(B)}_{\mathscr{W}_{\sigma, \alpha}},
\end{align}
where $\Inner{A}{B}_{\mathscr{W}_{\sigma, \alpha}} := \Inner{A}{\mathscr{W}_{\sigma, \alpha}(B)}$.
\end{theorem}

\begin{remark}
The operator $\mathscr{W}_{\sigma, \alpha}$ is a noncommutative way to multiply $\sigma$ and the condition \eqref{eqn::general_db_cond} is equivalent to
\begin{align}
\label{eqn::general_db_cond_2}
\mathscr{W}_{\sigma, \alpha} \circ \lbop \circ \mathscr{W}_{\sigma, \alpha}^{-1} = \lbop^{\dagger}.
\end{align}
\end{remark}

Motivated by the definition of GNS and KMS detailed balance conditions, we shall define a detailed balance condition based on sandwiched \renyi{} divergences, for the convenience of discussion below.
\begin{defn}[Sandwiched \renyi{} divergence (SRD) detailed balance condition]
\label{defn::srd_dbc}
For a primitive Lindblad equation with a unique stationary state $\sigma\in \dsetplus$, it is said to satisfy the \emph{sandwiched \renyi{} divergence (SRD) detailed balance condition} if the generator $\lbop$ is self-adjoint with respect to the inner product $\Inner{\cdot}{\cdot}_{\mathscr{W}_{\sigma, \alpha}}$ for \emph{any} $\alpha\in (0,\infty)$.
\end{defn}

\begin{proof}[Proof of \thmref{thm::lb_necessary}]

We first need to introduce some notations following \cite[Sec. 2.3]{carlen_non-commutative_2018}.
Define the operator $\mathscr{G}_{\rho}: \aset_0\rightarrow \aset_0$ by $\Inner{A}{\mathscr{G}_{\rho}(B)} = g_{\rho}(A, B)$; since $g_{\rho}$ is a $\Real$-valued inner product, $\mathscr{G}_{\rho}$ must be invertible and self-adjoint.
The inverse of $\mathscr{G}_{\rho}$ is denoted by $\mathscr{K}_{\rho}$, which is also self-adjoint. From \eqref{eqn::riemannian_grad}, one could easily derive that the gradient flow dynamics can be expressed by
\begin{align}
\label{eqn::rho_K_dD}
\dot{\rho}_t = -\mathscr{K}_{\rho_t} \left(\frac{\delta \renyidivg{\rho}{\sigma}}{\delta \rho}\eva_{\rho = \rho_t} \right).
\end{align}
For convenience, notice that the domain of $\mathscr{K}_{\rho}$ can be extended from $\aset_0$ to $\aset$ by setting $\mathscr{K}_{\rho}(\unit) := 0$.

 Under the assumption that $\dot{\rho}_t = \lbop^{\dagger}(\rho_t)$ is the gradient flow dynamics of the sandwiched \renyi{} divergence $\renyidivg{\rho}{\sigma}$ and by \eqref{eqn::rho_K_dD},
\begin{align}
\label{eqn::lbop_K_op}
\lbop^{\dagger}(\rho) = -\mathscr{K}_{\rho} \left(\frac{\delta \renyidivg{\rho}{\sigma}}{\delta \rho} \right).
\end{align}
Choose $\rho_{\eps} = \sigma + \eps A$ with $A\in \aset_0$ and let $\eps > 0$ small enough such that $\rho_{\eps} \in \dsetplus$.
Then
\begin{align*}
\lbop^{\dagger}(A) &= \frac{\lbop^{\dagger}(\rho_{\eps}) - \lbop^{\dagger}(\sigma)}{\eps} = \lim_{\eps\downarrow 0^{+}} \frac{\ud }{\ud \eps} \lbop^{\dagger}(\rho_{\eps})
= - \lim_{\eps\downarrow 0^{+}} \frac{\ud}{\ud \eps} \mathscr{K}_{\rho_{\eps}} \left(\frac{\delta \renyidivg{\rho}{\sigma}}{\delta \rho} \eva_{\rho = \rho_{\eps}}\right) \\
&\myeq{\eqref{eqn::fdrenyi}} - \frac{\alpha}{\alpha-1}\lim_{\eps\downarrow 0^{+}} \frac{\ud}{\ud \eps} \left( \left(\mathscr{K}_{\sigma} + \eps (\delta \mathscr{K}_{\sigma}) + \mathcal{O}(\eps^2) \right) \left( \frac{\Gamma_{\sigma}^{\frac{1-\alpha}{\alpha}}(\rho_{\sigma}^{\alpha-1})  }{\tr\left(\rho_{\sigma}^{\alpha}\right)}\eva_{\rho = \rho_{\eps}}\right)\right) \\
&= - \frac{\alpha}{\alpha-1}\lim_{\eps\downarrow 0^{+}} \frac{\ud}{\ud \eps} \left( \left(\mathscr{K}_{\sigma} + \eps (\delta \mathscr{K}_{\sigma}) + \mathcal{O}(\eps^2) \right) \left( \frac{ \unit + \eps B + \mathcal{O}(\eps^2) }{1 + \mathcal{O}(\eps)}\right)\right) \\
&= -\frac{\alpha}{\alpha-1} \mathscr{K}_{\sigma} (B),
\end{align*}
where the term $B$ in the expansion is (see \eg, \cite[Proposition 7.2]{carlen_analog_2014})
\begin{align*}
B &= \frac{\ud}{\ud \eps}\eva_{\eps=0} \Gamma_{\sigma}^{\frac{1-\alpha}{\alpha}} \left(\left(\Gamma_{\sigma}^{\frac{1-\alpha}{\alpha}}(\rho_{\eps})\right)^{\alpha-1}\right)\\
&=\Gamma_{\sigma}^{\frac{1-\alpha}{\alpha}}\left( \int_{0}^{1} \int_{0}^{\alpha-1} \frac{\sigma^{\frac{\alpha-1-\beta}{\alpha}}}{(1-s)\unit + s \sigma^{\frac{1}{\alpha}}} \Gamma_{\sigma}^{\frac{1-\alpha}{\alpha}}(A) \frac{\sigma^{\frac{\beta}{\alpha}}}{(1-s)\unit + s \sigma^{\frac{1}{\alpha}}} \ud \beta \ud s\right) \\
&= (\alpha-1) \int_{0}^{1} \int_{0}^{1} \frac{\bigl(\sigma^{\frac{\alpha-1}{\alpha}}\bigr)^{1-\beta}}{(1-s)\unit + s \sigma^{\frac{1}{\alpha}}} \Gamma_{\sigma}^{\frac{2(1-\alpha)}{\alpha}}(A) \frac{\bigl(\sigma^{\frac{\alpha-1}{\alpha}}\bigr)^{\beta}}{(1-s)\unit + s \sigma^{\frac{1}{\alpha}}} \ud \beta \ud s \\
&= (\alpha-1) \mop{\sigma^{\frac{\alpha-1}{\alpha}}}{} \circ \mop{\sigma^{\frac{1}{\alpha}}}{}^{-1} \circ \Gamma_{\sigma}^{\frac{2(1-\alpha)}{\alpha}} (A) = (\alpha-1) \mathscr{W}_{\sigma, \alpha}^{-1}(A).
\end{align*}
To get the third line, we change the variable $\beta$ by $(\alpha-1)\beta$; in the fourth line, we use expressions \eqref{eqn::kmb_op} and \eqref{eqn::inverse_mop}. Note that all three operators in the fourth line above pairwise commute. By combining the last two equations,
\begin{align}
\lbop^{\dagger}(A) = -\alpha \mathscr{K}_{\sigma} \left(\mathscr{W}_{\sigma, \alpha}^{-1}(A)\right).
\end{align}
Therefore, for any $C, D\in \aset_0$,
\begin{align*}
\Inner{\lbop(C)}{D}_{\mathscr{W}_{\sigma, \alpha}} = \Inner{C}{\lbop^{\dagger}\circ \mathscr{W}_{\sigma, \alpha} (D)} = -\alpha \Inner{C}{\mathscr{K}_{\sigma}(D)},
\end{align*}
and similarly,
\begin{align*}
\Inner{C}{\lbop(D)}_{\mathscr{W}_{\sigma, \alpha}} = (-\alpha) \Inner{\mathscr{K}_{\sigma}(C)}{D}.
\end{align*}

As mentioned earlier at the beginning of this proof, the operator $\mathscr{K}_{\sigma}$ is self-adjoint. Thus $\lbop$ is self-adjoint with respect to the inner product $\Inner{\cdot}{\cdot}_{\mathscr{W}_{\sigma, \alpha}}$.
\end{proof}

\subsection{Decomposition of the operator $\mathscr{W}_{\sigma, \alpha}$}
\label{sec::necessity_decomposition}

To understand better the operator $\mathscr{W}_{\sigma, \alpha}$, we would like to study its decomposition and various special cases.
Suppose the eigenvalue decomposition of $\sigma$ is $\sigma = \sum_{k} \lambda_k \Ket{\psi_k} \Bra{\psi_k}$ where $\lambda_k > 0$ for all $k$.
When $\alpha\neq 1$, the expansion of \eqref{eqn::weight_op_w} in terms of eigenbasis of $\sigma$ leads into
\begin{align}
\mathscr{W}_{\sigma, \alpha}(\cdot) = \sum_{k, j} f_{k, j}^{\alpha} \Ket{\psi_k} \Bra{\psi_k} \cdot \Ket{\psi_j} \Bra{\psi_j},
\end{align}
where the coefficients $f_{k, j}^{\alpha}$ are given by
\begin{equation}
\label{eqn::coef_f_in_W}
f_{k, j}^{\alpha} = \left\{\begin{split}
 \lambda_k,  & \qquad \lambda_k = \lambda_j; \\
 (\alpha-1)  \frac{\lambda_k^{\frac{1}{\alpha}} - \lambda_j^{\frac{1}{\alpha}}}{\lambda_j^{\frac{1-\alpha}{\alpha}} - \lambda_k^{\frac{1-\alpha}{\alpha}}}, & \qquad \lambda_k \neq \lambda_j. \\
\end{split}\right.
\end{equation}

Let us consider a few special weight operators $\mathscr{W}_{\sigma, \alpha}$:
\begin{itemize}
\item ($\alpha=1$). The operator $\mathscr{W}_{\sigma, \alpha}$ reduces to $\mop{\sigma}{}(\cdot) = \int_{0}^{1} \sigma^{1-s}\cdot \sigma^{s}\ud s$ and the inner product $\Inner{\cdot}{\cdot}_{\mathscr{W}_{\sigma, \alpha}}$ is known as the \emph{BKM inner product}. As a remark, this case has been shown in \cite[Theorem 2.9]{carlen_non-commutative_2018}. 

\item ($\alpha=2$). The operator $\mathscr{W}_{\sigma, \alpha}$ reduces to $\Gamma_{\sigma}$ and the inner product $\Inner{\cdot}{\cdot}_{\mathscr{W}_{\sigma, \alpha}}$ is KMS inner product. This immediately implies that in the context of the gradient flow of sandwiched \renyi{} divergences, one cannot work on a class of Lindblad equations larger than the one with KMS detailed balance condition.
\end{itemize}

Even though $\alpha\in (0,\infty)$ in the Definition \ref{defn::srd_dbc}, we can still apply the limiting argument to define $\mathscr{W}_{\sigma, \infty}$ and $\mathscr{W}_{\sigma, 0}$, so that we could better understand the SRD detailed balance condition.

\begin{itemize}

\item ($\alpha=\infty$). As $\alpha\rightarrow\infty$, $\sigma^{\frac{1}{\alpha}}\rightarrow \unit$ and $\sigma^{\frac{\alpha-1}{\alpha}}\rightarrow \sigma$. As a result, the operator $\mathscr{W}_{\sigma, \alpha}$ converges to $\mathscr{W}_{\sigma, \infty} := \mop{\sigma}{}^{-1} \circ \Gamma_{\sigma}^2$.

\item ($\alpha = 0$). This case is slightly more subtle, since $\sigma^{\frac{\alpha-1}{\alpha}}$ will blow up.
When $\lambda_k\neq \lambda_j$, as $\alpha\downarrow 0$, one could show that $f_{k, j}^{\alpha} \rightarrow \max(\lambda_k, \lambda_j)$. Therefore,
\begin{align*}
\mathscr{W}_{\sigma, 0}(\cdot) = \sum_{k, j} \max(\lambda_k, \lambda_j)  \Ket{\psi_k} \Bra{\psi_k}\cdot \Ket{\psi_j} \Bra{\psi_j}.
\end{align*}
\end{itemize}

\subsection{Relations between GNS, SRD, KMS detailed balance conditions}
\label{sec::necessity_relation}
\begin{prop}
\label{prop::gns_srd_kms}
For a primitive Lindblad equation with the unique stationary state $\sigma\in \dsetplus$,
\begin{enumerate}
\item GNS detailed balance condition implies SRD detailed balance condition.

\item SRD detailed balance condition implies KMS detailed balance condition, while conversely it is generally not true.
\end{enumerate}
\end{prop}

\begin{remark}
\begin{enumerate}
\item This suggests that the class of Lindblad equations with GNS detailed balance considered in \thmref{thm::renyi_gradient} seems to be general enough, because it is impossible to generalize \thmref{thm::renyi_gradient} to the class of Lindblad equations with KMS detailed balance by \propref{prop::gns_srd_kms}. More specifically, as we shall show later in the proof of \propref{prop::gns_srd_kms}, there exists at least one primitive Lindblad equation that satisfies KMS detailed balance condition, but does not satisfy SRD detailed balance condition (in particular, $\mathscr{W}_{\sigma, \alpha} \circ \lbop \circ \mathscr{W}_{\sigma, \alpha}^{-1} \neq \lbop^{\dagger}
$ for any $\alpha\neq 2$); because SRD detailed balance condition is necessary for having a gradient flow structure, this particular Lindblad equation cannot be expressed as the gradient flow dynamics of $\renyidivg{\rho}{\sigma}$ for any $\alpha \neq 2$.

\item It is still unknown whether the SRD detailed balance condition is equivalent to the GNS detailed balance condition or not. Thus it might be interesting to characterize Lindblad equations with SRD detailed balance, though this task seems to be rather technical.
\end{enumerate}
\end{remark}

\begin{proof}[Proof of \propref{prop::gns_srd_kms}]
The first statement comes immediately by combining \thmref{thm::renyi_gradient} and \thmref{thm::lb_necessary}.
As for the second statement, since the SRD detailed balance condition contains KMS detailed balance condition as a special instance when $\alpha=2$, obviously SRD implies KMS. In the following, we provide a primitive two-level Lindblad equation with KMS detailed balance from \cite[Appendix B]{carlen_gradient_2017} to demonstrate that KMS does not imply SRD.

Let us consider $K_1 = \Ket{\psi} \Bra{0}$ and $K_2 = \Ket{\phi} \Bra{1}$ where $\Ket{\psi}= \frac{1}{\sqrt{2}} \left(\Ket{0} + \Ket{1}\right)$ and $\Ket{\phi} = \frac{1}{\sqrt{5}} \left(\Ket{0} + 2\Ket{1}\right)$.

\begin{itemize}
\item Define an operator $\mathscr{K}$ by $\mathscr{K}(A) := K_1^* A K_1 + K_2^* A K_2$. Thus $\mathscr{K}$ is CP and $\mathscr{K}(\unit_2) = \unit_2$.

\item The adjoint operator
\begin{align*}
\mathscr{K}^{\dagger}(A) = K_1 A K_1^* + K_2 A K_2^* = \Ket{\psi}\Bra{0} A \Ket{0} \Bra{\psi} + \Ket{\phi}\Bra{1} A \Ket{1}\Bra{\phi}
\end{align*}
is also CP. Note that $\mathscr{K}^{\dagger} (\Ket{0} \Bra{1}) = \mathscr{K}^{\dagger} (\Ket{1} \Bra{0}) = 0$ and $\mathscr{K}^{\dagger}(\Ket{0}\Bra{0}) = \Ket{\psi}\Bra{\psi}$ and $\mathscr{K}^{\dagger}(\Ket{1}\Bra{1}) = \Ket{\phi} \Bra{\phi}$.
It can be readily verified that there is one and only one eigenstate for $\mathscr{K}^{\dagger}$ associated with eigenvalue $1$, which is
\begin{align*}
\sigma = \frac{1}{7} \begin{bmatrix} 2 & 3 \\ 3 & 5 \end{bmatrix}.
\end{align*}
The other non-zero eigenvalue of $\mathscr{K}^{\dagger}$ is $3/10$.

\item Define $\wt{K}_j = \sigma^{1/2} K_j^* \sigma^{-1/2}$ for $j=1, 2$ and define a CP operator $\wt{\mathscr{K}}$ via $\wt{\mathscr{K}}(A) = \wt{K}^*_1 A \wt{K}_1 + \wt{K}^*_2 A \wt{K}_2$. Then $\wt{\mathscr{K}} (\unit_2) = \unit_2$ and $\wt{\mathscr{K}}^{\dagger}(\sigma) = \sigma$; furthermore, we can verify that
\begin{align*}
\Inner{\wt{\mathscr{K}} B}{A}_{1/2} = \Inner{B}{\mathscr{K} A}_{1/2}.
\end{align*}
Hence,
the QMS $\qms_t = e^{t \lbop}$ with $\lbop := \wt{\mathscr{K}} \mathscr{K} - \id_2$ can be readily verified to satisfy the KMS detailed balance condition. Moreover, it is easy to show $\lbop^{\dagger}(\sigma) = 0$ and one could verify that it is the only eigenvector of $\lbop^{\dagger}$ with eigenvalue $0$, \ie, the Lindblad equation is primitive.

\end{itemize}

Next, we numerically show that \eqref{eqn::general_db_cond_2} does not hold. \figref{fig::wopLwopinv_Ldag} plots the trace-norm of $\mathscr{W}_{\sigma, \alpha} \circ \lbop \circ \mathscr{W}_{\sigma, \alpha}^{-1} -\lbop^{\dagger}$ for various $\alpha$;
from this figure, it is clear that \eqref{eqn::general_db_cond_2} holds only when $\alpha = 2$ (\ie, KMS detailed balance condition).
\end{proof}

\begin{figure}[h!]
\includegraphics[width=0.6\textwidth]{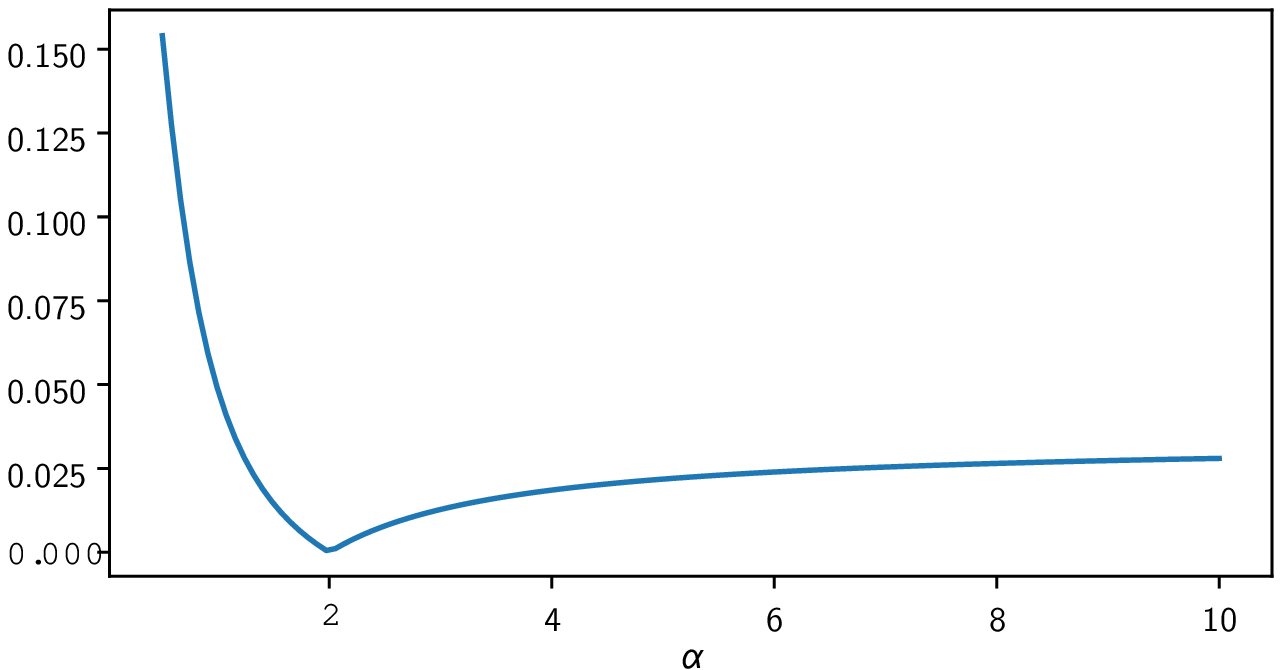}
\caption{$\Norm{ \mathscr{W}_{\sigma, \alpha} \circ \lbop \circ \mathscr{W}_{\sigma, \alpha}^{-1} -\lbop^{\dagger}}_{\text{Tr}}$ with respect to various $\alpha$}
\label{fig::wopLwopinv_Ldag}
\end{figure}

\section{Exponential decay of sandwiched \renyi{} divergences}
\label{sec::exp_decay}

This section is devoted to proving
\thmref{thm::decay_quantum_renyi}, \ie, the exponential
decay of sandwiched \renyi{} divergences for primitive Lindblad equations with GNS-detailed balance.
We start by recalling the definition of a spectral gap and then we prove a \poin{} inequality
(see \propref{prop::q_poin_2}).  With this \poin{} inequality, we
 derive a uniform lower bound of the quantum relative $2$-Fisher information, which
 immediately implies
\thmref{thm::decay_quantum_renyi} in the case $\alpha=2$.
Next we prove the \propref{prop::primitive_lsi}, which shows that for primitive \lb{} equations with GNS-detailed balance, the quantum LSI \eqref{eqn::lsi_alpha} holds for $\alpha = 1$ (\ie, there exists $K > 0$).
Then we prove a quantum comparison theorem
(see \propref{prop::comparison}), which implies that the exponential
decay rates of sandwiched \renyi{} divergences along the Lindblad equation are the same for all $\alpha > 1$. This together with the
monotonicity of sandwiched \renyi{} divergences with respect to the order
$\alpha$ concludes the proof of \thmref{thm::decay_quantum_renyi}.

\subsection{Spectral gap and \poin{} inequality}

Let us recall from \secref{sec::preliminaries} that if the \lb{} equation
satisfies GNS-detailed balance, then it also satisfies KMS-detailed balance
(\ie, $\lbop$ is self-adjoint with respect to the inner product
$\Inner{\cdot}{\cdot}_{1/2}$).
Moreover, by \lemref{lem::lb_quadratic}, $-\lbop$ is a positive semi-definite operator with respect to the inner product $\Inner{\cdot}{\cdot}_{1/2}$.
By the additional assumption that the \lb{} equation is primitive, we know $\unit$ is the only eigenvector of $-\lbop$ with respect to the eigenvalue zero. Hence, spectral theory shows that
there exists an orthonormal basis $\{L_1, L_2, \cdots, L_{n^2-1}, \unit\}$
such that
\begin{align}
\label{eqn::lbop_eig}
-\lbop (L_j) = \theta_j L_j, \qquad \theta_{1}\ge \theta_{2}\ge \cdots \ge \theta_{n^2-1} > 0;\qquad -\lbop(\unit) = 0.
\end{align}
The spectral gap of the operator $-\mathcal{L}$ is $\lbopgap := \theta_{n^2-1}$. It is worth
remarking that $\{L_1, L_2, \cdots, L_{n^2-1}, \unit\}$ is an orthonormal basis in the
Hilbert space $\mset$ equipped with the inner product $\Inner{\cdot}{\cdot}_{1/2}$, but not necessarily with other inner products.
The following \poin{} inequality follows directly from the definition of the spectral gap.

\begin{prop}[\poin{} inequality]
	\label{prop::q_poin_2}
	Assume that the primitive \lb{} equation satisfies GNS-detailed balance \eqref{eqn::lb}.
	For any $A\in \mset$ such that $\Inner{\unit}{A}_{1/2} = 0$ (or equivalently, $tr(\sigma A) = 0$),
	\begin{equation}
	\label{eqn::poin_2}
	\Inner{\nabla A}{\nabla A}_{1/2} = \Inner{A}{-\lbop (A)}_{1/2} \ge \lbopgap \Inner{A}{A}_{1/2}.
	\end{equation}
	Moreover, the equality can be achieved when $A = L_{n^2-1}$.
\end{prop}
\begin{proof}
		For any $A\in \mset$ such that $\Inner{\unit}{A}_{1/2} = 0$, we know that $A\in \text{span}\{L_j\}_{j=1}^{n^2-1}$. Thus by the definition of the spectral gap,
		$
		\Inner{A}{-\lbop (A)}_{1/2} \ge \lbopgap \Inner{A}{A}_{1/2}.
		$
		 Moreover, when $A = V_{n^2-1}$, the equality holds.
		Then \eqref{eqn::poin_2} follows immediately from \eqref{eqn::lb_quadratic}.
\end{proof}

\begin{remark}
In fact, the above \poin{} inequality is a special case of a whole family of \poin{} inequalities (see  \cite[Eq. (51)]{rouze_concentration_2019}): given any function $\varphi: (0,\infty)\rightarrow (0,\infty)$, we could  consider
\begin{align*}
C_{\varphi} \Inner{A}{\wt{\varphi}(\Delta_{\sigma})(A)}_{1/2}  &\equiv C_{\varphi} \Inner{A}{\varphi(\Delta_{\sigma}) (A)}_{1} \\
&\le \Inner{A}{\varphi(\Delta_{\sigma})\circ (-\lbop) (A)}_{1}
\equiv \Inner{A}{\wt{\varphi}(\Delta_{\sigma}) \circ (-\lbop) (A)}_{1/2},
\end{align*}
for all $A\in \mset$ such that $\tr(\sigma A) = 0$, where $\wt{\varphi}(x) := x^{-1/2}\varphi(x)$.
The inequality in \eqref{eqn::poin_2} simply refers to the case $\varphi(x) = x^{1/2}$.
Because $\Comm{-\lbop}{\Delta_{\sigma}} = 0$ (see \lemref{lem::commute_L_modular}), $-\lbop$ and $\Delta_{\sigma}$ are simultaneously diagonalizable; \revise{thus we might as well choose $L_j$ \eqref{eqn::lbop_eig} in a way that all $L_j$ are eigenvectors of both $-\lbop$ and $\Delta_{\sigma}$}.
Also note that both $-\lbop$ and $\Delta_{\sigma}$ are positive semi-definite operators, so they have non-negative spectrum.
By expanding $A =\sum_{j=1}^{n^2-1} a_j L_j$, it is not hard to show that $C_{\varphi} = \lbopgap$ in our situation. Therefore, all such \poin{} inequalities have the same prefactor $C_{\varphi}$ for primitive \lb{} equations with GNS-detailed balance.
\end{remark}

\subsection{Proof of \thmref{thm::decay_quantum_renyi} for the case $\alpha=2$}
\label{sec::alpha=2}

\begin{prop}[Lower bound of $\relafisheralpha{\rho}{\sigma}{2}$]
	\label{prop::fisher_bound_2}
	Assume that the primitive \lb{} equation \eqref{eqn::lb} satisfies GNS-detailed balance.
	Then
	\begin{equation}
	\label{eqn::fisher_bound_2}
	\relafisheralpha{\rho}{\sigma}{2} \ge  2\lbopgap \left(1 - e^{-\renyidivgalpha{\rho}{\sigma}{2}} \right),\ \forall \rho \in \dset.
	\end{equation}
\end{prop}

\begin{proof}
  The case $\alpha = 2$ is easier to deal with as certain quantities
  become more explicit:
	\begin{align*}
	\fdrenyidivgalpha{\rho}{\sigma}{2} &=  \frac{2\sigma^{-1/2} \rho \sigma^{-1/2}}{\Inner{\rho}{\sigma^{-1/2}\rho \sigma^{-1/2}}} = \frac{2 \Gamma_{\sigma}^{-1}(\rho) }{\Inner{\Gamma_{\sigma}^{-1}(\rho)}{\Gamma_{\sigma}^{-1}(\rho)}_{1/2}}; \\
	\renyidivgalpha{\rho}{\sigma}{2} &= \log \tr\left(\sigma^{-1/2}\rho \sigma^{-1/2} \rho\right) = \log\left(\Inner{\Gamma_{\sigma}^{-1}(\rho)}{\Gamma_{\sigma}^{-1}(\rho)}_{1/2}\right).
	\end{align*}
	Then, by the definition of the quantum relative $\alpha$-Fisher information \eqref{eqn::q_relative_fisher}
	\begin{align*}
	\relafisheralpha{\rho}{\sigma}{2}
	&= -\Inner{\frac{2 \Gamma_{\sigma}^{-1}(\rho) }{\Inner{\Gamma_{\sigma}^{-1}(\rho)}{\Gamma_{\sigma}^{-1}(\rho)}_{1/2}}}{\lbop^{\dagger}(\rho)} \\
	&= \frac{-2}{\Inner{\Gamma_{\sigma}^{-1}(\rho)}{\Gamma_{\sigma}^{-1}(\rho)}_{1/2}} \Inner{\Gamma_{\sigma}^{-1}(\rho)}{\Gamma_{\sigma}^{-1}\circ \lbop^{\dagger}\circ \Gamma_{\sigma} (\Gamma_{\sigma}^{-1}(\rho))}_{1/2} \\
	&\myeq{\eqref{eqn::lbop_lbop_dag}} \frac{-2}{\Inner{\Gamma_{\sigma}^{-1}(\rho)}{\Gamma_{\sigma}^{-1}(\rho)}_{1/2}} \Inner{ \Gamma_{\sigma}^{-1}(\rho)}{\lbop (\Gamma_{\sigma}^{-1}(\rho))}_{1/2} \\
	&= \frac{-2}{\Inner{\Gamma_{\sigma}^{-1}(\rho)}{\Gamma_{\sigma}^{-1}(\rho)}_{1/2}} \Inner{ \Gamma_{\sigma}^{-1}(\rho-\sigma)}{\lbop (\Gamma_{\sigma}^{-1}(\rho-\sigma))}_{1/2} \\
	& \myge{\eqref{eqn::poin_2}} \frac{2}{\Inner{\Gamma_{\sigma}^{-1}(\rho)}{\Gamma_{\sigma}^{-1}(\rho)}_{1/2}} \lbopgap \Inner{\Gamma_{\sigma}^{-1}(\rho-\sigma)}{\Gamma_{\sigma}^{-1}(\rho-\sigma)}_{1/2}.
	\end{align*}
	It is not hard to verify that
	\begin{align}
	\label{eqn::chi_2_with_one_half}
	\Inner{\Gamma_{\sigma}^{-1}(\rho)}{\Gamma_{\sigma}^{-1}(\rho)}_{1/2} = 1 + \Inner{\Gamma_{\sigma}^{-1} (\rho-\sigma)}{\Gamma_{\sigma}^{-1} (\rho-\sigma)}_{1/2}.
	\end{align}
	It is interesting to note that the term
    $\Inner{\Gamma_{\sigma}^{-1} (\rho-\sigma)}{\Gamma_{\sigma}^{-1}
      (\rho-\sigma)}_{1/2}$
    turns out to be the quantum $\chi^2$-divergence
    $\chi^2_{1/2}(\rho,\sigma)$ studied in
    \cite{temme_2-divergence_2010}.  Then
	\begin{align*}
		\relafisheralpha{\rho}{\sigma}{2}
		& \ge \frac{2 \lbopgap}{\Inner{\Gamma_{\sigma}^{-1}(\rho)}{\Gamma_{\sigma}^{-1}(\rho)}_{1/2}} \left(\Inner{\Gamma_{\sigma}^{-1}(\rho)}{\Gamma_{\sigma}^{-1}(\rho)}_{1/2} - 1\right) \\
		&= \frac{2\lbopgap}{e^{\renyidivgalpha{\rho}{\sigma}{2}}} \left(e^{\renyidivgalpha{\rho}{\sigma}{2}} - 1\right),
	\end{align*}
	which yields \eqref{eqn::fisher_bound_2}.
\end{proof}

With \propref{prop::fisher_bound_2}, we can immediately show the exponential decay of the sandwiched \renyi{} divergence of order $\alpha=2$.
\begin{proof}[Proof of \thmref{thm::decay_quantum_renyi} for the case $\alpha=2$]
	From the definition
	of the quantum relative $2$-Fisher information in
	\eqref{eqn::q_relative_fisher} and by \propref{prop::fisher_bound_2},
	\begin{equation*}
	\frac{\ud}{\ud t} \renyidivgalpha{\rho_t}{\sigma}{2} = -\relafisheralpha{\rho_t}{\sigma}{2} \le -2\lbopgap \left(1 - e^{-\renyidivgalpha{\rho_t}{\sigma}{2}}\right).
	\end{equation*}
	Then
	$ \frac{\ud}{\ud t} \log\left(e^{\renyidivgalpha{\rho_t}{\sigma}{2}} -
	1\right) \le -2 \lbopgap$.
	After integrating it from time $0$ to $t$,
	one could find after
	some straightforward simplification that
	\begin{equation*}
	\renyidivgalpha{\rho_t}{\sigma}{2} \le \log\left(1 + (e^{\renyidivgalpha{\rho_0}{\sigma}{2}} - 1) e^{-2\lbopgap t}\right) \le (e^{\renyidivgalpha{\rho_0}{\sigma}{2}} - 1) e^{-2\lbopgap t}.
	\end{equation*}
	Thus $\renyidivgalpha{\rho_t}{\sigma}{2}$ decays exponentially fast with rate $2\lbopgap$. Apparently, the prefactor $C_{2,\eps}$ and the waiting time $\waittimealpha{2}$ could be chosen as
	\begin{equation*}
	C_{2,\eps} = \frac{1}{\renyidivgalpha{\rho_0}{\sigma}{2}}\bigl( e^{\renyidivgalpha{\rho_0}{\sigma}{2}} - 1 \bigr),\qquad \waittimealpha{2} = 0.
	\end{equation*}
\end{proof}

The lower bound of $\relafisheralpha{\rho}{\sigma}{2}$ in \propref{prop::fisher_bound_2} immediately leads into a lower bound of the quantum $2$-log Sobolev constant $K_2$ in \eqref{eqn::lsi_alpha}.
\begin{coro}[{\cite[Theorem 4.2]{muller-hermes_sandwiched_2018}}]
	\label{coro::K_2_bound}
	For primitive \lb{} equations with GNS-detailed balance,
	the quantum $2$-log Sobolev constant $K_2$ is bounded below by the spectral gap; specifically,
	\begin{equation}
	\label{eqn::K_2_specgap}
		K_2 \ge \frac{1 - \specmin{\sigma}}{\log\left(\specmin{\sigma}^{-1}\right)} \lbopgap.
	\end{equation}
\end{coro}

\begin{proof}
	From \propref{prop::fisher_bound_2}, we know that
	\begin{equation*}
	\frac{\relafisheralpha{\rho}{\sigma}{2}}{\renyidivgalpha{\rho}{\sigma}{2}} \ge 2\lbopgap \frac{1-e^{-\renyidivgalpha{\rho}{\sigma}{2}}}{\renyidivgalpha{\rho}{\sigma}{2}} \ge 2\lbopgap \frac{1 - \specmin{\sigma}}{\log(\specmin{\sigma}^{-1})},
	\end{equation*}
	because $x\rightarrow \frac{1-e^{-x}}{x}$ is monotonically decreasing on $(0,\infty)$ and $0\le \renyidivgalpha{\rho}{\sigma}{2} \le \log(\specmin{\sigma}^{-1})$ \cite[Lemma 2.1]{muller-hermes_sandwiched_2018}. The above result follows immediately by taking the minimization over $\rho$ for both sides.
\end{proof}

\subsection{Proof of the quantum LSI}
\label{sec::proof_qlsi}
In this subsection, we prove the \propref{prop::primitive_lsi}, which shows the equivalence of the primitivity and the validity of quantum LSI \eqref{eqn::lsi_alpha}, for \lb{} equations with GNS-detailed balance.

\revise{One key ingredient is the $l_2$ mixing time defined as the following for any $\eps > 0$,
\begin{align}
\label{eqn::t2_time}
\begin{split}
t_2(\eps) &= \inf \{t \ge 0: \Norm{e^{t\lbop} (A) - \tr(\sigma A) \unit}_{2,\sigma} \le \eps,\ \forall A \in \mset \text{ s.t.} \Norm{A}_{1,\sigma} = 1\}\\
&\equiv \inf\{t\ge 0: \Norm{e^{t\lbop} (\cdot) - \tr(\sigma \cdot) \unit}_{1\rightarrow2, \sigma} \le \eps \}.
\end{split}
\end{align}
We shall need the following two results.
\begin{lemma}
\label{lem::t2_time_spec}
Under the same assumption as in \propref{prop::primitive_lsi}, if the Lindblad equation is primitive with the spectral gap $\lbopgap$, then
\begin{align}
\label{eqn::t2_time_spec}
t_2(\eps) \le \max\left(0, \frac{1}{2\lbopgap} \log\left(\frac{1}{\specmin{\sigma} \eps^2}\right)\right).
\end{align}
\end{lemma}
Its proof will be given slightly later in this subsection.
As a remark, when $\eps \ge \sqrt{\frac{1}{\specmin{\sigma}}}$, we have $t_2(\eps) = 0$; this is intuitively reasonable since when $\eps$ is large enough, the inequality in \eqref{eqn::t2_time} trivially holds.

\begin{theorem}[{\cite[Theorem 5.3]{muller-hermes_sandwiched_2018}}]
Under the same assumption as in \propref{prop::primitive_lsi}, if the Lindblad equation is primitive, then
	\begin{align}
	\label{eqn::t_2_gamma_2}
		t_2(e^{-1}) \kappa_2 \ge \frac{1}{2}.
	\end{align}
\end{theorem}

\begin{proof}[Proof of \propref{prop::primitive_lsi}]
	First, recall that \cite[Theorem 16]{kastoryano_quantum_2013} has proved that $K \le \lbopgap $ by linearizing the quantum LSI, thus $K > 0$ implies $\lbopgap > 0$.

	As for the other direction, assume that the \lb{} equation is primitive. Suppose $\rho_t = e^{t\lbop^{\dagger}}(\rho_0)$ is the solution of the \lb{} equation.
Combining the estimates above,  we have
\begin{align*}
K & \myeq{\eqref{eqn::kappa_1_K}} 2 \kappa_1 \myge{\eqref{eqn::qsvi_1_2}} 2\kappa_2
 \myge{\eqref{eqn::t_2_gamma_2}} \frac{1}{t_2(e^{-1})} \myge{\eqref{eqn::t2_time_spec}} \frac{2\lbopgap}{2-\log(\specmin{\sigma})} = \frac{1}{1 - \log(\sqrt{\specmin{\sigma}})} \lbopgap.
\end{align*}
\end{proof}

\begin{remark}
	In the proof above, we use the mixing time $t_2(e^{-1})$ as a bridge to connect
	$\kappa_2$ and $\lbopgap$, and use the quantum Stroock-Varopoulos inequality to connect $K$ and $\kappa_2$. To the best of our knowledge, a direct proof of the quantum LSI \eqref{eqn::lsi_alpha} has not appeared in the literature. For classical systems, Bakry-{\'E}mery \cite{bakry1985diffusions}  condition is an important criterion  for classical LSI to hold for Fokker-Planck equations.  It is an  interesting  open question to see whether there is a quantum analog.
\end{remark}

\begin{proof}[Proof of \lemref{lem::t2_time_spec}]
	Let us introduce $B = e^{t\lbop} (A) - \tr(\sigma A) \unit$ and decompose $A = a_0 \unit + \sum_{j=1}^{n^2-1} a_j L_j$, where $L_j$ are eigenvectors of the operator $\lbop$; recall \eqref{eqn::lbop_eig} for notations. Then we easily know that $B = \sum_{j=1}^{n^2-1} a_j e^{-t \theta_j} L_j$ and that \begin{align*}
		\Norm{e^{t\lbop}(A) - \tr(\sigma A) \unit}_{2,\sigma}^2  &= \tr\left(\sigma^{\frac{1}{4}} B^* \sigma^{\frac{1}{4}} \sigma^{\frac{1}{4}} B \sigma^{\frac{1}{4}}\right) = \Inner{B}{B}_{1/2} \\
	&= \sum_{j=1}^{n^2-1} e^{-2t\theta_j} \abs{a_j}^2 \le e^{-2 t \lbopgap} \sum_{j=1}^{n^2-1} \abs{a_j}^2.
	\end{align*}

Let $\wt{A} := \Abs{\sigma^{\frac{1}{2}} A \sigma^{\frac{1}{2}}}$. From the condition that $\Norm{A}_{1,\sigma} = 1$ in \eqref{eqn::t2_time}, we have $\tr(\wt{A}) = 1$.
Then
\begin{align*}
1 &\ge \tr(\wt{A}^2) = \tr\left(\sigma^{\frac{1}{2}} A^* \sigma^{\frac{1}{2}} \sigma^{\frac{1}{2}} A \sigma^{\frac{1}{2}} \right) = \Inner{\sigma^{\frac{1}{4}} \left(\sigma^{\frac{1}{4}} A \sigma^{\frac{1}{4}}\right) \sigma^{\frac{1}{4}}}{\sigma^{\frac{1}{4}}  \left(\sigma^{\frac{1}{4}} A \sigma^{\frac{1}{4}}\right) \sigma^{\frac{1}{4}} } \\
& \ge \specmin{\sigma} \Inner{\sigma^{\frac{1}{4}} A \sigma^{\frac{1}{4}}}{\sigma^{\frac{1}{4}} A \sigma^{\frac{1}{4}}} = \specmin{\sigma} \Inner{A}{A}_{1/2} \ge \specmin{\sigma} \left(\sum_{j=1}^{n^2-1} \abs{a_j}^2\right).
\end{align*}
Therefore, we know that
\begin{align*}
\Norm{e^{t\lbop}(A) - \tr(\sigma A) \unit}_{2,\sigma}^2 \le e^{-2t\lbopgap} \frac{1}{\specmin{\sigma}}.
\end{align*}
Easily, we know that whenever $t \ge \frac{1}{2\lbopgap} \log\left(\frac{1}{\specmin{\sigma} \eps^2}\right)$, we have
$\Norm{e^{t\lbop}(A) - \tr(\sigma A) \unit}_{2,\sigma} \le \eps$. Then \eqref{eqn::t2_time_spec} follows immediately.
\end{proof}
}

\subsection{Proof of \thmref{thm::decay_quantum_renyi} for $\alpha\in (0,\infty)$ via a quantum comparison theorem}
\label{sec::comparison}
As a quantum analog of \cite[Lemma 3.4]{prevpaper} for \fp{}
equations and \cite[Theorem 3.2.3]{CIT-064} for the Ornstein-Uhlenbeck process, in \propref{prop::comparison}, we will show that under primitive \lb{}
equations with GNS-detailed balance \eqref{eqn::lb}, the sandwiched \renyi{} divergence of higher
order $\alpha_1$ can be bounded above by the sandwiched \renyi{} divergence of lower order $\alpha_0$ ($1 < \alpha_0< \alpha_1$), at the expense of some time delay. This result helps to prove the exponential
decay for sandwiched \renyi{} divergences under \lb{}
equations, summarized in \propref{prop::exp_decay_renyi}.

\begin{prop}[Quantum comparison theorem]
	\label{prop::comparison}
	Let $\rho_t$ be the solution of the \lb{} equation with GNS-detailed balance \eqref{eqn::lb}.
	Assume that the
	quantum LSI \eqref{eqn::lsi_alpha} holds with constant $K$.
	Assume also that
	\begin{equation}\label{eq:ent0}
	\relaentropy{\rho_0}{\sigma} \le \epsilon  \text{ for some fixed }\epsilon \in \left(0, \frac{\specmin{\sigma}^2}{2}\right).
	\end{equation}
	Then for any two orders $1 < \alpha_0 \le \alpha_1$,
	we have
	\begin{equation}\label{eq:compare}
	\renyidivgalpha{\rho_{T}}{\sigma}{\alpha_1} \le \renyidivgalpha{\rho_0}{\sigma}{\alpha_0},
	\end{equation}
	with
	\begin{equation}
	\label{eqn::delay_T}
	T = \frac{1}{2K \eta} \log\left(\frac{\alpha_1-1}{\alpha_0 - 1}\right),
	\end{equation}
	where
	\begin{align*}
	\eta &= \min\left(\frac{1}{2}, \min_{j=1}^{\cardn} \left\{\frac{2\sqrt{e^{\omega_j} \frac{1}{\Lambda}}}{1+e^{\omega_j} \Lambda}\right\}\right),\\
	\Lambda &= \frac{\specmax{\sigma}}{\specmin{\sigma}} \exp\left(\alpha_0 \sqrt{2\eps}\frac{2\specmin{\sigma} - \sqrt{2\eps}}{\specmin{\sigma}(\specmin{\sigma} - \sqrt{2\eps})}\right).
	\end{align*}
\end{prop}

The proof is postponed to \secref{sec::proof_comparison}.
Later in \secref{sec::diss_q_comparison}, we will discuss this quantum comparison theorem in details. As a reminder, we have, in fact, implicitly assumed that Lindblad equations under consideration are primitive, by \propref{prop::primitive_lsi}.
	\begin{remark}
\begin{enumerate}
	\item

	The assumption \eqref{eq:ent0}
	on the initial condition $\rho_0$ is merely a
	technical assumption allowing
	us to obtain a neat expression for the delay time $T$.
	This assumption \eqref{eq:ent0} is not essential for the inequality \eqref{eq:compare}
	to hold; one could remove the assumption at the expense of longer delay time. In fact,
	the validity of the quantum LSI \eqref{eqn::lsi_alpha} leads to the exponential
	decay of the quantum relative entropy, so that
	$D({\rho_{t_0} || \sigma} ) \leq \epsilon $ holds for some $t_0 \ge 0 $. Taking $\rho_{t_0}$ in place of $\rho_0$, one sees that this would imply \eqref{eq:compare} with a larger decay time $T$.

\item
	Another reason for imposing the assumption \eqref{eq:ent0} is to avoid technicalities. Thanks to the quantum Pinsker's inequality
	\begin{equation}
	\label{eqn::qpinsker}
	\relaentropy{\rho}{\sigma} \ge \frac{1}{2} \left(\tr\Abs{\rho-\sigma}\right)^2,
	\end{equation}
		the assumption that $\relaentropy{\rho_0}{\sigma} \le \eps$ with $0 < \eps < \frac{\specmin{\sigma}^2}{2}$ implies that $\rho_0$ is a strictly positive density matrix;
		in fact, we can even show that
		$\rho_t\in \dsetplus$, for any $t\ge 0$.
		This assumption can help avoid many technical issues arising from degenerate density matrices and it validates the usage of results in the previous two sections, \eg, properties of the operator $\mop{X}{\omega}$ (which require $X$ to be strictly positive).
\end{enumerate}
	\end{remark}

\begin{prop}
\label{prop::exp_decay_renyi}
Let $\rho_t$ be the solution of the \lb{} equation with GNS-detailed balance \eqref{eqn::lb}. Assume that the
	quantum LSI \eqref{eqn::lsi_alpha} holds with constant $K$.
	If the sandwiched \renyi{} divergence decays exponentially fast with rate $R$ for some order $\alpha_0 > 1$, then it decays exponentially fast for any $\alpha\in (0,\infty)$ with rate at least $R$.
\end{prop}

\begin{proof}[Proof of \propref{prop::exp_decay_renyi}]
	The exponential decay of the sandwiched \renyi{} divergence means there exists some time $t_0$ and constant $C_0$ such that
	\begin{equation*}
	\renyidivgalpha{\rho_t}{\sigma}{\alpha_0} \le C_0 \renyidivgalpha{\rho_0}{\sigma}{\alpha_0} e^{- R t},\qquad \forall t \ge t_0.
	\end{equation*}

	{\noindent  \emph{Case (\rom{1}):} $\alpha\le \alpha_0$}.
	By the monotonicity of sandwiched \renyi{} divergences with respect to the order $\alpha$ \cite[Theorem 7]{muller-lennert_quantum_2013},
	\begin{align*}
	\renyidivg{\rho_t}{\sigma} \le \renyidivgalpha{\rho_t}{\sigma}{\alpha_0} \le C_0 \renyidivgalpha{\rho_0}{\sigma}{\alpha_0} e^{- R t} = C_{\alpha} \renyidivg{\rho_0}{\sigma} e^{- R t},
	\end{align*}
	for all $t\ge t_0$, where $C_{\alpha} = \frac{C_0 \renyidivgalpha{\rho_0}{\sigma}{\alpha_0}}{\renyidivg{\rho_0}{\sigma}}$.

	{\noindent \emph{Case (\rom{2}):} $\alpha > \alpha_0$}. Pick an arbitrary $0<\epsilon <
	\frac{\specmin{\sigma}^2}{2}$.
	Notice that the quantum LSI implies the exponential decay of the quantum relative entropy,
	namely,
	\begin{equation*}
	\relaentropy{\rho_t}{\sigma} \le \relaentropy{\rho_0}{\sigma} e^{-2K t}.
	\end{equation*}
	Hence when $t\ge \frac{1}{2K} \log\left(\frac{\relaentropy{\rho_0}{\sigma}}{\epsilon}\right)$, one has $\relaentropy{\rho_t}{\sigma} \le \epsilon$.
	Moreover, if  $t\ge T + \max\left(t_0, \frac{1}{2K}
	\log\left(\frac{\relaentropy{\rho_0}{\sigma}}{\epsilon}\right)\right)$,
	where $T$ is the time delay given in \propref{prop::comparison} with the choice
	$\alpha_1 = \alpha$, then one obtains from  \propref{prop::comparison}  that
	\begin{align*}
	\renyidivg{\rho_t}{\sigma} \le \renyidivgalpha{\rho_{t-T}}{\sigma}{\alpha_0} \le C_0 \renyidivgalpha{\rho_0}{\sigma}{\alpha_0} e^{- R (t-T)} = C_{\alpha} \renyidivg{\rho_0}{\sigma} e^{-R t},
	\end{align*}
	where $C_{\alpha} = \frac{C_0 \renyidivgalpha{\rho_0}{\sigma}{\alpha_0} e^{RT}}{\renyidivg{\rho_0}{\sigma}}$.

	Combining both cases above finishes the proof of \propref{prop::exp_decay_renyi}.
\end{proof}

\begin{proof}[Proof of \thmref{thm::decay_quantum_renyi}]
	Set $\alpha_0 = 2$ and $R = 2\lbopgap$ in \propref{prop::exp_decay_renyi}.
	\thmref{thm::decay_quantum_renyi} follows immediately from \propref{prop::exp_decay_renyi} and the case $\alpha=2$ proved in \secref{sec::alpha=2};
	the validity of the quantum LSI has been shown in \propref{prop::primitive_lsi}.
	The expression of $T_{\alpha,\eps}$ follows immediately from \propref{prop::comparison}.
\end{proof}

\subsection{Proof of quantum comparison theorem}
\label{sec::proof_comparison}

We now turn to the proof of \propref{prop::comparison}.
Consider the density matrix
\begin{equation*}
\varrho = \frac{\left(\Gamma_{\sigma}^{\frac{1-\alpha}{\alpha}}(\rho)\right)^{\alpha}}{Z}  \equiv \frac{\rho_{\sigma}^{\alpha} }{Z},
\end{equation*} where the normalization constant
$Z = \tr\left(\rho_{\sigma}^{\alpha}\right)$ and $\rho\in \dsetplus$.
Then apparently, $\varrho\in \dsetplus$. The \lsi{} \eqref{eqn::lsi_alpha},
with $\rho$ in \eqref{eqn::lsi_alpha} chosen as $\varrho$ above, becomes
\begin{equation}
\label{eqn::lsi_variant}
\begin{split}
& \tr\left(\rho_{\sigma}^{\alpha} \log\bigl(\rho_{\sigma}^{\alpha}\bigr)\right) - Z \log(Z) - \tr\left(\rho_{\sigma}^{\alpha} \log(\sigma)\right) 
 \le
\frac{1}{2K} \tr\left(-\lbop^{\dagger}(\rho_{\sigma}^{\alpha}) \bigl(\log(\rho_{\sigma}^{\alpha}) - \log(\sigma)\bigr) \right). \\
\end{split}
\end{equation}
This inequality can be regarded as a variant of the \lsi{} and it will be used later.
Let us define the operator $G_{X,\omega}(s):\mset\rightarrow\mset$ by
	\begin{align}
	\label{eqn::g_symm_op}
	G_{X,\omega}(s) (A) := e^{\omega s} X^s A X^{-s} + e^{\omega (1-s)} X^{1-s} A X^{-(1-s)}.
	\end{align}
The lemma below collects some properties of $G_{X,\omega}(s)$ which will be useful in the proof of \propref{prop::comparison}.
\begin{lemma}
	\label{lem::Gop_comparison}
	Assume that $X\in \psetplus$ and $\omega\in \Real$. Then the operator $G_{X,\omega}(s)$ satisfies:
	\begin{enumerate}
		\item \label{lem::Gop_comparison_order}
		For $0\le s_1 < s_2 \le \frac{1}{2}$,
		\begin{align*}
		G_{X,\omega}(s_1 ) \ge G_{X,\omega}(s_2).
		\end{align*}
		Moreover, the equality is obtained iff $\omega = 0$ and $X = c\unit$ for some constant $c > 0$.

		\item  \label{lem::Gop_comparison_bound}
		For $s\in [0, \frac{1}{2}]$, a bound for the spectrum of $G_{X,\omega}(s)$ is
		\revise{\begin{align*}
		2  \sqrt{e^{\omega}\frac{\specmin{X}}{\specmax{X}}} \id  \le G_{X,\omega}(s) \le  \left(1+e^{\omega}\frac{\specmax{X}}{\specmin{X}}\right) \id.
		\end{align*}}
		As a consequence,
				\begin{align*}
		\eta G_{X,\omega}(s_1) \le G_{X,\omega}(s_2),\qquad \forall s_1, s_2\in \left[0,\frac{1}{2}\right]
		\end{align*}
		with
		\begin{equation*}
		\eta = \frac{2\sqrt{e^{\omega}\frac{\specmin{X}}{\specmax{X}}}}{1+e^{\omega} \frac{\specmax{X}}{\specmin{X}}}.
		\end{equation*}

	\end{enumerate}
\end{lemma}

\begin{proof}[Proof of \lemref{lem::Gop_comparison}]
\begin{enumerate}
	\item

	For any positive number $b$, notice that $b^s + b^{1-s}$ is a non-increasing convex function with respect to $s$ on the interval $ [0,\frac{1}{2}]$. Moreover,
	\begin{itemize}
		\item the range of the function $b^s + b^{1-s}$ on this interval is $[2\sqrt{b}, 1+b]$;
		\item $b^{s_1} + b^{1-s_1} = b^{s_2} + b^{1-s_2}$ for some $0\le s_1 < s_2 \le \frac{1}{2}$ iff $b = 1$.
	\end{itemize}

	Let us write the spectral decomposition of $X$ as $X = \sum_{k=1}^{n} \lambda_k \Ket{k}\Bra{k}$ where $\lambda_k > 0$ by assumption. Then for any matrix $A$,
	\begin{align*}
	\Inner{A}{ G_{X,\omega}(s) A} &=
	\sum_{k,l=1}^{n} \left(   \left(e^{\omega } \frac{\lambda_k}{\lambda_l}\right)^s +  \left(e^{\omega} \frac{\lambda_k}{\lambda_l}\right)^{1-s}\right) \Abs{A_{k,l}}^2.
	\end{align*}
	Since $b^s + b^{1-s}$ is non-increasing with respect to $s$, so is $G_{X,\omega}(s)$. To achieve equality, we need $e^{\omega}\frac{\lambda_k}{\lambda_l} = 1$ for all $k, l$. Hence $\omega = 0$ and there exists some $c = \lambda_k$ for all $1\le k \le n$ (\ie, $X = c\unit$).

	\item From the last equation, by using the range of $b^s + b^{1-s}$, we immediately have
	\revise{
	\begin{align*}
	& \Inner{A}{G_{X,\omega}(s)A} \le \sum_{k,l} (1+e^{\omega} \frac{\lambda_k}{\lambda_l}) \Abs{A_{k,l}}^2 \le  \left(1+e^{\omega}\frac{\specmax{X}}{\specmin{X}}\right) \Inner{A}{A}; \\
	& \Inner{A}{G_{X,\omega}(s) A} \ge  \sum_{k,l} 2\sqrt{e^{\omega}\frac{\lambda_k}{\lambda_l}} \Abs{A_{k,l}}^2 \ge
	2  \sqrt{e^{\omega}\frac{\specmin{X}}{\specmax{X}}} \Inner{A}{A},
	\end{align*}
	which proves part (2).}
\end{enumerate}
\end{proof}

\begin{proof}[Proof of \propref{prop::comparison}]

	Let us first introduce some useful notations.
	Suppose $\rho_t$ is evolving according to the \lb{} equation \eqref{eqn::lb} and $\beta_t = 1 + (\alpha_0 - 1) e^{\eta 2Kt}$ is also changing with respect to time for some positive constant $\eta$, independent of any specific initial condition $\rho_0$. The parameter $\eta\in (0,1]$ is yet to be determined later. We are interested in the time interval $[0,T]$ such that $\beta_t\vert_{t=0} = \alpha_0, \beta_t\vert_{t=T} = \alpha_1$. More specifically, $T = \frac{1}{2K \eta} \log\left(\frac{\alpha_1-1}{\alpha_0 - 1}\right)$.

	Let us define an important quantity $\rho_{\sigma, t}$ by
	\begin{equation}
	 \rho_{\sigma,t}: = \Gamma_{\sigma}^{\frac{1-\beta_t}{\beta_t}} (\rho_t).
	\end{equation}
	We also define
	\begin{equation*}
	Z_t := \tr\left(\bigl(\Gamma_{\sigma}^{\frac{1-\beta_t}{\beta_t}} (\rho_t)\bigr)^{\beta_t}\right) \equiv \tr\left(\rho_{\sigma,t}^{\beta_t}\right) \text{ and } F_t := \frac{1}{\beta_t} \log Z_t.
	\end{equation*}
	After taking the derivative of $F_t$ with respect to time $t$ and arranging terms,
	\begin{equation*}
	\beta_t^2 Z_t \dot{F}_t = \beta_t \dot{Z}_t - \dot{\beta}_t Z_t \log(Z_t).
	\end{equation*}
	We claim that to prove the inequality \eqref{eq:compare} it suffices to
	show that $F_t$ is monotonically decreasing, \ie, $\dot{F}_t \le 0$ on the interval $[0,T]$.
	 In fact, from the definition of  $F$ and the
	 sandwiched \renyi{} divergence \eqref{eqn::swch}, we immediately obtain from $F_T \leq F_0$  that
	\begin{equation}
	\renyidivgalpha{\rho_{T}}{\sigma}{\alpha_1} \le \frac{\alpha_1}{\alpha_1 - 1} \frac{\alpha_0-1}{\alpha_0} \renyidivgalpha{\rho_0}{\sigma}{\alpha_0} \le \renyidivgalpha{\rho_0}{\sigma}{\alpha_0}.
	\end{equation}
	The rest devotes to the proof of $\dot{F}_t \le 0$, which is done in the following steps.

	{\noindent {\bf Step (\rom{1})}: Simplification of  $\dot{F}_t$. }

	Let us first compute $\dot{Z}_t$.
	\begin{align*}
	\dot{Z}_t &\equiv \frac{\ud}{\ud t} \tr\left( \bigl(\sigma^{\frac{1-\beta_t}{2\beta_t}} \rho_t \sigma^{\frac{1-\beta_t}{2\beta_t}}\bigr)^{\beta_t} \right)\\
	&\begin{aligned}
	= \beta_t\tr\left( \rho_{\sigma,t}^{\beta_t-1} \left(
	\sigma^{\frac{1-\beta_t}{2\beta_t}}
	\lbop^{\dagger}(\rho_t)
	\sigma^{\frac{1-\beta_t}{2\beta_t}}
	-\frac{\dot{\beta}_t}{2\beta_t^2} \sigma^{\frac{1-\beta_t}{2\beta_t}} \Anticomm{\log(\sigma)}{\rho_t} \sigma^{\frac{1-\beta_t}{2\beta_t}}
	\right)\right)
	+ \dot{\beta}_t  \tr\left(\rho_{\sigma,t}^{\beta_t}\log(\rho_{\sigma,t})\right) \\
	\end{aligned}\\
	&\begin{aligned}
	= \beta_t\tr\left(\sigma^{\frac{1-\beta_t}{2\beta_t}} \rho_{\sigma,t}^{\beta_t-1}
	\sigma^{\frac{1-\beta_t}{2\beta_t}}
	\lbop^{\dagger}(\rho_t)\right)
	-\frac{\dot{\beta}_t}{\beta_t}\tr\left( \rho_{\sigma,t}^{\beta_t} \log(\sigma) \right)
	+ \frac{\dot{\beta}_t}{\beta_t} \tr\left(\rho_{\sigma,t}^{\beta_t}\log(\rho_{\sigma,t}^{\beta_t})\right), \\
	\end{aligned}
	\end{align*}
	where the anti-commutator $\Anticomm{\cdot}{\cdot}$ is defined as $\Anticomm{A}{B} := AB + BA$ for all matrices $A$ and $B$.

	Hence by \eqref{eqn::lsi_variant} and the fact that $\dot{\beta}_t > 0$,
	\begin{equation}
	\label{eqn::derivative_F_t_bound}
	\begin{split}
	&\beta_t^2 Z_t \dot{F}_t \\
	& \begin{aligned}=
	 \beta_t^2\tr\left(\sigma^{\frac{1-\beta_t}{2\beta_t}} \rho_{\sigma,t}^{\beta_t-1}
	\sigma^{\frac{1-\beta_t}{2\beta_t}}
	\lbop^{\dagger}(\rho_t)\right)
	- \dot{\beta}_t \tr\left( \rho_{\sigma,t}^{\beta_t} \log(\sigma) \right)
	 + \dot{\beta}_t \tr\left(\rho_{\sigma,t}^{\beta_t}\log(\rho_{\sigma,t}^{\beta_t})\right) -\dot{\beta}_t Z_t \log(Z_t)
	\end{aligned}\\
	&\begin{aligned}
	\le \beta_t^2\tr\left(\sigma^{\frac{1-\beta_t}{2\beta_t}} \rho_{\sigma,t}^{\beta_t-1}
	\sigma^{\frac{1-\beta_t}{2\beta_t}}
	\lbop^{\dagger}(\rho_t)\right)
	+ \frac{\dot{\beta}_t}{2K} \tr\left(-\lbop^{\dagger}(\rho_{\sigma,t}^{\beta_t}) \bigl(\log(\rho_{\sigma,t}^{\beta_t}) - \log(\sigma)\bigr) \right)
	\end{aligned} \\
	& =: T_1 + \frac{\dot{\beta}_t}{2K} T_2. \\
	\end{split}
	\end{equation}
	In the last line, we introduce $T_1$ and $T_2$ as short-hand notations. Next, we further simplify $T_1$ and $T_2$.

	For $T_1$,
	\begin{equation*}
	\begin{split}
	T_1 &=\beta_t^2\tr\left(\sigma^{\frac{1-\beta_t}{2\beta_t}} \rho_{\sigma,t}^{\beta_t-1}
	\sigma^{\frac{1-\beta_t}{2\beta_t}}
	\lbop^{\dagger}(\rho_t)\right) \\
	&\myeq{\eqref{eqn::fdrenyi}} \beta_t(\beta_t-1) Z_t \tr\left(\fdrenyidivgalpha{\rho}{\sigma}{\beta_t}\eva_{\rho=\rho_t} \lbop^{\dagger}
	(\rho_t)\right)\\
	&\myeq{\eqref{eqn::lb_renyi_var}} -\beta_t (\beta_t-1)Z_t \Inner{\nabla \fdrenyidivgalpha{\rho}{\sigma}{\beta_t}\eva_{\rho=\rho_t}}{\moprenyialpha{\rho_t}{\vect{\omega}}{\beta_t} \left(\nabla \fdrenyidivgalpha{\rho}{\sigma}{\beta_t}\eva_{\rho=\rho_t}\right) } \\
	&\myeq{\eqref{eqn::moprenyi}} -(\beta_t-1) Z_t^2 \sum_{j} \Inner{A_j}{ \mop{\rho_{\sigma,t}^{\beta_t-1}}{\frac{(\beta_t-1)\omega_j}{\beta_t}}\circ \mop{\rho_{\sigma,t}}{\frac{\omega_j}{\beta_t}} (A_j) },
	\end{split}
	\end{equation*}
	where $A_j = \mop{\rho_{\sigma,t}^{\beta_t-1}}{\frac{(\beta_t-1)\omega_j}{\beta_t}}^{-1} \circ \Gamma_{\sigma}^{\frac{\beta_t-1}{\beta_t}} \left(\partial_j \fdrenyidivgalpha{\rho}{\sigma}{\beta_t}\eva_{\rho=\rho_t}\right)$.

	As for $T_2$, by using \eqref{eqn::lb_renyi_var} again (for the case $\alpha=1$ in \eqref{eqn::lb_renyi_var}),
	\begin{equation*}
	\begin{split}
	T_2
	&= -\Inner{\log(\rho_{\sigma,t}^{\beta_t}) - \log(\sigma)}{\lbop^{\dagger}(\rho_{\sigma,t}^{\beta_t})} \\
	&= \Inner
	{\nabla \bigl(\beta_t \log(\rho_{\sigma,t}) - \log(\sigma) \bigr)}
	{
		\mop{\rho_{\sigma,t}^{\beta_t}}{\vect{\omega}} \nabla \bigl({\beta_t} \log(\rho_{\sigma,t}) - \log(\sigma)\bigr)
	}.
	\end{split}
	\end{equation*}
	Consider the term
	\begin{equation*}
	\begin{split}
	&\partial_j \bigl(\beta_t \log(\rho_{\sigma,t}) - \log(\sigma) \bigr) \\
	=& \beta_t \left(\partial_j \log(\rho_{\sigma,t}) - \frac{\omega_j}{\beta_t} V_j \right) \\
	=& \beta_t \left(V_j \log\left(e^{-\frac{\omega_j}{2\beta_t}} \rho_{\sigma,t}\right) - \log\left(e^{\frac{\omega_j}{2\beta_t} } \rho_{\sigma,t}\right) V_j \right) \\
	\myeq{\eqref{eqn::pdfdrenyi_part}}& Z_t \left( \mop{\rho_{\sigma,t}^{\beta_t-1}}{\frac{\beta_t-1}{\beta_t} \omega_j}^{-1} \circ \Gamma_{\sigma}^{\frac{\beta_t-1}{\beta_t}}\right) \partial_j \left(\fdrenyidivgalpha{\rho}{\sigma}{\beta_t}\eva_{\rho=\rho_t}\right) \\
	=& Z_t A_j. \\
	\end{split}
	\end{equation*}
	Here, to obtain the second line, we have used $\partial_j \log(\sigma) \equiv \Comm{V_j}{\log(\sigma)} = \partial_s\left(\Delta_{\sigma}^{-s} V_j\right)\vert_{s=0} = \partial_s\left(e^{\omega_j s} V_j\right)\vert_{s=0} = \omega_j V_j$ (see \cite[Lemma 5.9]{carlen_gradient_2017}).
	Then, it follows immediately that
	\begin{align*}
	T_2 = 	Z_t^2 \sum_{j} \Inner{A_j}{\mop{\rho_{\sigma,t}^{\beta_t}}{\omega_j} (A_j)}.
	\end{align*}

	\revise{By plugging the expressions of $T_1$ and $T_2$ back into \eqref{eqn::derivative_F_t_bound}, and using the fact that $\frac{\dot{\beta_t}}{2K} = \frac{\eta 2K (\alpha_0 - 1)e^{\eta 2K t}}{2K} = \eta (\beta_t - 1)$, we have
	\begin{align*}
	\beta_t^2 Z_t \dot{F}_t \le
	& (\beta_t - 1) Z_t^2
	\sum_{j} \left(
	\eta \Inner{A_j}{\mop{\rho_{\sigma,t}^{\beta_t}}{\omega_j} (A_j)}
	- \Inner{A_j}{ \mop{\rho_{\sigma,t}^{\beta_t-1}}{\frac{(\beta_t-1)\omega_j}{\beta_t}}\circ \mop{\rho_{\sigma,t}}{\frac{\omega_j}{\beta_t}} (A_j) }
	\right).
	\end{align*}}

	{\noindent {\bf Step (\rom{2})}: We show that $\dot{F}_t \le 0$. }

	By observing the last equation, since $\beta_t - 1>0$ and $Z_t > 0$, it is then sufficient to show that
	each term indexed by $j$ on the right hand side of the last equation is less than or equal to zero, namely,
	\begin{equation}
	\eta \Inner{A_j}{\mop{\rho_{\sigma,t}^{\beta_t}}{\omega_j} (A_j)} \le \Inner{A_j}{\mop{\rho_{\sigma,t}^{\beta_t-1}}{\frac{(\beta_t-1)\omega_j}{\beta_t}}\circ \mop{\rho_{\sigma,t}}{\frac{\omega_j}{\beta_t}} (A_j) },
	\end{equation}
	for any time $t\in [0,T]$ and index $j$.
	In the commutative case, both left and right hand sides represent multiplication of the operator $\rho_{\sigma,t}^{\beta_t}$ and the above inequality holds trivially.
	For noncommutative (quantum) systems, the above inequality is non-trivial. In this step, we will simplify the expression in the last equation and furthermore find sufficient conditions for $\eta$ so that the last inequality holds.

	By definition \eqref{eqn::kmb_op} and let $B_j = A_j \rho_{\sigma,t}^{\frac{\beta_t}{2}}$, the above inequality can be simplified to
	\begin{align}
	\label{eqn::equiv_hyper_eta}
	\begin{split}
	&\eta \Inner{B_j}{\int_{0}^{1} e^{\omega_j s} \rho_{\sigma,t}^{\beta_t s} B_j \rho_{\sigma,t}^{-\beta_t s}\ud s} \\
	\le & \Inner{B_j}{\int_{0}^{1}\int_{0}^{1} e^{\frac{(\beta_t-1)\omega_j v}{\beta_t}}e^{\frac{\omega_j u}{\beta_t}} \rho_{\sigma,t}^{(\beta_t-1)v + u} B_j \rho_{\sigma,t}^{-(\beta_t-1)v - u}\ud u\ud v} \\
	= & \Inner{B_j}{\int_{0}^{1} e^{\omega_j s} \rho_{\sigma,t}^{\beta_t s} B_j \rho_{\sigma,t}^{-\beta_t s} \left(\int_{U_s} \Abs{\frac{\partial(u,v)}{\partial(s,\wt{s})}}\ud \wt{s}\right) \ud s}\\
	= & \Inner{B_j}{\int_{0}^{1} e^{\omega_j s} \rho_{\sigma,t}^{\beta_t s} B_j \rho_{\sigma,t}^{-\beta_t s} f(s)\ud s}.
	\end{split}
	\end{align}
	From the second line to the third line, we have used the change of  variable $s = \frac{(\beta_t-1)v+u}{\beta_t}$  and $\wt{s} = \frac{(\beta_t-1)v-u}{\beta_t}$. The set $U_s$ is defined as
	$$
	U_s: = \Big\{\tilde{s}\in \Real: 0 \le \frac{\beta_t}{2(\beta_t - 1)}(\wt{s} + s)\le 1 \text{ and  } 0 \le \frac{\beta_t}{2} (s - \wt{s})\le 1\Big\}.$$
	The weight function $f(s):= \int_{U_s} \Abs{\frac{\partial(u,v)}{\partial(s,\wt{s})}} \ud \wt{s} = \int_{U_s} \frac{\beta_t^2}{2(\beta_t - 1)}\ \ud \wt{s}$ is a probability distribution on $[0,1]$, which can be expressed explicitly as
	\[f(s) = \frac{\beta_t^2}{2(\beta_t-1)} \left(\min(s, \frac{2(\beta_t-1)}{\beta_t}-s) - \max(-s, s-\frac{2}{\beta_t}) \right),\qquad s\in [0,1].
	\]
	The graphs of $f(s)$ for $\beta_t =1.5$, $2$, and $3$ are given in \figref{fig::func_f_weight} for illustration. The largest value of $f(s)$ on the interval $[0,1]$ is $\beta_t$ if $\beta_t \le 2$, and is $\frac{\beta_t}{\beta_t-1}$ if $\beta_t \ge 2$. In either case, since $\beta_t > 1$, the largest value of $f(s)$ must be strictly larger than $1$. Therefore, we can define $s_1$ and $s_2$ as the two zeros of the function $f(s) - 1$, as visualized in the figure. \revise{More explicitly, $s_1 = (\beta_t - 1)/\beta_t^2$ and by symmetry, $s_2 = 1 - s_1$.}

	\begin{figure}[h!]
		\centering
		\begin{subfigure}[b]{0.48\textwidth}
			\includegraphics[width=\textwidth]{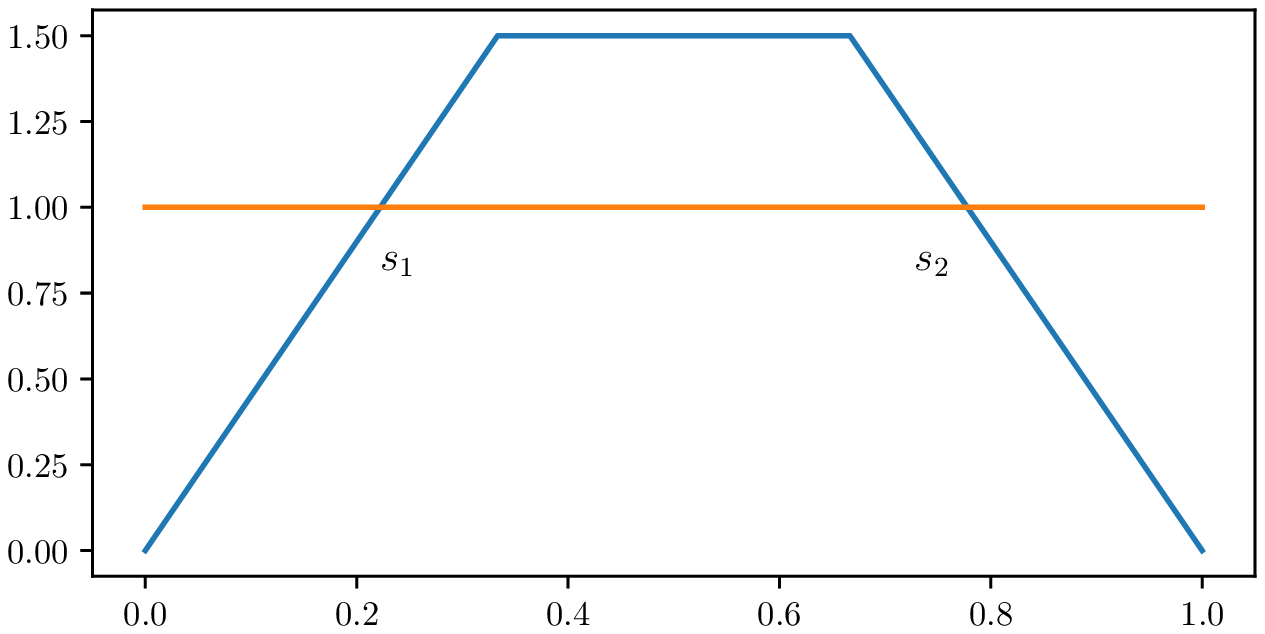}
			\caption{$\beta_t = 1.5$}
		\end{subfigure}
		~
		\begin{subfigure}[b]{0.48\textwidth}
		\includegraphics[width=\textwidth]{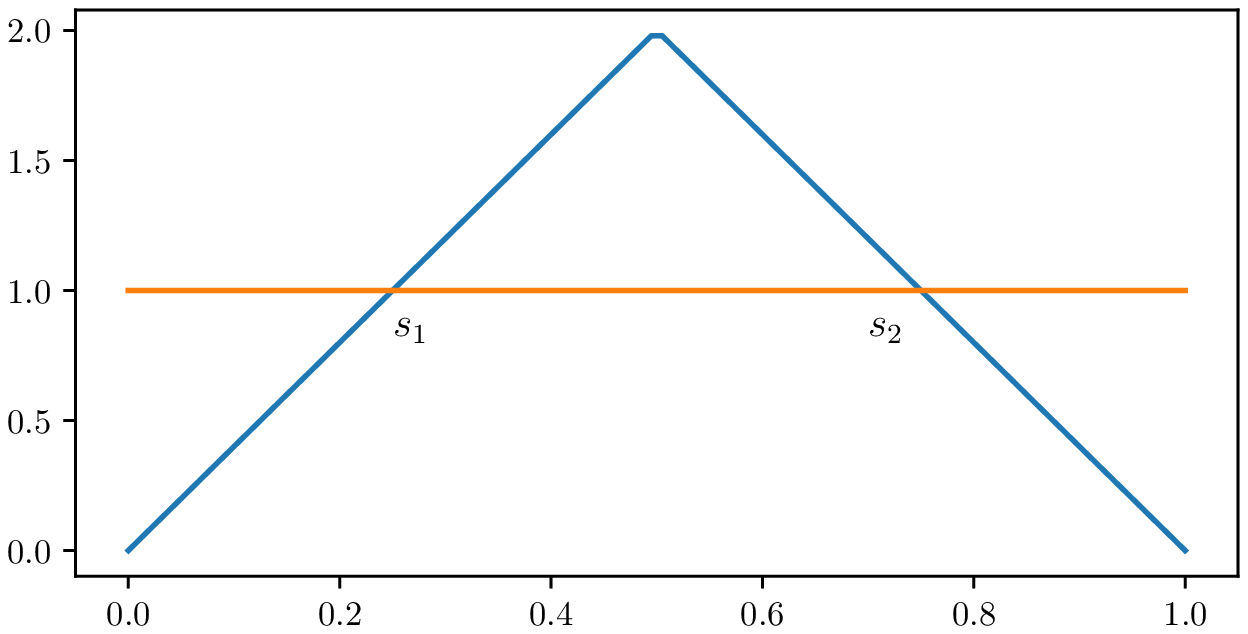}
		\caption{$\beta_t = 2$}
		\end{subfigure}
	\\
		\begin{subfigure}[b]{0.48\textwidth}
			\includegraphics[width=\textwidth]{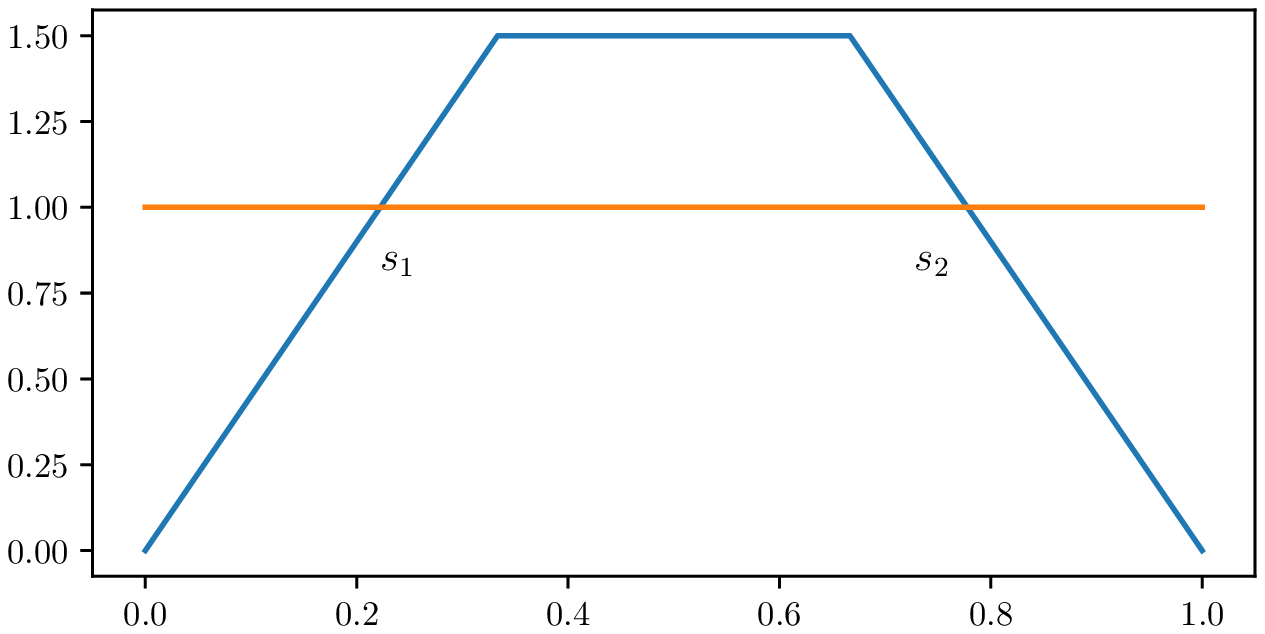}
			\caption{$\beta_t = 3$}
		\end{subfigure}
		\caption{Weight function $f(s)$ in blue color for $\beta_t=1.5$, $2$, $3$.}
		\label{fig::func_f_weight}
	\end{figure}

	To show \eqref{eqn::equiv_hyper_eta}, it is then sufficient to show that
	\begin{equation}
	\eta \int_{0}^{1} e^{\omega_j s} \rho_{\sigma,t}^{\beta_t s} (\cdot) \rho_{\sigma,t}^{-\beta_t s}\ud s \le \int_{0}^{1} e^{\omega_j s} \rho_{\sigma,t}^{\beta_t s} (\cdot) \rho_{\sigma,t}^{-\beta_t s} f(s)\ud s,
	\end{equation}
	in the sense of operators. In the above inequality, the notation $\rho_{\sigma,t}^{\beta_t s}(\cdot) \rho_{\sigma,t}^{-\beta_t s}$ is an operator mapping $A\in \mset$ to $\rho_{\sigma,t}^{\beta_t s} A \rho_{\sigma,t}^{-\beta_t s}\in \mset$.
	Notice that the function $f(s)$ is symmetric with respect to the axis $s = 1/2$ and recall the definition of $G_{X,\omega}(s)$ given by \eqref{eqn::g_symm_op}. We can rewrite the last inequality as
	\begin{equation}
	\label{eqn::opg_eta}
	\eta \int_{0}^{1/2} G_{\rho_{\sigma,t}^{\beta_t}, \omega_j}(s)\ud s \le \int_{0}^{1/2} G_{\rho_{\sigma,t}^{\beta_t}, \omega_j}(s) f(s)\ud s.
	\end{equation}
	By \lemref{lem::Gop_comparison} part (1), we have
	$$\int_{0}^{1/2} G_{\rho_{\sigma,t}^{\beta_t}, \omega_j}(s)\ud s  \ge \int_{0}^{1/2} G_{\rho_{\sigma,t}^{\beta_t}, \omega_j}(s) f(s)\ud s.$$ Apparently  we need $\eta \le 1$ for the  inequality \eqref{eqn::opg_eta} to hold. Thus from now on we shall restrict  to $\eta\le 1$.
	By subtracting the common parts in the integral \eqref{eqn::opg_eta}, in order to show \eqref{eqn::opg_eta}, it suffices to prove that
	\begin{equation}
	\label{eqn::gop_suff_cond}
	\eta \int_{0}^{s_1} G_{\rho_{\sigma,t}^{\beta_t}, \omega_j}(s) (1-f(s))\ud s \le \int_{s_1}^{\frac{1}{2}} G_{\rho_{\sigma,t}^{\beta_t}, \omega_j}(s) (f(s)-1)\ud s.
	\end{equation}
	 Note that $\int_{0}^{s_1} (1-f(s))\ud s = \int_{s_1}^{\frac{1}{2}} (f(s) - 1)\ud s$ since $f(s)$ is symmetric about the line $s=\frac{1}{2}$ and also $f(s)$ is a probability distribution on $[0,1]$. Then by \lemref{lem::Gop_comparison} part (1),
	\begin{align*}
	&\int_{s_1}^{\frac{1}{2}} G_{\rho_{\sigma,t}^{\beta_t}, \omega_j}(s) (f(s)-1)\ud s \ge \int_{s_1}^{\frac{1}{2}} G_{\rho_{\sigma,t}^{\beta_t}, \omega_j}\left(\frac{1}{2}\right) (f(s)-1)\ud s
	 = \int_{0}^{s_1} G_{\rho_{\sigma,t}^{\beta_t}, \omega_j}\left(\frac{1}{2}\right) (1-f(s))\ud s.
	\end{align*}
	 By \lemref{lem::Gop_comparison} part (2),
	 if we choose $\eta$ such that
	\begin{equation}
	\label{eqn::criterion_eta}
	\eta \le \frac{2\sqrt{e^{\omega_j}\frac{\specmin{\rho_{\sigma,t}^{\beta_t}}}{\specmax{\rho_{\sigma,t}^{\beta_t}}}}}{1+e^{\omega_j} \frac{\specmax{\rho_{\sigma,t}^{\beta_t}}}{\specmin{\rho_{\sigma,t}^{\beta,t}}}},
	\end{equation}
	then $\eta G_{\rho_{\sigma,t}^{\beta_t}, \omega_j}(s) \le G_{\rho_{\sigma,t}^{\beta_t}, \omega_j}\left(\frac{1}{2}\right)$ for all $0\le s\le \frac{1}{2}$, which implies \eqref{eqn::gop_suff_cond}.

	Hence, the remaining task is to show the existence of $\eta$ such that \eqref{eqn::criterion_eta} holds, for all $j\in \{1, 2, \cdots, \cardn\}$, all $t\in [0,T]$ and any initial condition $\rho_0$ satisfying the assumption that $\relaentropy{\rho_0}{\sigma} \le \epsilon$ for some $0 < \epsilon <  \frac{\specmin{\sigma}^2}{2}$;  step (\rom{3}) below is fully devoted into finding such an $\eta$.

	{\noindent {\bf Step (\rom{3}):} We show there exists $\eta\in (0,1]$ satisfying \eqref{eqn::criterion_eta}.}

	By observing the term on the right hand side of \eqref{eqn::criterion_eta}, it suffices to show that  $\frac{\specmax{\rho_{\sigma,t}^{\beta_t}}}{\specmin{\rho_{\sigma,t}^{\beta,t}}}$ is bounded from above.
	We first establish an upper bound for $\specmax{\rho_{\sigma,t}^{\beta_t}}$. In fact, from the quantum LSI \eqref{eqn::lsi_alpha} we have 
	\[\relaentropy{\rho_t}{\sigma} \le \relaentropy{\rho_0}{\sigma} e^{-2K t} \le \eps e^{-2K t} =: \eps_t.
	\]
	Moreover, by the quantum Pinsker's inequality \eqref{eqn::qpinsker},
	we obtain that
	\[
	\tr \Abs{\rho_t - \sigma} \le \sqrt{2\epsilon_t},
	\]
	which implies that
	\begin{equation*}
	-\sqrt{2\epsilon_t}\unit \le \rho_t - \sigma \le \sqrt{2\epsilon_t}\unit.
	\end{equation*}

	As a result,
	\begin{equation}
	\begin{aligned}\label{eq:rhotsigma}
	\sigma^{\frac{1-\beta_t}{2\beta_t}}\rho_t \sigma^{\frac{1-\beta_t}{2\beta_t}} &= \sigma^{\frac{1-\beta_t}{2\beta_t}} (\rho_t - \sigma) \sigma^{\frac{1-\beta_t}{2\beta_t}} + \sigma^{\frac{1}{\beta_t}} \\
	&\le \sigma^{\frac{1}{\beta_t}} + \sqrt{2\epsilon_t} \sigma^{\frac{1-\beta_t}{\beta_t}}\\
	& \le \left(\specmax{\sigma}^{\frac{1}{\beta_t}} + \sqrt{2\epsilon_t} \specmin{\sigma}^{\frac{1-\beta_t}{\beta_t}}\right)\unit.
	\end{aligned}
	\end{equation}
	It follows that
	\begin{align*}
	\lambda_{\max}(\rho_{\sigma,t}^{\beta_t}) &= \lambda_{\max}\left(\big(\sigma^{\frac{1-\beta_t}{2\beta_t}}\rho_t \sigma^{\frac{1-\beta_t}{2\beta_t}}\big)^{\beta_t}\right)
	\le \left(\specmax{\sigma}^{\frac{1}{\beta_t}} + \sqrt{2\epsilon_t} \specmin{\sigma}^{\frac{1-\beta_t}{\beta_t}}\right)^{\beta_t}  \\
	&= \specmax{\sigma} \left(1 + \frac{\sqrt{2\epsilon_t}}{\specmin{\sigma}}\left(\frac{\specmin{\sigma}}{\specmax{\sigma}}\right)^{\frac{1}{\beta_t}}\right)^{\beta_t}  \\
	&\le \specmax{\sigma} \exp\left(\beta_t \frac{\sqrt{2\epsilon_t}}{\specmin{\sigma}}\left(\frac{\specmin{\sigma}}{\specmax{\sigma}}\right)^{\frac{1}{\beta_t}}\right) .
	\end{align*}
	In the last inequality we used an elementary inequality $(1+x)^y = \exp(y\log(1+x))\le \exp(xy)$ for $x,y\ge 0$. By the definition of $\beta_t$,  the exponent on the right side of the above equation can be bounded above by
	\begin{align*}
	&\beta_t \frac{\sqrt{2\epsilon_t}}{\specmin{\sigma}}\left(\frac{\specmin{\sigma}}{\specmax{\sigma}}\right)^{\frac{1}{\beta_t}}\\
	=& (1 + (\alpha_0 - 1) e^{\eta 2Kt}) \frac{\sqrt{2\eps}}{\specmin{\sigma}}\exp\left(-K t -\frac{1}{\beta_t} \log \left( \frac{\specmax{\sigma}}{\specmin{\sigma}}\right)\right) \\
	 \le & \frac{\sqrt{2\eps}}{\specmin{\sigma}} + (\alpha_0 -1)\frac{\sqrt{2\eps}}{\specmin{\sigma}}\exp\left(\eta 2K t -K t -\frac{1}{\beta_t} \log \left( \frac{\specmax{\sigma}}{\specmin{\sigma}}\right)\right).
	\end{align*}
	If we restrict $\eta \le \frac{1}{2}$, then the quantity above is less than $\alpha_0 \frac{\sqrt{2\eps}}{\specmin{\sigma}}$ and thus
	\begin{equation}
	\label{eqn::spec_rho_sigma_beta_upper}
	\lambda_{\max}(\rho_{\sigma,t}^{\beta_t} ) \le \specmax{\sigma} \exp\left(\alpha_0 \frac{\sqrt{2\eps}}{\specmin{\sigma}}\right).
	\end{equation}

	Next, we prove a lower bound for $\specmin{\rho_{\sigma,t}^{\beta,t}}$. In fact, similar to \eqref{eq:rhotsigma}, one can show that
	\begin{align*}
	\sigma^{\frac{1-\beta_t}{2\beta_t}}\rho_t \sigma^{\frac{1-\beta_t}{2\beta_t}}
	&= \sigma^{\frac{1-\beta_t}{2\beta_t}} (\rho_t - \sigma) \sigma^{\frac{1-\beta_t}{2\beta_t}} + \sigma^{\frac{1}{\beta_t}}
	 \ge \sigma^{\frac{1}{\beta_t}} - \sqrt{2\epsilon_t} \sigma^{\frac{1-\beta_t}{\beta_t}}\\
	& \ge \left(\specmin{\sigma}^{\frac{1}{\beta_t}} - \sqrt{2\epsilon_t} \specmin{\sigma}^{\frac{1-\beta_t}{\beta_t}}\right)\unit
	= \specmin{\sigma}^{\frac{1}{\beta_t}} \left(1 - \frac{\sqrt{2\epsilon_t}}{\specmin{\sigma}}\right) \unit > 0,
	\end{align*}
	and hence
	\begin{align*}
	\lambda_{\min}(\rho_{\sigma,t}^{\beta_t}) &= \lambda_{\min}\left(\big(\sigma^{\frac{1-\beta_t}{2\beta_t}}\rho_t \sigma^{\frac{1-\beta_t}{2\beta_t}}\big)^{\beta_t} \right)
	\ge  \specmin{\sigma} \left(1 - \frac{\sqrt{2\epsilon_t}}{\specmin{\sigma}}\right)^{\beta_t}  \\
	& = \specmin{\sigma} \exp\left(\beta_t \log\left(1-\frac{\sqrt{2\eps}}{\specmin{\sigma}} e^{-K t}\right)\right).
	\end{align*}
	We show that the exponent on the right side of the above equation can be bounded from below. Indeed, since $1-\frac{\sqrt{2\eps}}{\specmin{\sigma}} e^{-K t} \ge 1 - \frac{\sqrt{2\eps}}{\specmin{\sigma}} > 0$ and  $\log(1-x)\ge -x/(1-x)$ for $x\in [0,1)$, if we restrict $\eta \le 1/2$, then
	\begin{align*}
	0 &> \beta_t \log\left(1-\frac{\sqrt{2\eps}}{\specmin{\sigma}} e^{-K t}\right)
	= (1 + (\alpha_0 - 1) e^{\eta 2Kt}) \log\left(1-\frac{\sqrt{2\eps}}{\specmin{\sigma}} e^{-K t}\right)\\
	&\ge \log\left(1 - \frac{\sqrt{2\eps}}{\specmin{\sigma}}\right) - (\alpha_0 - 1) e^{\eta 2K t} \frac{\frac{\sqrt{2\eps}}{\specmin{\sigma}} e^{-K t}}{1-\frac{\sqrt{2\eps}}{\specmin{\sigma}} e^{-K t}} \\
	&= \log\left(1 - \frac{\sqrt{2\eps}}{\specmin{\sigma}}\right) - (\alpha_0 - 1) e^{\eta 2K t-K t} \frac{\frac{\sqrt{2\eps}}{\specmin{\sigma}} }{1-\frac{\sqrt{2\eps}}{\specmin{\sigma}}} \\
	&\ge \log\left(1 - \frac{\sqrt{2\eps}}{\specmin{\sigma}}\right) - (\alpha_0 - 1) \frac{\frac{\sqrt{2\eps}}{\specmin{\sigma}} }{1-\frac{\sqrt{2\eps}}{\specmin{\sigma}}}
	\ge - \alpha_0 \frac{\frac{\sqrt{2\eps}}{\specmin{\sigma}} }{1-\frac{\sqrt{2\eps}}{\specmin{\sigma}}}.
	\end{align*}
	This implies that
	\begin{align}
	\label{eqn::spec_rho_sigma_beta_lower}
	\lambda_{\min}(\rho_{\sigma,t}^{\beta_t}) \ge \specmin{\sigma} \exp\left(- \alpha_0 \frac{\frac{\sqrt{2\eps}}{\specmin{\sigma}} }{1-\frac{\sqrt{2\eps}}{\specmin{\sigma}}}\right).
	\end{align}
\revise{Combining the estimations \eqref{eqn::spec_rho_sigma_beta_upper} and \eqref{eqn::spec_rho_sigma_beta_lower} yields}
	\begin{align*}
	\frac{\specmax{\rho_{\sigma,t}^{\beta_t}}}{\specmin{\rho_{\sigma,t}^{\beta_t}}} &\le
	\frac{\specmax{\sigma} \exp\left(\alpha_0 \frac{\sqrt{2\eps}}{\specmin{\sigma}}\right)}{\specmin{\sigma} \exp\left(- \alpha_0 \frac{\frac{\sqrt{2\eps}}{\specmin{\sigma}} }{1-\frac{\sqrt{2\eps}}{\specmin{\sigma}}}\right)} \\
	&= \frac{\specmax{\sigma}}{\specmin{\sigma}} \exp\left(\alpha_0 \sqrt{2\eps}\frac{2\specmin{\sigma} - \sqrt{2\eps}}{\specmin{\sigma}(\specmin{\sigma} - \sqrt{2\eps})}\right)
	=: \Lambda.
	\end{align*}
	This leads to
	\begin{align*}
	\frac{2\sqrt{e^{\omega_j}\frac{\specmin{\rho_{\sigma,t}^{\beta_t}}}{\specmax{\rho_{\sigma,t}^{\beta_t}}}}}{1+e^{\omega_j} \frac{\specmax{\rho_{\sigma,t}^{\beta_t}}}{\specmin{\rho_{\sigma,t}^{\beta_t}}}} \ge \frac{2\sqrt{e^{\omega_j} \frac{1}{\Lambda}}}{1+e^{\omega_j} \Lambda}.
	\end{align*}
	This finishes the proof of   \eqref{eqn::criterion_eta} with
	\begin{equation}
	\label{eqn::choice_eta}
	\eta := \min\left(\frac{1}{2}, \min_{j=1}^{\cardn} \left\{\frac{2\sqrt{e^{\omega_j} \frac{1}{\Lambda}}}{1+e^{\omega_j} \Lambda}\right\}\right) > 0.
	\end{equation}

\end{proof}

\subsection{Discussion on quantum comparison theorem}
\label{sec::diss_q_comparison}
Finally, we would like to comment on the quantum comparison theorem (see \propref{prop::comparison}).

\emph{Discussion on its proof and the hypercontractivity:}
The main idea of proving \propref{prop::comparison} is to verify the
\emph{hypercontractivity} in the sense of noncommutative
$\lpsp$ space (see \eg, \cite[Definition
12]{kastoryano_quantum_2013}) under the assumption of the quantum LSI.
More specifically, suppose
$\beta_t = 1 + (\alpha_0 - 1) e^{\eta 2K t}$ for some constant
$\eta > 0$ and $\rho_t$ follows the \lb{} equation, \ie,
$\dot{\rho}_t = \lbop^{\dagger}(\rho_t)$, then for the time interval
$[0,T]$, where
$T =\frac{1}{2K \eta} \log\left(\frac{\alpha_1-1}{\alpha_0 -
    1}\right)$, we need to show that
	\begin{equation}
	\label{eqn::hypercontractivity_primitive}
	\left[\tr\left(\left(\Gamma_{\sigma}^{\frac{1-\beta_t}{\beta_t}}(\rho_t)\right)^{\beta_t}\right)\right]^{\frac{1}{\beta_t}} \text{ is a non-increasing function of time } t\in [0,T].
	\end{equation}
	The function $F_t$ in the last subsection is simply the logarithm of this function.
	To see why this requirement \eqref{eqn::hypercontractivity_primitive} is closely related to the hypercontractivity in the noncommutative $\lpsp$ spaces, let us introduce the relative density $X_t = \Gamma_{\sigma}^{-1}(\rho_t)$; the time-evolution of the relative density $X_t$ has the generator $\lbop$, \ie, $X_t = e^{t\lbop} (X_0)$, which can be directly verified with the help of \eqref{eqn::lbop_lbop_dag}.
	With these notations and definitions, \eqref{eqn::hypercontractivity_primitive} means $\Norm{e^{t\lbop}(X_0)}_{\beta_t, \sigma} \le \Norm{X_0}_{\alpha_0, \sigma}$, which has the same form as \cite[Definition 12]{kastoryano_quantum_2013}. In a recent paper \cite{beigi_quantum_2018}, for primitive \lb{} equations with GNS-detailed balance, a quantum Stroock-Varopoulos inequality (see \cite[Theorem 14]{beigi_quantum_2018}) has been proved, whose implication (see \cite[Corollary 17]{beigi_quantum_2018}) shows that one can pick $\eta_{SVI} = 2\kappa_2/K$ (see \eqref{eqn::lsic_Lp} for the definition of $\kappa_2$); the subscript \emph{SVI} is used to emphasize using the quantum \svi{}. Recall that $K = 2\kappa_1$ \eqref{eqn::kappa_1_K};
	the quantum \svi{} shows that $\kappa_1 \ge \kappa_2$ (see also \thmref{thm::qsvi}), hence $\eta_{SVI}  \le 1$, which is consistent with our observation in the proof that $\eta\le 1$. We make some comments on the difference between \propref{prop::comparison} and the above estimation using the quantum \svi{} in \cite{beigi_quantum_2018}:
	\begin{enumerate}
    \item Using the estimation from the quantum \svi{}, the prefactor $\frac{1}{2K\eta}$ in \eqref{eqn::delay_T}, is simply $\frac{1}{4\kappa_2}$; in this paper, we directly use
      $K/2 \equiv \kappa_1$ and the information about the spectrum of $\sigma$.

	\item There is, however, a more stringent assumption on the initial condition $\rho_0$ in \propref{prop::comparison}. Because of such a restriction on initial conditions, even though our proof is related to verify the hypercontractivity, strictly speaking, we haven't really proved the hypercontractivity by solely assuming the quantum LSI \eqref{eqn::lsi_alpha} for $\alpha = 1$.

		\item Despite the restriction on initial conditions $\rho_0$,  our proof is probably more similar to the classical result for the \fp{} equation case in \revise{\cite[Lemma 3.4]{prevpaper}}. Moreover, our proof for the quantum comparison theorem does not require any prior knowledge about noncommutative $\lpsp$ spaces.
	\end{enumerate}

	\emph{Comparison with the classical result:}
	By comparing  \propref{prop::comparison} and a similar result for the Fokker-Planck equation (c.f. \revise{\cite[Lemma 3.4]{prevpaper}}), it appears that the  waiting time $T$ in \propref{prop::comparison}  is much larger than that for the \fp{} equation. More specifically, the parameter $\eta \le 1$ for the \lb{} equation, and $\eta = 1$ for the \fp{} equation.

\section{Conclusion and discussions}
\label{sec::outlook}

In this paper, we have extended the gradient flow structure of the quantum relative entropy in \cite{carlen_gradient_2017} to sandwiched \renyi{} divergences of any order $\alpha\in (0,\infty)$, for primitive \lb{} equations with GNS-detailed balance.
\revise{The necessary condition for the validity of such a gradient flow structure has been briefly discussed.}
Furthermore, we have proved the exponential decay of the sandwiched \renyi{} divergence of any order $\alpha\in (0,\infty)$.
We conclude with some remarks on open questions and future research directions.

\begin{enumerate}

\item \revise{(\emph{More general entropy}).} In this paper, we have considered the entropy production of the sandwiched \renyi{} divergence, which is a generalization of the quantum relative entropy. One natural question is whether the gradient flow structure and exponential decay proved in this paper for the sandwiched \renyi{} divergence can be extended to the more general $(\alpha,z)$-\renyi{} divergence \cite{audenaert_alpha-z-renyi_2015,carlen_inequalities_2018}, for certain ranges of $(\alpha,z)$.

\item \revise{(\emph{Gap between sufficient and necessary conditions for having a gradient flow structure}).
We have mentioned in \secref{sec::necessary_condition} that there is still a gap between sufficient and necessary conditions to regard Lindblad equations as the gradient flow dynamics of sandwiched \renyi{} divergences (including the quantum relative entropy \cite{carlen_non-commutative_2018}).
Characterizing the QMS with various detailed balance conditions, beyond GNS-detailed balance, might be important to resolve this issue, \ie, studying different versions of  \thmref{thm::lindblad_detail_balance} by considering various detailed balance assumptions.}

\item \revise{(\emph{Quantum Wasserstein distance}).}
In the classical optimal transport theory,
Wasserstein distance, introduced to capture the cost of transportation between two probability measures,
has been widely studied \cite{villani2009optimal}. Due to its success for classical systems, it is a natural question to explore the notion \emph{quantum Wasserstein distance}.
There are many attempts to define this concept, via the Monge formalism \cite{zyczkowski_monge_1998, zyczkowski_monge_2001}, the Kantorovich formalism \cite{golse_mean_2016,agredo_quantum_2017,golse_wave_2018} and the \bbren{} formalism \cite{carlen_gradient_2017,carlen_analog_2014,wirth_noncommutative_2018,hornshaw_L2-wasserstein_2018, chen_matrix_2017_mk,chen_matrix_2018_optimal,tempo_wasserstein_2018}.
The one that is most relevant to this paper is \bbren{} formalism \cite{benamou_computational_2000}, which offers a dynamical picture to view the Wasserstein distance between any two states $\rho_0$ and $\rho_1$ as the minimal Lagrangian along all possible paths connecting $\rho_0$ and $\rho_1$. This Lagrangian could possibly be defined via metric tensors in \rie{} manifolds.
In \cite[Sec. 8]{carlen_gradient_2017}, with the variational formalism of primitive \lb{} equations with GNS-detailed balance for the quantum relative entropy, quantum Wasserstein distance is defined via \bbren{} formalism.
With the generalization of this variational formalism to the case of the \swch{}, one could straightforwardly generalize the quantum Wasserstein distance defined in \cite{carlen_gradient_2017} to a quantum $(\alpha,q)$-Wasserstein distance.
\begin{defn}[Quantum $(\alpha,q)$-Wasserstein distance]
	For any order $\alpha\in (0,\infty)$ and power $q \ge 1$, a quantum $(\alpha,q)$-Wasserstein distance is defined in the following way: for any $\rho_0, \rho_1\in \dset$,
	\begin{equation*}
	W_{\alpha,q}(\rho_0,\rho_1) := \inf_{\gamma(s)}\ \left(\int_{0}^{1} \bigl(g_{\alpha,\gamma(s), \lbop}\left(\dot{\gamma}(s), \dot{\gamma}(s)\right)\bigr)^{\frac{q}{2}}\ud s \right)^{\frac{1}{q}},
	\end{equation*}
	where the infimum is taken over \revise{all smooth paths} $\gamma(\cdot): [0,1]\rightarrow \dsetplus$  that connect $\rho_0$ and $\rho_1$; more specifically, $\gamma(0) = \rho_0$, $\gamma(1) = \rho_1$ and $\rho(s) \in \dsetplus$ for all $s\in (0,1)$.
\end{defn}
The study of the properties of this quantum Wasserstein distance is also an interesting topic for future works.

\item \revise{(\emph{Lindblad equation with energy-conservation term}). In general, Lindblad equation has both energy-conservation term and dissipative term, \ie, it has the form 
\begin{equation}
\dot{\rho}_t = \lbop^{\dagger}(\rho_t) = -i \Comm{H}{\rho_t} + \sum_{j} c_j \left( 
\Comm{V_j \rho_t}{V_j^{*}} + \Comm{V_j}{\rho_t V_j^*}
\right),
\end{equation} 
where $c_j$ are non-negative constants.
The Hamiltonian $H$ in general does not vanish \cite{lindblad_generators_1976,gorini_completely_1976}. 
Therefore, it is also an interesting question to explore how we could generalize results in the present paper to such general Lindblad equations with non-trivial $H$. 

An exact construction of gradient flow structure in a Riemannian manifold (see \secref{sec::preliminary_gradient}) does not seem to be possible, even for classical kinetic Fokker-Planck equations. 
However, for a generalized Kramers equation (probably regarded as the classical analog of Lindblad equations with a nontrivial Hamiltonian term), some JKO schemes have already existed in \cite{duong_conservative-dissipative_2014}.
As far as we know, currently, there is no any type of JKO scheme for Lindblad equations, which is probably also an interesting question to study, due to the current active research on quantum Wasserstein distance, as we mentioned above.  

As for convergence rate to the equilibrium, there is another approach known as \emph{hypocoercivity} (see \eg{} \cite{villani2009hypocoercivity}), for kinetic Fokker-Planck equations. We are curious about whether this approach admits a quantum extension to help study the convergence of Lindblad equations with non-trivial Hamiltonian term $H$; however, this is far beyond the scope of this paper. 

}

\end{enumerate}

\bigskip
\noindent\textbf{Acknowledgment.}
This work is partially supported by the National Science Foundation
under grant DMS-1454939. We thank Iman Marvian and Henry Pfister for
helpful discussions.

\bibliography{reference}

\end{document}